\documentclass[12pt]{article}
\usepackage{latexsym,amsmath,amssymb}
\usepackage{amsthm,enumerate}
\usepackage{graphicx}
\usepackage{xcolor}
\usepackage[colorlinks]{hyperref}
\usepackage[utf8]{inputenc}

\usepackage[T1]{fontenc}
\usepackage{lmodern}
\definecolor{mapblue}{RGB}{25,65,180}
\definecolor{mapred}{RGB}{170,25,30}
\hypersetup{colorlinks=true,linkcolor=red,citecolor=green,urlcolor=magenta}

\newfont{\bb}{msbm10 at 12pt}

\def\R{\mathbb{R}}

\def\H{\mathbb{H}}

\def\G{\Gamma}

\def\S{\Sigma}

\def\esf{\mathbb{S}}

\newcommand{\meta}[2]{\langle #1,#2 \rangle }

\newcommand{\ee}{\begin{equation*}}
\newcommand{\fe}{\end{equation*}}

\newcommand{\set}[1]{\left\{#1\right\}}

\newcommand{\bp}{\mathbf{p}}

\newtheorem{theorem}{Theorem}[section]
\newtheorem{lemma}[theorem]{Lemma}

\newtheorem{remark}[theorem]{Remark}

\newtheorem{definition}[theorem]{Definition}

\newtheorem{example}[theorem]{Example}

\begin{document}

\title{Frobenius conjugation and Jacobi functions\\
	for minimal immersions in $\mathbb S^3$\\
	and CMC~$1$ immersions in $\mathbb R^3$}
\author{Jos\'e M. Espinar \and Joaqu\'in P\'erez}
\maketitle
\begin{abstract}
We give a first-order Frobenius description of conjugate minimal
immersions in the round three-sphere and use it to define the conjugate
of an arbitrary Jacobi function. The construction is intrinsic at the
linearized level and does not require the Jacobi function to be
integrable through a normal deformation by minimal surfaces. We then describe how the boundary conditions are exchanged
under conjugation.  Along a great-circle boundary arc, a Dirichlet
Jacobi function has, after a suitable ambient gauge, a conjugate Jacobi function with
vanishing Neumann data.  Conversely, along a reflection curve, vanishing
Neumann data yield vanishing Dirichlet data for the conjugate modulo a
Jacobi--Killing function. We also construct conjugate Jacobi functions for conjugate CMC~$1$ immersions in $\R^3$ and exhibit cylindrical counterexamples to both pure boundary exchanges, even modulo ambient Killing fields.
\begin{center}
	\emph{Dedicated to the memory of William H. Meeks III}
\end{center}
\end{abstract}

\section{Introduction}
The conjugate-surface construction couples pairs of simply connected minimal or constant-mean-curvature (CMC) immersions in space forms. We consider the following four classes of pairs $(f,f^*)$ of immersions: 
\begin{enumerate}
	\item[(a)] Conjugate minimal immersions in $\mathbb{R}^3$;
	\item[(b)] Cousin pairs $(f,f^*)$ where $f\colon \Sigma\to \mathbb{R}^3$ is a CMC 1-immersion and $f^*\colon \Sigma\to \mathbb{S}^3$ is its conjugate cousin minimal immersion; 
	\item[(c)] Conjugate minimal immersions in $\mathbb{S}^3$;
	\item[(d)] Conjugate CMC 1-immersions in $\mathbb{R}^3$.
\end{enumerate} 
The purpose of this paper is twofold: first, to formulate explicitly
the conjugation of immersions in cases (c) and (d) \emph{à la
	Frobenius}, i.e., in terms of a coupled system of first-order PDEs;
second, to linearize these systems and construct conjugate Jacobi functions without assuming integrability through families of minimal or CMC immersions. In case~(c), we prove the exchange between Dirichlet and Neumann data along a single boundary arc, modulo ambient Killing fields. For case~(d), Theorem~\ref{thm:CMC-Jacobi-conjugation} gives the corresponding interior construction, while Example~\ref{ex:CMC-cylinder-boundary} shows that both pure boundary exchanges fail in general.

The first-order formulation of conjugate and cousin surface
correspondences has a substantial history.  Lawson's correspondence
between CMC surfaces in space forms was originally formulated in terms of the fundamental forms and their integrability equations~\cite{la3}.  A particular case of the Lawson correspondence is the so-called {\em conjugate cousin construction} between CMC $H=1$ surfaces in $\mathbb{R}^3$ and minimal surfaces in $\mathbb{S}^3$. A first-order, quaternionic version of the conjugate cousin construction was initiated by Karcher~\cite{ka5}, developed further by Gro{\ss}e-Brauckmann~\cite{gb1}, and used
systematically by Gro{\ss}e-Brauckmann, Kusner, and
Sullivan~\cite{gbks,gbks1} in the construction and
classification of triunduloids and coplanar $k$-unduloids in $\mathbb{R}^3$. In the classical setting of minimal surfaces in $\mathbb R^3$, conjugate Jacobi functions have been used in several contexts. In particular, Cos{\'i}n and Ros applied this construction to
bounded Jacobi functions on finite-total-curvature minimal surfaces
in their study of coplanar minimal $k$-noids
in $\mathbb R^3$~\cite{coros}. For CMC surfaces in $\mathbb{R}^3$, Korevaar, Kusner, and Ratzkin introduced a conjugate Jacobi field construction by linearizing the conjugate cousin system; in particular, their construction exchanges Neumann and Dirichlet boundary behavior and yields nondegeneracy results for triunduloids~\cite{kkr1}.  This method was subsequently developed by Gro{\ss}e-Brauckmann, Korevaar, Kusner, Ratzkin, and Sullivan to prove that all coplanar $k$-unduloids are nondegenerate and to analyze the differential of their classifying map~\cite{GKKRS2009}. These nondegeneracy results are closely related to the analytic structure of CMC moduli spaces developed by Kusner, Mazzeo, and Pollack~\cite{kmp1}.

As mentioned in the last paragraph, the formulation {\em à la Frobenius} is well-known for cases (a) and (b); interestingly, in cases (c) and (d) a new ingredient appears in the first order system, namely a frame map $R\colon \Sigma\to SO(3)$ that relates the frames of $f,f^*$ and that in cases (a) and (b) reduces to a constant map, which can be normalized to be the identity $I\in SO(3)$. Another feature that requires some care is that since the conjugate immersion is determined only up to an ambient isometry, the conjugate Jacobi function is defined only modulo Jacobi--Killing functions; this feature produces a necessary adjustment when deducing the exchange between the Dirichlet and Neumann conditions for conjugate Jacobi functions in case (c).

{\em Organization of the paper.} 
After recalling the quaternionic model in
Section~\ref{SectPreliminaries} and the conjugate correspondence
\emph{à la Frobenius} for cases (a) and (b) in
Section~\ref{secconj}, we derive Frobenius systems for cases (d)
and (c), respectively, in Theorems~\ref{thm8.3}
and~\ref{thm8.4}.
Starting in Section~\ref{secJacconj} we focus on case (c): first we linearize the Frobenius system; a direct curvature computation shows that its integrability
condition is precisely the Jacobi equation. Section~\ref{sec:integrable-gauge} applies
the linearized system of case~(c) to boundary arcs and proves the Dirichlet--Neumann
and Neumann--Dirichlet rules, with the necessary gauge freedom stated
explicitly. Section~\ref{sec:associated-family} shows that the Frobenius system and its linearization extend naturally from the conjugate immersion/Jacobi function to the whole associated family. In particular, every Jacobi function on $f$ determines, for each associated immersion $f_\theta$, a Jacobi function $u_\theta$, uniquely defined modulo Jacobi--Killing functions. The conjugate construction studied in detail in Sections~\ref{secJacconj} and~\ref{sec:integrable-gauge} is recovered at $\theta=\pi/2$. The exchange of Dirichlet and Neumann data, however, is a special feature of the conjugate pair and does not extend to a general associated immersion. Section~\ref{sec:CMC-Jacobi-conjugation} then returns to case~(d): it constructs the CMC~$1$ conjugate of an arbitrary Jacobi function, uniquely modulo Euclidean Jacobi--Killing functions, and derives a boundary identity that retains the tangential and conormal derivatives while exhibiting extra terms. An explicit cylindrical example proves the failure of both pure boundary exchanges at the level of the immersions and of the Jacobi functions.

This work belongs to the circle of ideas surrounding minimal and
CMC surfaces, Jacobi fields, and deformation
theory to which William H. Meeks III made fundamental contributions.
It is a privilege to dedicate this paper to his memory.
\par\vspace{.2cm}
{\em Use of AI.} During the preparation of this paper, the authors used Anthropic's Claude and OpenAI's ChatGPT Sol as a research-assistance tool in developing the reformulation that removes the integrability hypothesis from the proof of Theorem~\ref{propos8.9} and from the Dirichlet--Neumann correspondence, as well as in extending the main
results from conjugate pairs to the full associated family. All
AI-assisted suggestions were independently checked, revised, and
integrated by the authors, who take full responsibility for every
mathematical statement, argument, and conclusion in the paper.

\section{Preliminaries}\label{SectPreliminaries}

We will start by settling the notation and basic properties used throughout the paper. 

\subsection{Quaternions}

We will use the standard identification of $\R^4$ with the quaternions:
\[
\H =\R {\bf e}\oplus \mbox{Span}\{ {\bf i},{\bf j},{\bf k}\}=
	\R {\bf e}\oplus \Im(\H),
\]
where ${\bf e}=(1,0,0,0), \ {\bf i}=(0,1,0,0), \ {\bf j}=(0,0,1,0), \ {\bf k}=(0,0,0,1)$.
We will identify $\Im(\H)$ with $\R^3$ in the usual manner, i.e., ${\bf i},{\bf j}, {\bf k}$ identify with the canonical basis of $\R^3$ by removing the real part of quaternions. For ${\bf p},{\bf q}\in \H$, we have as usual
\begin{enumerate}[(Q1)]
\item Decomposition  into real and imaginary parts: ${\bf p}=\Re({\bf p}){\bf e}+\Im({\bf p})$. 
\item Product of quaternions: 
$${\bf p}\cdot {\bf q}=\left( \Re({\bf p})\Re({\bf q})-\langle \Im({\bf p}),\Im({\bf q}) \rangle \right) {\bf e}
+\left[ \Re({\bf p})\Im({\bf q})+\Re({\bf q})\Im({\bf p})+\Im({\bf p})\times \Im({\bf q})\right] .$$
(hence $\langle {\bf p},{\bf q}\rangle =-\Re({\bf p}\cdot {\bf q})$ $\forall {\bf p},{\bf q}\in \Im(\H)$)
\item Conjugate of a quaternion: $\overline{{\bf p}}=\Re({\bf p}){\bf e}-\Im ({\bf p})$, $\forall {\bf p}\in \H$. 
\item Hermitian product: $({\bf p},{\bf q})_{\H}=\overline{\bf p}\cdot {\bf q}$.
\item Norm: $\| {\bf p} \|=\sqrt{\overline{\bf p}\cdot {\bf p}}$ (hence $\| {\bf p}\cdot {\bf q}\|=\| {\bf p}\| \| {\bf q}\|$).
\item ${\bf a}\cdot {\bf b}-{\bf b}\cdot {\bf a}=2\, {\bf a}\times {\bf b}$
for all ${\bf a},{\bf b}\in \Im(\H)$.
\item $\langle {\bf p}\cdot {\bf b},{\bf q}\rangle 
+\langle {\bf p},{\bf q}\cdot {\bf b}\rangle=0$, for all ${\bf p},{\bf q}\in \H$,
${\bf b}\in \Im(\H)$.
\end{enumerate}

\subsection{\texorpdfstring{The standard model of $\esf^3$ in $\H$}{The standard model of S3 in H}}\label{sec2.1a}

We will mostly use the standard model of the three-dimensional sphere $\esf^3$ inside $\H$ as the unitary quaternions, $\esf ^3 = \set{ \bp = (p_1 ,p_2 ,p_3 ,p_4) \in \H \, | \, \, \meta{\bp}{\bp} =1 }$,
endowed with its round metric of constant sectional curvature $1$ given by the restriction of the standard Euclidean product $\langle \cdot ,\cdot \rangle$ of $\R^4\equiv \H$ to $\esf^3$. Therefore, $(\esf^3,\cdot )$ is a Lie group with identity element ${\bf e}$ and $T_{\bf e}\esf^3=\Im(\H)$ identifies with the Lie algebra of $\esf^3$.

The identities
\[
\langle{\bf p}\cdot {\bf q},{\bf p}\cdot {\bf r}\rangle =\| {\bf p}\|^2\langle {\bf q},{\bf r}\rangle=\langle{\bf q}\cdot {\bf p},{\bf r}\cdot {\bf p}\rangle 
\qquad \forall {\bf p},{\bf q},{\bf r}\in \H,
\]
tell us that given ${\bf p}\in \esf^3$, the left and right translations
\[
\begin{array}{cl}
l_{\bf p},r_{\bf p}\colon \H\to \H, & l_{\bf p}({\bf q})={\bf p}\cdot {\bf q}, \quad r_{\bf p}({\bf q})={\bf q}\cdot {\bf p},
\end{array}
\]
are isometries of $(\H,\langle \cdot ,\cdot \rangle)$.

Since the left multiplication by elements in $\esf^3$ is an isometry of the inner product in $\R^4$,
given ${\bf p}\in \esf^3$ we can induce a cross product $\times_{\bf p}$ in 
$T_{\bf p}\esf^3=\langle {\bf p}\rangle ^{\perp}$ 
by left translating by ${\bf p}$ the usual cross product 
$\times=\times_{\bf e}$ in $\R^3=
\Im(\H)$, i.e.,
\[
({\bf p}^{-1}\cdot {\bf w}_1)\times ({\bf p}^{-1}\cdot {\bf w}_2)=
{\bf p}^{-1}\cdot({\bf w}_1\times_{\bf p}{\bf w}_2),
\qquad \forall {\bf w}_1,{\bf w}_2\in T_{\bf p}\esf^3.
\]
The same holds with the volume element $\det_{\bf p}$ in
$T_{\bf p}\esf^3$, given by 
\[
\textstyle{\det_{\bf p}}({\bf w}_1,{\bf w_2},{\bf w_3})
=\det_{\Im(\H)}({\bf p}^{-1}\cdot {\bf w}_1,{\bf p}^{-1}\cdot {\bf w}_2,{\bf p}^{-1}\cdot {\bf w}_3)=\det({\bf p},{\bf w}_1,{\bf w}_2,{\bf w}_3),
\]
for all ${\bf w}_1,{\bf w}_2,{\bf w}_3\in T_{\bf p}\esf^3$ (the last 
determinant is the usual one in $\R^4$).

\section{\texorpdfstring{The conjugation of minimal surfaces in $\R^3$ and the cousin correspondence}{The conjugation of minimal surfaces in R3 and the cousin correspondence}}
\label{secconj}
Throughout this section, $(\Sigma,ds^2)$ will denote a simply-connected Riemannian surface with area element $da$. It is a classical fact that an isometric minimal immersion $f\colon (\Sigma,ds^2) \to \R^3$ and its conjugate minimal immersion $f^*
\colon \Sigma \to \R^3$ are related (up to rigid motions)
by the first order PDE
\begin{equation}\label{123}
df^*=df\circ J,
\end{equation}
where $J$ is the $ds^2$-rotation in the tangent bundle $T\Sigma$
by angle $\pi/2$. The integrability condition of~\eqref{123} 
in the unknown $f^*$ is that the $\R^3$-valued one-form $(p,v)\in T\Sigma \mapsto df_p(Jv)$ is closed (hence exact). Regardless of whether $f$
is minimal or not, we have
\begin{equation}\label{124}
d(df\circ J)=-2H\, N\, da,
\end{equation}
where $H\, N$ is the mean curvature vector of $f$ and $N$ is a unit normal vector. Thus, given $(p_0,x_0)
\in \Sigma\times \R^3$, there exists
a unique smooth map $f^*\colon \Sigma \to \R^3$ satisfying \eqref{123}
with $f^*(p_0)=x_0$ if and only if $f$ is minimal. \eqref{123} implies that $f^*$ is an 
immersion, and the pullback metric $f^*\langle ,\rangle $ is $ds^2$ 
since $J$ is an isometry of $ds^2$. $f^*$ turns out to be minimal as well,
and is called the conjugate minimal immersion to $f$. $f$ and $f^*$ are 
uniquely determined up to translations in $\R^3$.

Gro{\ss}e-Brauckmann, Kusner and Sullivan~\cite{gbks}, see also 
Gro{\ss}e-Brauckmann and
Polthier~\cite{gbp1}, gave a similar description  {\em à la Frobenius}
for the minimal cousin in $\mathbb{S}^3$ of a CMC $H=1$ immersion in $\mathbb{R}^3$, using quaternionic geometry.
The isometric immersions are
\[
 f\colon (\Sigma,ds^2)\to \mathbb S^3,
 \qquad f_1\colon(\Sigma,ds^2)\to \mathbb R^3.
\]
Seeing $f_1,f$ as taking values respectively in the imaginary quaternions $\R^3\equiv \Im(\H)$
and in the unitary quaternions $\esf^3\subset \H$, they are related by the first order PDE 
\begin{equation}\label{eq13}
df=f\cdot df_1\circ J,
\end{equation}
where $\cdot $ denotes quaternionic product (observe the analogy between~\eqref{123} and~\eqref{eq13}).
Since $f_1$ is valued in $\Im(\H)$, its differential $df_1$ is a 1-form along $f_1$ with
values in $\Im(\H)=T_{\bf e}\esf^3$, and thus, the right-hand side 
of~\eqref{eq13} is a 1-form taking values in $\H$. As 
$\Sigma$ is simply connected, the existence of a smooth solution
$f\colon\S \to \H$ of~\eqref{eq13} 
is again equivalent to the fact that the right-hand side
of~\eqref{eq13} is closed. Without assuming that $f_1$ has constant
mean curvature, a direct computation gives
\begin{equation*}
d(f\cdot df_1\circ J)=2(1-H_1)f\cdot N_1\, da,
\end{equation*}
where $H_1\, N_1$ is the mean curvature vector of $f_1$. Therefore,
the integrability condition of~\eqref{eq13} is that
$f_1$ has constant mean curvature $H_1=1$.

By the Frobenius theorem, every choice of an initial condition $f(p_0)
={\bf x}_0\in \esf^3$ determines uniquely a solution 
$f\colon\S \to \H$ of~\eqref{eq13} with this initial condition.

For every $X\in T\Sigma$, equation~\eqref{eq13} and the
quaternionic product identities give 
\[
\langle f,df(X)\rangle
=
\bigl\langle f,f\cdot df_1(JX)\bigr\rangle
=
\|f\|^2
\bigl\langle {\bf e},df_1(JX)\bigr\rangle
=
0.
\]
Hence $d(\|f\|^2)=0$, so $\|f\|$ is constant on $\Sigma$.
Since $f(p_0)\in\mathbb S^3$, it follows that
\(f(\Sigma)\subset\mathbb S^3\).

Consequently, left multiplication by \(f(p)\) is an isometry for
every \(p\in\Sigma\). Since \(J\) is an isometry and \(f_1\) is an
isometric immersion, equation~\eqref{eq13} implies that \(f\) is an
isometric immersion as well.
A straightforward computation gives that $f$ is minimal.
Reciprocally, every isometric minimal immersion $f\colon(\S,ds^2) \to \esf^3$
gives rise to an isometric CMC $H=1$ immersion into
$\R^3=\Im(\H)$ via~\eqref{eq13}. The immersions $f_1$ and $f$ are uniquely determined up to translations in $\R^3$ and left translations in $\esf^3$, respectively.

The main geometric objects associated to $f,f_1$ are related as follows,
where $\{e_1,e_2=Je_1\}$ is a positively oriented principal frame of $f$, with principal curvatures $\kappa_1,\kappa_2$. We use the convention $A:=-dN$, identifying tangent vectors with their images under the immersion; intrinsically, $df\circ A=-dN$. Thus $\sigma(X,Y)=ds^2(AX,Y)$ and $H=\tfrac12\mbox{trace}(A)$:
\begin{eqnarray}
\mbox{Unit normal vectors:}&N=f\cdot N_1.& \label{eq12}
\\
\mbox{Second fundamental forms:}&
\sigma_1(X,Y)=ds^2(X,Y)-\sigma(X,JY).&\nonumber
\\
\mbox{Shape operators:}&
A_1=I_2+J\circ A=I_2-A\circ J.&\nonumber
\\
\mbox{Principal directions:}&e_{1,\pm}=\frac{1}{\sqrt{2}}(e_1\pm e_2).&\nonumber 
\\
\mbox{Principal curvatures:}&
\kappa_{1,+}=1+\kappa_1,
\quad
\kappa_{1,-}=1+\kappa_2. &\nonumber
\\
\mbox{Norm of second fundamental forms:}&|A_1|^2=|A|^2+2.\nonumber
\\
\mbox{Jacobi operators:}&\mathcal{L}_{f_1}:=\Delta+|A_1|^2=\Delta+|A|^2+2:=
\mathcal{L}_{f}.\nonumber
\end{eqnarray}
To check the signs in this table, set $\alpha=f^{-1}\cdot df$.
Equation~\eqref{eq13} gives $df_1=-\alpha\circ J$, and the chosen normal
satisfies $N_1=f^{-1}\cdot N$. Since
$\alpha(X)\cdot N_1=-\alpha(JX)$, the Weingarten equation yields
\[
 dN_1=-\alpha\circ A-\alpha\cdot N_1
      =\alpha\circ(J-A).
\]
On the other hand, $dN_1=-df_1\circ A_1=\alpha\circ JA_1$.
The injectivity of $\alpha$ gives $JA_1=J-A$, and hence
$A_1=I_2+JA=I_2-AJ$, because $AJ=-JA$ for a minimal surface.
This also determines the second fundamental form and the stated
ordering of the two principal curvatures.

\begin{remark}\label{rmk3.1}
{\rm
\begin{enumerate}
\item \label{rmk3.1.1} Equations~\eqref{123} and~\eqref{eq13} are two instances of the classic
Lawson correspondence, and both of them can be encoded by a single 
application of the Frobenius theorem, which deals with first order PDEs; 
this is in contrast with the fact that the Gauss and Codazzi equations (which govern the
Lawson correspondence), are PDEs of second order. The reason for this
apparent simplification is that the unit
normal vector of $f^*$ in~\eqref{123} and of $f_1$ in~\eqref{eq13} are
directly determined by that of $f$, by the well-known equation $N^*=N$
in the first case and by~\eqref{eq12} in the second one. 

\item Equation~\eqref{eq12} can be rewritten as $f^{-1}\cdot N=N_1$, i.e., the left invariant
Gauss map of $f$ coincides with the classical Gauss map of $f_1$. 
\end{enumerate}
}
\end{remark}

\section{\texorpdfstring{Conjugation of CMC 1 surfaces in $\R^3$}{Conjugation of CMC 1 surfaces in R3}}
\label{sec8.1}
Next, we will explicitly describe 
how to obtain the first-order Frobenius formulation of cases (c), (d) of the Lawson correspondence stated in the introduction. 
The construction requires two applications of the
Frobenius theorem: one at the level of the immersions and another
at the level of their tangent frames, or equivalently of their unit
normal fields. A closely related flat $\mathrm{SO}(3)$-connection
was introduced by Korevaar, Kusner, and Ratzkin
in~\cite[Proposition~13 and Appendix~A]{kkr1}; their connection
admits the interpretation of a sphere spinning with angular speed
$2$ along the surface. To the best of our knowledge, however, the
two-step Frobenius reconstruction of the conjugate CMC immersion
developed below has not previously been written explicitly.

We start with an isometric immersion $f\colon (\Sigma,ds^2)\to \R^3$ with unit normal vector $N$ and area element $da$. The mean curvature $H$ of $f$ with respect to $N$ satisfies~\eqref{124}
(we do not assume for the moment that $H$ is constant). Consider the
following first order PDE in the unknowns $f^*\colon \S\to \R^3$, $R\colon \S \to SO(3)$ (both smooth maps) and $c\in \R$:
\begin{equation}\label{eq134}
R\, df^*=df\circ (J+c\, I),
\end{equation}
where $I$ is the identity map in $T\Sigma$. First suppose that there exist $f^*,R,c$ such that~\eqref{eq134} holds.  Since $I,J$ are linearly independent, $f^*$ is an immersion. A direct computation gives that the pullback metric by $f^*$ of the inner product $\langle ,\rangle $
in $\R^3$ is $(1+c^2)\langle, \rangle$, hence the only way $f^*
\colon (\S,ds^2)\to\R^3$
can be an isometric immersion is that $c=0$. We will assume this
in the sequel.

As in the previous cases, in order to analyze the 
existence of $f^*$
solving~\eqref{eq134} we must investigate
when the $\R^3$-valued one-form 
\[
\theta_1:=R^{\top}df\circ J
\]
is closed. Let $\{X,Y\}$ be a positively oriented local orthonormal
basis of $T\Sigma$; thus $\{df(X),df(Y),N\}$ is a positively oriented
orthonormal basis of $\R^3$.
\begin{lemma}
\label{lem8.1}
\begin{equation}
d\theta_1 (X,Y)=
R^{\top}\left[ 		dR(X)\, R^{\top}\, df(X)+dR(Y)\, R^{\top}\, df(Y)-2H\, N
\right].	\label{eq135}
\end{equation}
\end{lemma} 
\begin{proof}
A direct computation using that $d\theta_1(X,Y)=X(\theta_1(Y))-Y(\theta_1(X))-\theta_1([X,Y])$, the 
product rule and that $(dR)\, R^{\top}+R\, d(R^{\top})=d(R\, R^{\top})=dI_3=0$, gives	
\begin{equation}
	\label{eq136}
\begin{aligned}
d\theta_1(X,Y)=R^{\top}\bigl[&dR(X)R^{\top}df(X)
 +dR(Y)R^{\top}df(Y)\\
&-X(df(X))-Y(df(Y))-df(J[X,Y])\bigr].
\end{aligned}
\end{equation}
Let $\nabla ,\sigma$ be
the Levi-Civita connection of $(\S,ds^2)$ and the second fundamental form of $f$, respectively. Since $X(df(X))=df(\nabla _XX)+\sigma(X,X)N$ (analogously for $Y$) and 
\begin{equation}\label{eq137'}
\nabla_XX+\nabla _YY+J[X,Y]=0,
\end{equation}
the right-hand sides of~\eqref{eq135} and~\eqref{eq136} coincide.
\end{proof}

By Lemma~\ref{lem8.1},
we are interested in the condition in $f$ in order to solve for 
$R\colon \S\to SO(3)$ the first order PDE
\begin{equation}
\label{eq137}
dR(X)\, R^{\top}\, df(X)+dR(Y)\, R^{\top}\, df(Y)=2H\, N.
\end{equation}
Consider the vector space $\mathcal{M}_3(\R)$ of real $3\times 3$ matrices, the linear subspace $\mathfrak{so}(3)\subset \mathcal{M}_3(\R)$ of skew-symmetric matrices ($\mathfrak{so}(3)$ is the Lie algebra of $SO(3)$), and the linear isomorphism $\phi\colon \R^3\to \mathfrak{so}(3)$ given by
\[
\phi(w)=\left( \begin{array}{rrr}
	0 & -c & b
	\\
	c & 0 & -a
	\\
	-b & a & 0
\end{array}\right),
\]
where $w=(a,b,c)$. A direct computation gives that for all $w,w'\in \R^3$,
\begin{eqnarray}
\phi(w)\, w'&=&w\times w', \label{eq138'}
\\
\phi(w)\, \phi(w')-\phi(w')\, \phi(w)&=&\phi (w\times w').  \label{eq138''}
\end{eqnarray}

\begin{lemma}\label{lem8.2}
Let $f\colon (\Sigma,ds^2)\to \R^3$  be an isometric immersion with unit normal vector $N$ and mean curvature $H$. Then, the integrability condition for the first order PDE 
in the unknown $R\colon \S \to SO(3)$
\begin{equation}\label{eq139}
dR=-[\phi \circ df\circ (I+J)]\, R
\end{equation}
is that $f$ has CMC $H=1$.
\end{lemma}
\begin{proof}
We will show that~\eqref{eq139} satisfies that the formal mixed partial derivatives
are equal, i.e., the $\mathcal{M}_3(\R)$-valued 1-form 
\[
V\in T\S \mapsto 
\theta_2(V):=\left[ \phi\circ df\circ (I+J)\right](V)\, R
\]
is closed.
We continue using a positively oriented local orthonormal basis $\{ X,Y\}$ of $T\S$.
\begin{align}
d\theta_2(X,Y)
&=X(\theta_2(Y))-Y(\theta_2(X))-\theta_2([X,Y]) \notag\\
&=X\bigl(\phi(df(Y-X))R\bigr)-Y\bigl(\phi(df(X+Y))R\bigr) \notag\\
&\quad-\phi(df([X,Y]))R-\phi(df(J[X,Y]))R \notag\\
&\stackrel{(a)}{=}\phi\bigl[X(df(Y))-X(df(X))-Y(df(X))-Y(df(Y))\bigr]R \notag\\
&\quad+\phi(df(Y-X))dR(X)-\phi(df(X+Y))dR(Y) \notag\\
&\quad-\phi(df([X,Y]))R-\phi(df(J[X,Y]))R.\label{eq141}
\end{align}
where in $(a)$ we have used the product rule and
the fact that $\phi$ is linear. The Gauss equation for $f$ and~\eqref{eq137'} give
\begin{equation}\label{eq142}
X(df(Y))-X(df(X))-Y(df(X))-Y(df(Y))=df([X,Y])+df(J[X,Y])-2H\, N.
\end{equation}
Plugging~\eqref{eq142} into~\eqref{eq141} we get
\begin{align*}
d\theta_2(X,Y)
&=-2H\phi(N)R+\phi(df(Y-X))dR(X)-\phi(df(X+Y))dR(Y)\\
&\stackrel{\eqref{eq139}}{=}
 \bigl[-2H\phi(N)-\phi(df(Y-X))\phi(df(X+Y))\\
&\hspace{4.4cm}+\phi(df(X+Y))\phi(df(Y-X))\bigr]R\\
&=2\bigl[-H\phi(N)-\phi(df(Y))\phi(df(X))
 +\phi(df(X))\phi(df(Y))\bigr]R\\
&\stackrel{\eqref{eq138''}}{=}
2\bigl[-H\phi(N)+\phi(df(X)\times df(Y))\bigr]R
=2(1-H)\phi(N)R.
\end{align*}
Now the lemma is proved.
\end{proof}
\begin{remark}
{\rm 
The flatness criterion in Lemma~\ref{lem8.2} is closely related to, but distinct from, the spinning-sphere connection of
\cite[Proposition~13 and Appendix~A]{kkr1}.}
\end{remark}

From now on, we will assume that the isometric immersion $f\colon (\Sigma,ds^2)\to \R^3$ has CMC $H=1$ with respect to $N$.
By Lemma~\ref{lem8.2}, every choice $(p_0,R_0)\in \S\times SO(3)$ produces a unique 
smooth map $R\colon \S\to SO(3)$ solving~\eqref{eq139} and such that $R(p_0)=R_0$.
Thus, for every $p\in \S$ the differential
$dR_p$ maps $T_p\S$ into $T_{R(p)}SO(3)=R(p)\, \mathfrak{so}(3)$. Equation~\eqref{eq139} gives that
\[
dR_p(v)=-[\phi \circ df(v+Jv)]\, R(p)=R(p)\, M_v,
\]
where $M_v=-R(p)^{\top}\, [\phi \circ df(v+Jv)]\, R(p)\in \mathfrak{so}(3)$, or in short,
\begin{equation}\label{eq144}
M=-R^{\top}\, [\phi \circ df\circ (I+J)]\, R\colon \S\to \mbox{Hom}(T\S,\mathfrak{so}(3)).
\end{equation}

Now, we have
\begin{align*}
&dR(X)R^{\top}df(X)+dR(Y)R^{\top}df(Y)\\
&\quad\stackrel{\eqref{eq144}}{=}
-\phi(df(X+Y))df(X)-\phi(df(Y-X))df(Y)\\
&\quad\stackrel{\eqref{eq138'}}{=}
-df(X+Y)\times df(X)-df(Y-X)\times df(Y)=2N,
\end{align*}
that is, $R$ solves~\eqref{eq137}. Lemma~\ref{lem8.1} and the Frobenius Theorem imply that the first order PDE~\eqref{eq134} (with $c=0$) 
admits a unique solution $f^*\colon \S\to \R^3$
for any initial condition $f^*(p_0)=x^*_0\in \R^3$. As discussed in the 
first paragraph of this section, $f^*(\Sigma,ds^2)\to \R^3$ is an
isometric immersion. It remains to show that $f^*$ has CMC $H^*=1$.
Clearly, from~\eqref{eq134} one has that 
\begin{equation}\label{eq146}
N^*:=R^{\top}\, N
\end{equation}
is a unit normal vector for $f^*$. Denoting by $H^*$ the mean curvature
of $f^*$ with respect to $N^*$ and applying~\eqref{124} to $f^*$, we have
\begin{eqnarray*}
2H^*\, N^*&=&-[d(df^*\circ J)](X,Y) 
\stackrel{\eqref{eq134}}{=}[d\left( R^{\top}\, df\right)](X,Y)
\\
&=&
X(R^{\top}\, df(Y))-Y(R^{\top}\, df(X))-R^{\top}\, df[X,Y]
\\
&=& -R^{\top}\, dR(X)\,  R^{\top}\, df(Y)+R^{\top}\, dR(Y)\,  R^{\top}\, df(X)
\\
& & + R^{\top}\, X(df(Y))- R^{\top}\, Y(df(X))
-R^{\top}\, df[X,Y]
\\
&=& -R^{\top}\, \left[ dR(X)\,  R^{\top}\, df(Y)- dR(Y)\,  R^{\top}\, df(X)\right]
\\
&\stackrel{\eqref{eq144}}{=}& 
R^{\top}\, \left[ \phi (df(X+Y))\, df(Y)- \phi (df(Y-X))\, df(X)\right]
\\
&\stackrel{\eqref{eq138'}}{=}& 
R^{\top}\, \left[ df(X+Y)\times df(Y)- df(Y-X)\times df(X)\right]=2R^{\top}\,N
\\
&\stackrel{\eqref{eq146}}{=}&2N^*,
\end{eqnarray*}
hence $H^*=1$. To identify the resulting immersion as the conjugate CMC immersion,
we also compute its shape operator. Since $df(V)\times N=-df(JV)$,
\[
dN^*=d(R^{\top}N)\stackrel{\eqref{eq139}}{=}R^{\top}\bigl[\phi(df\circ(I+J))N-df\circ A\bigr]
       =R^{\top}df\circ(I-J-A).
\]
Comparing with $dN^*=-df^*\circ A^*=-R^{\top}df\circ JA^*$ gives
\begin{equation}\label{eq:CMC-shape-conjugate}
 A^*=I-J(A-I).
\end{equation}
Thus the traceless shape operators satisfy $A^*-I=-J(A-I)$, which is
precisely the conjugation convention used here.
In summary, we have the following result.

\begin{theorem}\label{thm8.3}
Every pair of conjugate isometric immersions $f,f^*\colon (\Sigma,ds^2)\to \R^3$ with CMC $H=1$ defined on a simply connected Riemannian surface are related by the
system of first order PDEs
\begin{eqnarray}
R\, df^*&=&df\circ J, \label{eq147}
\\
dR&=&-[\phi \circ df\circ (I+J)]\, R. \label{eq148}
\end{eqnarray}
For a fixed immersion $f$, each choice of $R(p_0)=R_0\in SO(3)$
and $f^*(p_0)=x_0^*\in\mathbb R^3$ determines a unique solution.
Changing these choices changes $f^*$ by an orientation-preserving rigid
motion. Prescribing $R_0$ is equivalent to prescribing a full oriented
orthonormal frame for $f^*$ at $p_0$, rather than only its tangent plane.
\end{theorem}

\section{\texorpdfstring{Conjugation of minimal surfaces in $\esf^3$}{Conjugation of minimal surfaces in S3}}
\label{secconjminS3}

By concatenating the cousin correspondence and the conjugation of CMC 1 surfaces developed in Section~\ref{sec8.1}, we can describe the Frobenius scheme for the conjugation of minimal surfaces in $\esf^3$:
\begin{theorem}\label{thm8.4}
Every pair of conjugate isometric minimal immersions $f,f^*\colon (\Sigma,ds^2)\to \esf^3$ defined on a simply connected Riemannian surface are related by the
system of first order PDEs
\begin{eqnarray}
(f^*)^{-1}\cdot df^*&=&R^{\top}(f^{-1}\cdot df)\circ J. \label{eq150}
\\
dR&=&[\phi \circ (f^{-1}\cdot df)\circ (J-I)]\, R. \label{eq151}
\end{eqnarray}
Moreover, the unit normal vectors $N,N^*$ of $f,f^*$ satisfy
\begin{equation}
(f^*)^{-1}\cdot N^*=R^{\top}(f^{-1}\cdot N). \label{eq152}
\end{equation}
\end{theorem}
\begin{proof}
Consider the isometric CMC~$1$ immersion $f_1\colon (\Sigma,ds^2) \to \R^3$, cousin to $f$, given 
by equation~\eqref{eq13} (defined uniquely up to a choice $f_1(p_0)\in \R^3$), and let
$N_1$ be the unit normal vector to $f_1$ given by~\eqref{eq12}. 
Consider
the conjugate CMC $1$ immersion $f_1^*\colon (\Sigma,ds^2)\to \R^3$ given by 
Theorem~\ref{thm8.3} (uniquely defined up to the choices of $f_1^*(p_0)\in \R^3$ and
$R(p_0)\in SO(3)$) and let $N_1^*=R^{\top}\, N_1$ be the unit normal vector to $f_1^*$ given
by~\eqref{eq146}. Finally, consider the isometric minimal immersion $f^*\colon (\Sigma,ds^2) \to \esf^3$, cousin to $f_1^*$, given by~\eqref{eq13} (with $f_1^*$ instead of $f_1$ and $f^*$ instead of $f$). Then,
\begin{eqnarray*}
df^*&=&f^*\cdot df_1^*\circ J\stackrel{\eqref{eq147}}{=}f^*\cdot R^{\top}df_1\circ J^2=
-f^*\cdot R^{\top}df_1\stackrel{\eqref{eq13}}{=}f^*\cdot R^{\top}(f^{-1}\cdot df)\circ J,
\hspace{1cm}
\\
dR&\stackrel{\eqref{eq148}}{=}&-[\phi\circ df_1\circ (I+J)]R
\stackrel{\eqref{eq13}}{=}
[\phi\circ (f^{-1}\cdot df)\circ (J-I)]R,
\end{eqnarray*}
which proves~\eqref{eq150} and~\eqref{eq151}. Also,
\[
N^*\stackrel{\eqref{eq12}}{=}f^*\cdot N_1^*\stackrel{\eqref{eq146}}{=}\\
f^*\cdot R^{\top}\, N_1\stackrel{\eqref{eq12}}{=}f^*\cdot R^{\top}(f^{-1}\cdot N),
\]
which gives~\eqref{eq152}.
\end{proof}

\begin{remark}
{\rm 
\begin{enumerate}
\item From any of the equations~\eqref{eq150} and~\eqref{eq152} we deduce
the geometric meaning of the map $R\colon \S\to SO(3)$, which tells us 
how to construct the tangent plane to $f^*$ at each point in terms of that
of $f$ (and vice versa); because of this, we will call $R$ the {\em frame map}. This generalizes the case of conjugate minimal
surfaces in $\R^3$ and that of the conjugate cousin minimal immersion 
in $\esf^3$ to a CMC $H=1$ surface in $\R^3$: in both situations, $R$
is constant (hence we can prescribe the initial condition so that $R= I_3\in SO(3)$), compare this to item~\ref{rmk3.1.1} of Remark~\ref{rmk3.1}.
		
\item Using the expression that relates the second fundamental forms of correspondent CMC immersions via the Lawson correspondence (see~Lawson~\cite[proof of Theorem~8]{la3}), it can be easily 
proven that the main geometric objects associated to $f,f^*$ are related as follows, with the obvious notation:
\begin{eqnarray}
\mbox{Second fundamental forms:}&\sigma^*(X,Y)=\sigma(X,JY).&\nonumber 
\\
\mbox{Shape operators:}&A^*=-J\circ A.&\nonumber
\\
\mbox{Principal curvatures:}&\kappa_{i}^*=\kappa_i, \ i=1,2.&\nonumber
\\
\mbox{Principal (resp. asymptotic) directions:}& \mbox{They interchange for $f$ and $f^*$.}\nonumber
\\
\mbox{Norm of second fundamental forms:}&|A^*|^2=|A|^2.\nonumber
\\
\mbox{Jacobi operators:}&\mathcal{L}_{f^*}=\mathcal{L}_{f}=\Delta+|A|^2+2.
\label{eq146'}
\end{eqnarray}
\end{enumerate}
}
\end{remark}

\subsection{\texorpdfstring{Interpretation of the frame map and its PDE}{Interpretation of the frame map and its PDE}}
Suppose that $f\colon (\Sigma,ds^2)\to \esf^3$ is an isometric immersion, not necessarily minimal. Define the one-forms
\begin{eqnarray}
\alpha&:=&f^{-1}\cdot df\colon T\Sigma\to \Im (\mathbb{H}),\nonumber
\\
\mathcal A&:=&
\phi\circ\alpha\circ(J-I)\colon T\Sigma\to \mathfrak{so}(3).
\label{A}
\end{eqnarray}
$\alpha$ is the (left) Maurer–Cartan form of $f$, or more precisely, the pullback by $f$ of the Maurer-Cartan form of $\mathbb{S}^3$. Thus, $\alpha$ satisfies 
\begin{equation*}
\begin{split}
d\alpha& =
d(f^{-1})\wedge df
+
f^{-1}\cdot d(df)=-f^{-1}\cdot df\cdot f^{-1}\wedge df
 \\
 & =
-\bigl(f^{-1}\cdot df\bigr)
\wedge
\bigl(f^{-1}\cdot df\bigr)
=
-\alpha\wedge\alpha,
\end{split}
\end{equation*}
or equivalently, 
\begin{equation}\label{alphaMC}
d\alpha+\alpha\wedge\alpha=0.	
\end{equation}

Consider the trivial Euclidean vector bundle
$E=\Sigma\times\mathbb R^3\to\Sigma$.
Since $\mathcal A$ is an $\mathfrak{so}(3)$-valued one-form, it defines a connection $\nabla^{\mathcal A}$ on the space $\Gamma(E)=C^{\infty}(\Sigma,\mathbb{R}^3)$ of sections of $E$, given by
\begin{equation*}
	\nabla^{\mathcal A}_Xv
	:=dv(X)-\mathcal A(X)v,
	\qquad
	X\in\mathfrak X(\Sigma),\quad
	v\in\Gamma(E).
\end{equation*}

The connection $\nabla^{\mathcal A}$ is metric. Indeed, since
$\mathcal A(X)$ is skew-symmetric,
\[
\langle \mathcal A(X)v,w\rangle
+\langle v,\mathcal A(X)w\rangle=0, \qquad \mbox{for all $X\in\mathfrak X(\Sigma),\quad v,w\in \Gamma(E)$,}
\]
and therefore $X\langle v,w\rangle
=\langle\nabla^{\mathcal A}_Xv,w\rangle
+\langle v,\nabla^{\mathcal A}_Xw\rangle$.

The curvature of $\nabla^{\mathcal A}$ is the endomorphism-valued two-form
defined by
\[
\mathcal R^{\mathcal A}(X,Y)v
:=\nabla^{\mathcal A}_X\nabla^{\mathcal A}_Yv
-\nabla^{\mathcal A}_Y\nabla^{\mathcal A}_Xv
-\nabla^{\mathcal A}_{[X,Y]}v.
\]
A direct computation gives
\begin{equation}\label{eq:A-connection-curvature}
	\mathcal R^{\mathcal A}(X,Y)v
	=-\bigl(d\mathcal A-\mathcal A\wedge\mathcal A\bigr)(X,Y)v,
\end{equation}
where $d\mathcal{A},\mathcal A\wedge\mathcal A\in \Omega^2(\Sigma,\mathfrak{so}(3))$ are, respectively, the exterior differential of $\mathcal A$ and the form defined by
\[
(\mathcal A\wedge\mathcal A)(X,Y)
=\mathcal A(X)\mathcal A(Y)-\mathcal A(Y)\mathcal A(X).
\]
Because of~\eqref{eq:A-connection-curvature}, the two-form $d\mathcal A-\mathcal A\wedge\mathcal A$ is called the {\em curvature form} of $\nabla^{\mathcal{A}}$.

\begin{lemma}\label{lema5.3}
The curvature form is related to the mean curvature function $H$ of $f$ with respect to a unit normal $N$ by the formula
\begin{equation*}
d\mathcal A-\mathcal A\wedge\mathcal A
=
-2H\,\phi(f^{-1}\cdot N)\,da.
\end{equation*}
\end{lemma}
\begin{proof}
Choose a positively oriented local orthonormal basis $\{E_1,E_2=JE_1\}$ of $T\Sigma$ such that $\nabla_{E_i}E_j(p)=0$, $i,j\in\{1,2\}$. In particular, $[E_1,E_2](p)=0$. Set
\begin{equation}\label{ab}
a:=\alpha(E_1)=f^{-1}\cdot df(E_1),
\qquad
b:=\alpha(E_2)=f^{-1}\cdot df(E_2),
\end{equation}
hence $N_1:=f^{-1}\cdot N=a\times b$ and $(a,b,N_1)$
is a positively oriented orthonormal basis of $\Im (\mathbb{H})$. 
Evaluating~\eqref{alphaMC} on $\{E_1,E_2\}$, we obtain
\[
d\alpha(E_1,E_2)=-(\alpha\wedge\alpha)(E_1,E_2)
=
-\bigl(
\alpha(E_1)\cdot\alpha(E_2)
-
\alpha(E_2)\cdot\alpha(E_1)
\bigr)=-(a\cdot b-b\cdot a).
\]
Since $a,b$ are imaginary quaternions, $a\cdot b-b\cdot a=2a\times b=2 N_1$, hence $d\alpha(E_1,E_2)=-2N_1$, or equivalently,
\begin{equation}\label{dalpha}
d\alpha=-2N_1\, da.
\end{equation}
On the other hand, at the point $p$, where $[E_1,E_2]=0$, we have
\begin{align}
d(\alpha\circ J)(E_1,E_2)
	&=
	E_1\bigl((\alpha\circ J)(E_2)\bigr)
	-
	E_2\bigl((\alpha\circ J)(E_1)\bigr)
	-
	(\alpha\circ J)([E_1,E_2])\notag\\
	&=
	-E_1(a)-E_2(b).
	\label{eq:dalphaJ-first}
\end{align}

It remains to compute $E_1(a)$ and $E_2(b)$. 
\begin{equation}\label{eq:E1a}
E_1(a)=E_1(f^{-1}\cdot df(E_1))	=
-f^{-1}\cdot df(E_1)\cdot f^{-1}\cdot df(E_1)
+
f^{-1}\cdot \overline{\nabla}_{E_1}\bigl(df(E_1)\bigr),
\end{equation}
where $\overline{\nabla}$ denotes the standard connection of
$\mathbb R^4\equiv \mathbb H$. Since $a$ is a unit imaginary quaternion, the first term in the last right-hand side is $-
a^2={\bf e}$. The second term can be computed by using the Gauss formulas of $f$ in $\mathbb{S}^3$ and of $\mathbb{S}^3$ in $\mathbb{R}^4$:
\[
\overline{\nabla}_{E_1}\bigl(df(E_1)\bigr)
=
df(\nabla_{E_1}E_1)
+
\sigma (E_1,E_1)N
-
\langle E_1,E_1\rangle f
=
\sigma (E_1,E_1)N-f,
\]
where we have used that $\nabla_{E_1}E_1=0$ at $p$. Therefore,
\[
f^{-1}\cdot \overline{\nabla}_{E_1}\bigl(df(E_1)\bigr)
=\sigma (E_1,E_1)N_1-{\bf e},
\]
and substituting in~\eqref{eq:E1a} we get
$E_1(a)=\sigma (E_1,E_1)N_1$. Analogously,
$E_2(b)=\sigma (E_2,E_2)N_1$. 
Consequently, \eqref{eq:dalphaJ-first} becomes
\[
d(\alpha\circ J)(E_1,E_2)
=
-(\sigma (E_1,E_1)+\sigma (E_2,E_2))N_1=-2H\, N_1,
\]
or equivalently,
\begin{equation}\label{dalphaJ}
d(\alpha\circ J)
=
-2H\, N_1\, da.
\end{equation}
Observe that formulas~\eqref{dalpha} and~\eqref{dalphaJ} hold globally on $\Sigma$, although to prove them we have chosen a basis $\{E_1,E_2\}$ adapted to a point $p$.

Let
\(
w:=\alpha\circ(J-I),
\)
so $\mathcal A=\phi\circ w$. Since $\phi\colon\mathbb R^3\to\mathfrak{so}(3)$ is a fixed linear
map, exterior differentiation commutes with $\phi$. More precisely,
$d\mathcal{A}=d(\phi \circ w)=\phi \circ dw$. Therefore, the curvature form of $\nabla^{\mathcal{A}}$ is determined by the skew-symmetric matrix
\begin{eqnarray*}
[d\mathcal A-\mathcal A\wedge\mathcal A](E_1,E_2)&=&
(\phi \circ dw)(E_1,E_2)- [(\phi\circ w)\wedge(\phi \circ w)](E_1,E_2) \\
&\stackrel{\eqref{eq138''}}{=}&
\phi(dw(E_1,E_2))-[\phi(w(E_1)),\phi(w(E_2))]
\\
&=&
\phi\bigl[
dw(E_1,E_2)-w(E_1)\times w(E_2)
\bigr].
\end{eqnarray*}	
On the other hand,
\begin{eqnarray*}
w(E_1)&=&[\alpha\circ(J-I)](E_1)=\alpha(E_2)-\alpha(E_1)=b-a,
\\
w(E_2)&=&[\alpha\circ(J-I)](E_2)=-\alpha(E_1)-\alpha(E_2)=-a-b,
\\
w(E_1)\times w(E_2)&=&(b-a)\times (-a-b)=2\, a\times b=2\, N_1.
\end{eqnarray*}

Consequently,
\begin{equation}\label{45}
[d\mathcal A-\mathcal A\wedge\mathcal A](E_1,E_2)
=
\phi\bigl(dw(E_1,E_2)\bigr)-2\phi(N_1).
\end{equation}
But
\begin{eqnarray*}
dw(E_1,E_2)&=&
d[\alpha\circ(J-I)](E_1,E_2)=
d(\alpha\circ J)(E_1,E_2)-d\alpha(E_1,E_2)
\\
&\stackrel{\eqref{dalpha},\eqref{dalphaJ}}{=}&
-2H\, N_1+2N_1=2(1-H)N_1.
\end{eqnarray*}
Plugging this into~\eqref{45} we get
\[
[d\mathcal A-\mathcal A\wedge\mathcal A](E_1,E_2)
=
\phi\bigl(2(1-H)N_1\bigr)-2\phi(N_1)=-2H\, \phi(N_1),
\]
which finishes the proof of the lemma.
\end{proof}

We again suppose that $f\colon \Sigma \to \mathbb{S}^3$ is minimal, satisfying the hypotheses of Theorem~\ref{thm8.4}. The first order PDE~\eqref{eq151} can be rewritten as
\begin{equation}\label{eq151a}
dR=\mathcal A R.
\end{equation}
Moreover, if $R=(r_1,r_2,r_3)$ is written in columns, then
\eqref{eq151a} is equivalent to $dr_i=\mathcal A r_i$, $i=1,2,3$,
and hence
\[
\nabla^{\mathcal A}r_i=0.
\]
Thus, the columns of $R$ form a positively oriented orthonormal parallel
frame for $\nabla^{\mathcal A}$.

Next we analyze the integrability condition of~\eqref{eq151a}, i.e. the obstruction to the local existence of a solution of this PDE. Differentiating~\eqref{eq151a} one gets
\[
0=d(dR)=d(\mathcal A R)\stackrel{(\star)}{=}
d\mathcal A\,R-\mathcal A\wedge dR
=
d\mathcal A\,R-\mathcal A\wedge (\mathcal{A} R)
=
\bigl(d\mathcal A-\mathcal A\wedge\mathcal A\bigr)R. 
\]
where in $(\star)$ we have used that $\mathcal A$ is a one-form.
Since $R$ is invertible, the last equality is equivalent to 
$d\mathcal A-\mathcal A\wedge\mathcal A=0$,
which is precisely the integrability condition of the first-order PDE~\eqref{eq151a} (and of equation~\eqref{eq151}). By Lemma~\ref{lema5.3}, this is equivalent to the flatness of the connection $\nabla^{\mathcal{A}}$ and to the minimality of
$f\colon (\Sigma,ds^2)\to \esf^3$ (this could be also deduced from Lemma~\ref{lem8.2} and the cousin correspondence).

Moreover, if $R=(r_1,r_2,r_3)$ is written in columns, then
\eqref{eq151a} is equivalent to $dr_i=\mathcal A r_i$, $i=1,2,3$,
and hence
\(
 \nabla^{\mathcal A}r_i=0.
\)
Thus, the columns of $R$ form a positively oriented orthonormal parallel
frame for $\nabla^{\mathcal A}$. 

We have seen that the integrability condition of~\eqref{eq151a} is equivalent to the flatness of the connection $\nabla^{\mathcal{A}}$.
This is also equivalent to the property that the parallel transport associated to $\nabla^{\mathcal{A}}$ is then locally independent of the path. 

\subsection{The frame map determines conjugate pairs of great circles and reflection geodesics}
\label{section5.2}
It is well-known (see Lawson~\cite[Proposition 13.1]{la3}) that for a pair of conjugate minimal surfaces $f,f^*\colon (\Sigma,ds^2) \to \esf^3$, given a curve $\gamma $ in the parameter surface $\S$, $f(\gamma )$ is a
geodesic arc in $\esf^3$ if and only if $f^*(\gamma )$ is a
geodesic of reflective symmetry. We now give a
more precise information about the relative position
of the ambient geodesic arc $f(\gamma )$ and the reflection two-sphere that contains $f^*(\gamma )$, via the frame map.

\begin{lemma}\label{lem8.14}
Suppose that $f,f^*
\colon (\Sigma,ds^2) \to \esf^3$ are conjugate minimal immersions, and let $\gamma $ be a smooth
curve in $\S$. If $\G=f(\gamma )$ is an arc of a great circle
in $\esf^3$, parametrized by arc length as $\G(t)=\cos t\, {\bf p}+\sin t \, {\bf v}$, $t\in [0,\ell]$ (here $({\bf p},{\bf v})\in U\esf^3$), then:
\begin{enumerate}
\item The left-translated tangent vector to $\G$ is constant; more
precisely,
\[
\alpha_{\gamma (t)}(\gamma '(t))=f(\gamma (t))^{-1}\cdot df_{\gamma (t)}(\gamma '(t))=\overline{\bf p}\cdot {\bf v}
\qquad \text{for every $t\in[0,\ell]$}.
\]

\item Denote by $\eta=J\, \gamma '$ the unit conormal vector to $\Sigma$ along $\gamma $, so that $\{ \gamma ',\eta \}$ is a positively oriented basis of $T\Sigma$. Then,
\[
{\bf a}:=df^*_{\gamma (t)}(\eta_{\gamma (t)})=
-f^*(\gamma (t))\cdot R_{\gamma (t)}^{\top}\alpha_{\gamma (t)}(\gamma '(t))
=-f^*(\gamma (t))\cdot R_{\gamma (t)}^{\top}(\overline{\bf p}\cdot {\bf v})
\] 
is a constant unit vector along~$\gamma $. Furthermore,
the curve $f^*(\gamma )$ is contained in
the two-sphere $\esf^3\cap \langle {\bf a}\rangle^{\perp}$.

\end{enumerate}
Conversely, if $f(\gamma )$ is a reflective geodesic in a totally
geodesic two-sphere
$\esf^3\cap\langle{\bf a}\rangle^\perp$, then, after fixing the sign
of $\eta$, $df_\gamma (\eta)={\bf a}$ is constant along $\gamma $, and $f^*(\gamma )$ is an arc of a great circle.
\end{lemma}
\begin{proof}
Put $T=\gamma '$. Since $f(\gamma (t))=\Gamma(t)$, we have
\begin{equation*}
\begin{split}
\alpha(T) &=\overline{\Gamma(t)}\cdot\Gamma'(t)
=
(\cos t\,\overline{\bf p}+\sin t\,\overline{\bf v})
\cdot(-\sin t\,{\bf p}+\cos t\,{\bf v})
 \\
 &=
\cos^2t\,\overline{\bf p}\cdot{\bf v}
-\sin^2t\,\overline{\bf v}\cdot{\bf p}
=
\overline{\bf p}\cdot{\bf v}.
\end{split}
\end{equation*}
In the last equality we used
$\overline{\bf v}\cdot{\bf p}=-\overline{\bf p}\cdot{\bf v}$,
which follows from $\langle{\bf p},{\bf v}\rangle=0$. This proves the
first assertion.

Set $N_1:=f^{-1}\cdot N$. Since $(T,\eta)$ is positively oriented and orthonormal, $(\alpha(T),\alpha(\eta),N_1)$ is a
positively oriented orthonormal basis of $\Im(\H)$. Evaluating~\eqref{eq150} on $\eta$ and using $J\eta=-T$, we obtain
\begin{equation}\label{eq:a-alpha-frame}
{\bf a}:=df^*(\eta)=-f^*\cdot R^{\top}\alpha(T).
\end{equation}

To see that ${\bf a}$ is constant along $\gamma $, we differentiate~\eqref{eq:a-alpha-frame} in the direction
$T$; all the following quantities are restricted to $\gamma $. Since $\alpha(T)$ is constant,~\eqref{eq:a-alpha-frame} gives
\begin{equation}\label{eq:a-alpha-derivative}
\frac{d{\bf a}}{dt}
=-df^*(T)\cdot R^{\top}\alpha(T)
  -f^*\cdot dR^{\top}(T)\alpha(T).
\end{equation}
By~\eqref{eq150},
\(
df^*(T)=f^*\cdot R^{\top}\alpha(\eta).
\)
Moreover, $dR=\mathcal A R$ and $\mathcal A^{\top}=-\mathcal A$ imply
\[
dR^{\top}=-R^{\top}\mathcal A,
\qquad
\mathcal A(T)=\phi\bigl(\alpha(JT-T)\bigr)=\phi(\alpha(\eta)-\alpha(T)).
\]
Because $R\in SO(3)$ preserves the cross product, and because the
quaternionic product of two orthogonal imaginary quaternions equals their cross product, we have
\begin{align*}
-df^*(T)\cdot R^{\top}\alpha(T)
&=-f^*\cdot(R^{\top}\alpha(\eta))\cdot(R^{\top}\alpha(T))
 =f^*\cdot R^{\top}N_1,\\
-f^*\cdot dR^{\top}(T)\alpha(T)
&=f^*\cdot R^{\top}\mathcal A(T)\alpha(T)
 =f^*\cdot R^{\top}\bigl((\alpha(\eta)-\alpha(T))\times \alpha(T)\bigr)
 =-f^*\cdot R^{\top}N_1.
\end{align*}
Thus, the two terms in~\eqref{eq:a-alpha-derivative} cancel and hence,
$\frac{d{\bf a}}{dt}=0$.
Furthermore, ${\bf a}$ is a unitary vector because $f^*$ is an isometric
immersion and $\eta$ is unit. Finally,
${\bf a}=df^*(\eta)\in T_{f^*(\gamma )}\esf^3$, and therefore
$\langle f^*(\gamma ),{\bf a}\rangle=0$.
It follows that $f^*(\gamma )\subset \esf^3\cap\langle{\bf a}\rangle^\perp$,
which proves the first statement of the lemma.

For the second statement, orthogonality of $f(\Sigma)$ to
$\esf^3\cap\langle{\bf a}\rangle^\perp$ along the reflection curve
implies, after choosing the sign of the conormal, that $df_\gamma (\eta)={\bf a}$.
Writing $T$ for the unit tangent to the reflection curve and
$\eta=JT$, the curve is an intrinsic geodesic and a curvature line.
Hence its geodesic curvature and $\sigma(T,\eta)$ vanish.
The two immersions induce the same metric, and
$\sigma^*(T,T)=\sigma(T,JT)=\sigma(T,\eta)=0$. Thus both the geodesic
curvature and the normal curvature of $f^*(\gamma )$ vanish, which means that
$f^*(\gamma )$ is an ambient geodesic in $\mathbb S^3$.
This completes the proof.
\end{proof}

\section{Jacobi conjugate functions}
\label{secJacconj}

We now linearize the conjugation system of
Theorem~\ref{thm8.4}. Let $f\colon(\Sigma,g)\to \mathbb S^3$
be a minimal immersion of a simply connected surface, and let
$u\in C^\infty(\Sigma)$ be a Jacobi function, i.e., $\mathcal L_fu=0$.  
We shall construct a vector field $X^*$ along the conjugate minimal immersion $f^*$ which formally
corresponds to conjugating a normal infinitesimal deformation $f_s$ of $f$ whose variational field is $uN$. 
The construction below does not assume that $u$ is integrable by a
family of minimal immersions; we will show that the formal first-order system to define $X^*$ is compatible
whenever $\mathcal{L}_fu=0$; equivalently, minimality of $f_s$ will only be required up to first order in $s$, 
and the normal component $u^*=\langle X^*,N^*\rangle$ will be a Jacobi function for $f^*$.

We shall use an actual normal variation $f_s$ only to compute first derivatives; these immersions $f_s$ need not be minimal when $s\ne0$. Nevertheless, if $H_s$ denotes the mean curvature function of $f_s$, we have 
\begin{equation}\label{eq:minimal-to-first-order}
2\dot H=\mathcal L_fu=0.
\end{equation}
Thus this variation is minimal to first order at $s=0$.

Let $g_s,J_s,N_s,H_s$ denote the metric, complex structure, unit
normal, and mean curvature of $f_s$, respectively, and consider the smooth one-form
\begin{equation*}
\alpha_s:=f_s^{-1}\cdot df_s\colon T\Sigma\to \Im (\mathbb{H}).
\end{equation*}
We use a dot to denote differentiation at $s=0$. Only the first jets
of these objects will be used; in particular, no frame map $R_s$ is
associated to the generally nonminimal immersions $f_s$.
\begin{remark}
{\rm 
The first jets obtained from any two normal variations of $f$ with initial velocity $uN$ coincide. Indeed, the formulas below express them solely in terms of $f,u$. 
}
\end{remark}

We begin with the variation of the metric and the complex structure; these formulas are well-known, but we include proofs for the sake of completeness.
\begin{lemma}\label{lemma8.7}
Let $f\colon(\Sigma,g)\to\mathbb S^3$ be an oriented isometric immersion, not
necessarily minimal, with second fundamental form $\sigma$, shape operator $A$, mean curvature function $H$, and traceless shape operator $B:=A-HI$.

For a normal variation with variational field $uN$, the first
variation of the metric and the complex structure are 
\begin{eqnarray}
\dot g(X,Y)&=&-2u\,g(AX,Y)=-2u\, \sigma (X,Y), \label{dotg}
\\
\dot J&=&-2uJB. \label{dotJ}
\end{eqnarray}
In particular, if $f$ is minimal, then
\begin{equation*}
\dot J=-2uJA.
\end{equation*}
Moreover,
\begin{equation}\label{eq:Jdot-anticommutes}
J\dot J+\dot J J=0.
\end{equation}
\end{lemma}
\begin{proof}
Let $f_s\colon\Sigma\to\mathbb S^3$ be a normal variation of $f=f_0$ with variational field $\left.\frac{\partial f_s}{\partial s}\right|_{s=0}
=uN$. Denote by $g_s=f_s^*\langle\,\cdot,\cdot\,\rangle$ the metric induced by $f_s$. Fix two smooth vector fields $X,Y\in\mathfrak X(\Sigma)$. Differentiating $g_s(X,Y)=\left\langle df_s(X),df_s(Y)\right\rangle$ at $s=0$ we get
\begin{eqnarray*}
\dot g(X,Y)
&=&
\langle d(uN)(X),df(Y)\rangle + \langle df(X),d(uN)(Y)\rangle
\\
&=&
u\big( \langle dN(X),df(Y)\rangle + \langle df(X),dN(Y)\rangle \big)
\\
&=&
u\big( \langle \overline{\nabla}_{df(X)}N,df(Y)\rangle + \langle df(X),\overline{\nabla}_{df(Y)}N\rangle \big)
\\
&=&
u\big( \langle \nabla^{\mathbb{S}^3}_{df(X)}N,df(Y)\rangle + 
\langle df(X),\nabla^{\mathbb{S}^3}_{df(Y)}N\rangle \big)\\
&=&
-2u\, \sigma(X,Y)=-2u\, g(AX,Y),
\end{eqnarray*}
where $\nabla^{\mathbb{S}^3}$ is the Levi-Civita connection of $\mathbb{S}^3$. This proves~\eqref{dotg}.

As for item~\eqref{dotJ}, fix an arbitrary point $p\in\Sigma$ and choose a positively
oriented $g$-orthonormal basis $\{e_1,e_2=Je_1\}$ of $T_p\Sigma$
consisting of principal directions: $Ae_1=\kappa_1e_1$, $Ae_2=\kappa_2e_2$.

Extend $\{e_1,e_2\}$ to a local frame, which is kept fixed while $s$
varies (in particular, $\{e_1,e_2\}$ is not necessarily $g_s$-orthonormal). By~\eqref{dotg}, the matrix of $g_s$ at $p$ with respect to $\{e_1,e_2\}$ can be written as
\[
\begin{pmatrix}
1-2su\kappa_1&0\\
0&1-2su\kappa_2
\end{pmatrix}
+\mathcal{O}(s^2).
\]
For a positively oriented metric whose matrix in a fixed basis is
$\left(\begin{smallmatrix}E&F\\F&G\end{smallmatrix}\right)$, the
compatible complex structure has matrix
\[
\frac{1}{\sqrt{EG-F^2}}
\begin{pmatrix}-F&-G\\ E&F\end{pmatrix}.
\]
Applying this formula to $g_s$ gives, at $p$, that the matrix of $J_s$ with respect to $\{e_1,e_2\}$ is
\[
\begin{pmatrix}
0&-1+su(\kappa_2-\kappa_1)\\
1+su(\kappa_2-\kappa_1)&0
\end{pmatrix}
+\mathcal{O}(s^2).
\]
Therefore, the matrix of $\dot{J}$ with respect to $\{e_1,e_2\}$ is
\[
u(\kappa_2-\kappa_1)
\begin{pmatrix}0&1\\1&0\end{pmatrix}.
\]
On the other hand, the matrix of $B,J$ with respect to $\{e_1,e_2\}$ are, respectively,
\[
B=\frac{\kappa_1-\kappa_2}{2}
\begin{pmatrix}1&0\\0&-1\end{pmatrix},
\qquad
J=
\begin{pmatrix}0&-1\\1&0\end{pmatrix},
\]
and direct multiplication gives $\dot J=-2uJB$. Since the argument
is pointwise and allows $\kappa_1=\kappa_2$, the identity also holds
at umbilic points. Finally, $J_s^2=-I$ for every $s$. Differentiating at $s=0$
gives~\eqref{eq:Jdot-anticommutes}.
\end{proof}

We next compute the first variations of the forms that occur in the
conjugation system. Put
\begin{equation*}
\alpha_s:=f_s^{-1}\cdot df_s\colon T\Sigma\to \Im (\mathbb{H}),
\qquad
\mathcal A_s:=\phi\circ\alpha_s\circ(J_s-I)\colon T\Sigma\to \mathfrak{so}(3).
\end{equation*}

\begin{lemma}[Linearized equations for $\alpha$ and $\mathcal A$]
\label{lem:linearized-R}
\begin{eqnarray}
\dot\alpha
&=&-u\,\alpha\circ J+f^{-1}\cdot d(uN), \label{alphadot}
\\
\dot{\mathcal A}
&=&
\phi\circ
\bigl[\dot\alpha\circ(J-I)+\alpha\circ\dot J\bigr]
\label{eq:Adot-formal}
\\
&=&
\phi\circ\left[
u\,\alpha\circ(I+J)
+\bigl(f^{-1}\cdot d(uN)\bigr)\circ(J-I)
+\alpha\circ\dot J
\right].
\label{eq:Adot-expanded}
\end{eqnarray}
\end{lemma}

\begin{proof}
Differentiating $\alpha_s=f_s^{-1}\cdot df_s$ at $s=0$ yields
\begin{equation}\label{dotalpha1}
\dot\alpha
=-(f^{-1}\cdot uN)\cdot(f^{-1}\cdot df)
+f^{-1}\cdot d(uN).
\end{equation}
Let $N_1:=f^{-1}\cdot N\in \Im (\mathbb{H})$, where $N$ is the unit normal that orients $f$. If $\{e_1,e_2=Je_1\}$ is a positively oriented orthonormal basis of $T\Sigma$, then 
$\{\alpha(e_1),\alpha(e_2),N_1\}$ is a positively oriented orthonormal
basis of $\Im (\mathbb{H})$. Since quaternionic multiplication of orthogonal imaginary quaternions agrees with their cross product, we deduce that $N_1\cdot \alpha(e_1)=\alpha(e_2)=\alpha(Je_1)$ and $N_1\cdot \alpha(e_2)=-\alpha(e_1)=\alpha(Je_2)$. Thus, 
\begin{equation}\label{alphaJ}
N_1\cdot\alpha(X)=\alpha(JX), \qquad \mbox{for every $X\in T\Sigma$.}
\end{equation}
Substituting~\eqref{alphaJ} into~\eqref{dotalpha1}
proves~\eqref{alphadot}.

Differentiating
$\mathcal A_s:=\phi\circ\alpha_s\circ(J_s-I)$ at $s=0$ gives
\eqref{eq:Adot-formal}. Substitution of~\eqref{alphadot}, together
with $J^2=-I$, gives~\eqref{eq:Adot-expanded}.
\end{proof}

We now separate the actual variation from the formal first-order
objects. If $E_s$ is any object associated
with the chosen variation, its first jet will be denoted by
$E_\varepsilon:=E+\varepsilon\dot E$.
All remainders $\mathcal O(\varepsilon^2)$ below are understood
locally in $C^k$ on compact subsets, for each fixed finite $k$. In
particular, the spatial differentiations of first-jet identities used
in the proof preserve the stated order. Actual variations are needed
only locally where they remain immersions.
Thus we set
\begin{equation}\label{alphaJAepsilon}
\begin{split}
g_\varepsilon:=g+\varepsilon\dot g,
\quad
J_\varepsilon:=J+\varepsilon\dot J,
\quad
\alpha_\varepsilon:=\alpha+\varepsilon\dot\alpha,
\\
\mathcal A_\varepsilon:=\mathcal A+
\varepsilon\dot{\mathcal A},
\quad
H_\varepsilon:=H+\varepsilon\dot H\stackrel{\eqref{eq:minimal-to-first-order}}{=}H=0.
\end{split}
\end{equation}

{\bf Claim.} The following relations hold:
\begin{eqnarray}
J_{\varepsilon}^2
&=&
-I+\mathcal{O}(\varepsilon^2),
\label{eqJepsilon2}
\\
g_\varepsilon(J_\varepsilon X,J_\varepsilon Y)
&=&
g_\varepsilon(X,Y)+\mathcal{O}(\varepsilon^2),
\label{eq:Jepsilon-compatible}
\\
g_\varepsilon(J_\varepsilon X,Y)
&=&
-g_\varepsilon(X,J_\varepsilon Y)
+\mathcal O(\varepsilon^2), \label{eq:Jepsilon-skew-adjoint}
\\
\langle\alpha_\varepsilon(X),\alpha_\varepsilon(Y)\rangle
&=&
g_\varepsilon(X,Y)+\mathcal{O}(\varepsilon^2),
\label{eq:alphaepsilon-isometric}
\\
d\alpha_\varepsilon+
\alpha_\varepsilon\wedge\alpha_\varepsilon
&=&
\mathcal{O}(\varepsilon^2),
\label{eq:formal-alpha}
\\
\mathcal A_\varepsilon
&=&	
\phi\circ\alpha_\varepsilon\circ(J_\varepsilon-I)+ \mathcal{O}(\varepsilon^2). \label{eq75}
\end{eqnarray}
{\em Proof of the claim.} Equation~\eqref{eqJepsilon2} follows from expanding $J_{\varepsilon}^2$ and using~\eqref{eq:Jdot-anticommutes}.
\begin{eqnarray*}
g_\varepsilon(J_\varepsilon X,J_\varepsilon Y)-g_\varepsilon(X,Y)
&\stackrel{\eqref{alphaJAepsilon}}{=}& [g(JX,JY)-g(X,Y)]
\\
&\quad +  & \varepsilon\left[ g(\dot{J}X,JY)+g(JX,\dot{J}Y)+\dot{g}(JX,JY)-\dot{g}(X,Y)\right] 
\\
&\quad +  & \mathcal{O}(\varepsilon^2)
\end{eqnarray*}
The first term vanishes because $J$ is an isometry of $g$, and the
second one also vanishes because it is the derivative at $s=0$ of
$g_s(J_sX,J_sY)-g_s(X,Y)=0$ (here we have chosen any actual smooth
variation $f_s$ of $f$ with variational field $uN$, possibly
non-minimal). This proves~\eqref{eq:Jepsilon-compatible}.

\[
g_\varepsilon(X,J_\varepsilon Y)
\stackrel{\eqref{eq:Jepsilon-compatible}}{=}
g_\varepsilon
(J_\varepsilon X,J_\varepsilon^2Y)
+\mathcal O(\varepsilon^2)
\stackrel{\eqref{eqJepsilon2}}{=}
-g_\varepsilon(J_\varepsilon X,Y)
+\mathcal O(\varepsilon^2),
\]
which demonstrates~\eqref{eq:Jepsilon-skew-adjoint}.
\begin{eqnarray*}
\langle \alpha_\varepsilon(X), \alpha_\varepsilon(Y)\rangle -g_\varepsilon(X,Y)
&\stackrel{\eqref{alphaJAepsilon}}{=}&
[\langle \alpha(X), \alpha(Y)\rangle -g(X,Y)]
\\
&\quad +& 
\varepsilon\left[ \langle \dot{\alpha}(X), \alpha(Y)\rangle +\langle \alpha(X),\dot{\alpha}(Y)\rangle -\dot{g}(X,Y)\right] 
+\mathcal{O}(\varepsilon^2).
\end{eqnarray*}
The first term vanishes because left multiplication by $f$ is an isometry of the ambient metric, and the second one also vanishes because it is the derivative at $s=0$ of $\langle \alpha_s(X), \alpha_s(Y)\rangle -g_s(X,Y)=0$. This proves~\eqref{eq:alphaepsilon-isometric}.

To prove~\eqref{eq:formal-alpha}, we recall that the Maurer--Cartan form $\alpha_s$ of $f_s$ satisfies $d\alpha_s+\alpha_s\wedge\alpha_s=0$ on $\Sigma$ (observe that this identity does not involve minimality). Differentiating at $s=0$ gives
\begin{equation}\label{eq:linearized-MC-alpha}
d\dot\alpha
+
\dot\alpha\wedge\alpha
+
\alpha\wedge\dot\alpha
=
0.
\end{equation}
Now,
\begin{align*}
d\alpha_\varepsilon
+\alpha_\varepsilon\wedge\alpha_\varepsilon
&\stackrel{\eqref{alphaJAepsilon}}{=}
d(\alpha+\varepsilon\dot\alpha)+(\alpha+\varepsilon\dot\alpha)\wedge (\alpha+\varepsilon\dot\alpha)
\\
&=
d\alpha+\alpha\wedge\alpha
+\varepsilon \left[
d\dot\alpha
+\dot\alpha\wedge\alpha
+\alpha\wedge\dot\alpha
\right]
+\varepsilon^2\dot\alpha\wedge\dot\alpha
\stackrel{\eqref{alphaMC},\eqref{eq:linearized-MC-alpha}}=
\mathcal{O}(\varepsilon^2),
\end{align*}
which gives~\eqref{eq:formal-alpha}. 

Given any $X\in T\Sigma$,
\begin{align*}
\alpha_\varepsilon\bigl((J_\varepsilon-I)X\bigr)
&\stackrel{\eqref{alphaJAepsilon}}{=}
	(\alpha+\varepsilon\dot\alpha)
	\bigl((J-I)X+\varepsilon\dot J(X)\bigr)\\
	&=
	\alpha\bigl((J-I)X\bigr)
	+\varepsilon
	\left[
	\dot\alpha\bigl((J-I)X\bigr)
	+
	\alpha\bigl(\dot J(X)\bigr)
	\right]
	+ \mathcal{O}(\varepsilon^2).
\end{align*}
Thus,
\[
\alpha_\varepsilon\circ(J_\varepsilon-I)
=
\alpha\circ(J-I)
+
\varepsilon
\left[
\dot\alpha\circ(J-I)+\alpha\circ\dot J
\right]
+ \mathcal{O}(\varepsilon^2).
\]
Since $\phi$ is linear, it follows that
\begin{eqnarray*}
	\phi\circ\alpha_\varepsilon\circ(J_\varepsilon-I)
	&=&
	\phi\circ\alpha\circ(J-I)
	+\varepsilon\,
	\phi\circ
	\left(
	\dot\alpha\circ(J-I)+\alpha\circ\dot J
	\right)
	+ \mathcal{O}(\varepsilon^2) 
	\\
	&\stackrel{\eqref{A},\eqref{eq:Adot-expanded}}{=}&
	\mathcal A+\varepsilon\dot{\mathcal A}
	+ \mathcal{O}(\varepsilon^2)
	\stackrel{\eqref{alphaJAepsilon}}{=}
	\mathcal A_\varepsilon
	+ \mathcal{O}(\varepsilon^2),
\end{eqnarray*}
which demonstrates~\eqref{eq75}.
This completes the proof of the claim.
\par\vspace{.2cm}
We emphasize that no $R_s$ has been introduced. The formal frame
map will be constructed only after its compatibility equation has
been proved from $\mathcal L_fu=0$.

\begin{theorem}\label{propos8.9}
Let $f\colon(\Sigma,g)\to\mathbb S^3$ be a minimal immersion of a simply connected surface, let $f^*$ be a conjugate minimal immersion, and let $R\colon\Sigma\to SO(3)$ be the corresponding frame map:
\begin{equation}\label{eq:conjugation-data-section6}
\alpha^*:=(f^*)^{-1}\cdot df^*=R^{\top}\alpha\circ J,
\qquad
dR=\mathcal AR.
\end{equation}
If $u$ is any Jacobi function for $f$, then the following hold.

\begin{enumerate}
\item The one-form $R^{\top}\dot{\mathcal A}R$ is closed. Therefore, for every $\Omega_0\in\mathfrak{so}(3)$, the initial value problem
\begin{equation}\label{eq:Omega-IVP}
d\Omega=R^{\top}\dot{\mathcal A}R,
\qquad
\Omega(p_0)=\Omega_0,
\end{equation}
has a unique global solution $\Omega\colon\Sigma\to\mathfrak{so}(3)$.

\item Define
\begin{equation}\label{dotalpha*}
\dot\alpha^*
:=-\Omega\alpha^*
+R^{\top}\bigl(\dot\alpha\circ J+\alpha\circ\dot J\bigr).
\end{equation}
Then, $\dot\alpha^*$ satisfies the linearized Maurer--Cartan equation
\begin{equation}\label{eq:linearized-MC-dotalpha*}
d\dot\alpha^*
+\alpha^*\wedge\dot\alpha^*
+\dot\alpha^*\wedge\alpha^*=0.
\end{equation}

\item For every $q_0\in\Im (\mathbb{H})$, the initial value problem
\begin{equation}
	\label{eq:q-linearized-system}
	dq+\alpha^*\cdot q-q\cdot \alpha^*=\dot\alpha^* \quad \mbox{on $\Sigma$,}
	\qquad
	q(p_0)=q_0,
\end{equation}
has a unique global solution $q\colon \Sigma \to \Im (\mathbb{H})$ (here, both $dq$, $\dot\alpha^*$ and $\alpha^*\cdot q-q\cdot \alpha^*$ are one-forms on $\Sigma$ taking values on $\Im (\mathbb{H})$, while both $\alpha^*\cdot q$ and $q\cdot \alpha^*$ take values on $\mathbb{H}$). Thus, the vector field
$X^*:=f^*\cdot q\colon \Sigma \to T\mathbb{S}^3$
is tangent to $\mathbb S^3$ along $f^*$, and its normal component
$u^*:=\langle X^*,N^*\rangle$ is a Jacobi function for $f^*$, i.e., $\mathcal L_{f^*}u^*=0$.
\end{enumerate}
\end{theorem}
\begin{proof}
We first prove the compatibility of the equation for $\Omega$,  i.e., that $R^{\top}\dot{\mathcal A}R$ is closed. We may compute the first variation of $d\mathcal A-\mathcal A\wedge\mathcal A$ using any smooth normal variation with initial velocity $uN$ (nearby immersions $f_s$ need not be minimal), and Lemma~\ref{lema5.3} gives
\[
\left.\frac{d}{ds}\right|_{s=0}
(d\mathcal{A}_s-\mathcal{A}_s\wedge\mathcal{A}_s)	
=
-2\left.\frac{d}{ds}\right|_{s=0}H_s\,\phi(f_s^{-1}\cdot N_s)\,da_s
\]
Since $f$ is minimal, all the terms obtained by
differentiating $\phi(f_s^{-1}\cdot N_s)\,da_s$ are multiplied by $H=0$. Thus,
\[
\left.\frac{d}{ds}\right|_{s=0}
(d\mathcal{A}_s-\mathcal{A}_s\wedge\mathcal{A}_s)	
=
-2\dot{H}\,\phi(f^{-1}\cdot N)\,da
=-(\mathcal{L}_fu)\,\phi(f^{-1}\cdot N)\,da.
\]
The last expression vanishes since $u$ is a Jacobi function for $f$,
and so, 	
\(
d\dot{\mathcal A}
-\dot{\mathcal A}\wedge\mathcal A
-\mathcal A\wedge\dot{\mathcal A}
=
0.
\)
Using $dR=\mathcal AR$, $dR^{\top}=-R^{\top}\mathcal A$, and the
Leibniz rule, we obtain
\begin{align*}
	d\bigl(R^{\top}\dot{\mathcal A}R\bigr)
	&=R^{\top}
	\bigl(
	d\dot{\mathcal A}
	-\mathcal A\wedge\dot{\mathcal A}
	-\dot{\mathcal A}\wedge\mathcal A
	\bigr)R=0.
\end{align*}
Since $\Sigma$ is simply connected,
\eqref{eq:Omega-IVP} has a unique global solution. This proves the
first item of the theorem.

Set
\begin{equation}\label{eq:formal-R-epsilon}
\dot R:=R\Omega,
\qquad
R_\varepsilon:=R+\varepsilon\dot R=R(I+\varepsilon\Omega);
\end{equation}
observe that since $\Omega$ is a smooth map from $\Sigma $ to $\mathfrak{so}(3)$, $\dot{R}$ is a section of the induced bundle $R^*TSO(3)$ over $\Sigma$. $R_{\varepsilon}$ is a smooth map from $\Sigma$ to
$\mathcal M_3(\mathbb R)$, treated here as a formal first jet. We do not
assume that it arises from frame maps associated with a family of
minimal immersions. It is the first jet of the $SO(3)$-valued family
$R\exp(\varepsilon\Omega)$, but this algebraic realization does not
supply a nonlinear conjugate family.

Since $\Omega^{\top}=-\Omega$,
\begin{equation}\label{eq:R-epsilon-orthogonal}
	R_\varepsilon^{\top}R_\varepsilon
	\stackrel{\eqref{eq:formal-R-epsilon}}{=}
	(I+\varepsilon\Omega)^{\top}
	R^{\top}R
	(I+\varepsilon\Omega)=
	(I-\varepsilon\Omega)(I+\varepsilon\Omega)
	=I-\varepsilon ^2\Omega^2
	\quad \mbox{on $\Sigma$.}
\end{equation}
Furthermore,
\begin{align}
dR_\varepsilon
&\stackrel{\eqref{eq:formal-R-epsilon}}{=}
dR+\varepsilon\,d(R\Omega)
\stackrel{\eqref{eq151a}}{=}
\mathcal AR
+
\varepsilon
\bigl(
\mathcal AR\Omega+R\,d\Omega
\bigr) \nonumber
\\
&\stackrel{(\star)}{=}
\mathcal AR
+
\varepsilon
\bigl(
\mathcal AR\Omega+\dot{\mathcal A}R
\bigr)
=
(\mathcal A+\varepsilon\dot{\mathcal A})
(R+\varepsilon R\Omega)-\varepsilon ^2\dot{\mathcal A}R\Omega \nonumber
\\
&\stackrel{\eqref{eq:formal-R-epsilon},\eqref{alphaJAepsilon}}{=}
\mathcal A_\varepsilon R_\varepsilon
+\mathcal{O}(\varepsilon^2), \label{eq:formal-R}
\end{align}
where in $(\star)$ we have used item 1 of this proposition. Using~\eqref{eq:R-epsilon-orthogonal} 
and~\eqref{eq:formal-R}, we have
\begin{equation}\label{eq80}
dR_\varepsilon^{\top}
=
-R_\varepsilon^{\top}\mathcal A_\varepsilon
+\mathcal{O}(\varepsilon^2) 
\quad \mbox{on $\Sigma$.}
\end{equation}

Define
\begin{equation}\label{eq:formal-alpha-star}
\alpha_\varepsilon^*
:=R_\varepsilon^{\top}\alpha_\varepsilon\circ J_\varepsilon.
\end{equation}

Since $R_\varepsilon^{\top}=
(I-\varepsilon\Omega)R^{\top}$ (by~\eqref{eq:formal-R-epsilon}),
expanding \eqref{eq:formal-alpha-star} and using~\eqref{alphaJAepsilon} gives
\begin{align}
	\alpha_\varepsilon^*
	&=
	(I-\varepsilon\Omega)R^{\top}
	(\alpha+\varepsilon\dot\alpha)
	\circ(J+\varepsilon\dot J) \nonumber
	\\
	&=
	R^{\top}\alpha\circ J
	+\varepsilon
	\left[
	-\Omega R^{\top}\alpha\circ J
	+
	R^{\top}
	\bigl(
	\dot\alpha\circ J+\alpha\circ\dot J
	\bigr)
	\right]
	+\mathcal{O}(\varepsilon^2)  \nonumber
	\\
	&\stackrel{\eqref{eq:conjugation-data-section6}}{=}
	\alpha^*+\varepsilon\dot\alpha^*
	+\mathcal{O}(\varepsilon^2),  \label{alphaepsilon*}
\end{align}
where $\dot\alpha^*$ is precisely~\eqref{dotalpha*}.

We next prove that $\alpha_\varepsilon^*$ satisfies the
Maurer--Cartan equation up to second order. Let
\(
\beta_\varepsilon
:=
\alpha_\varepsilon\circ J_\varepsilon,
\)
so that~\eqref{eq:formal-alpha-star} can be rewritten as 
$\alpha_\varepsilon^*=R_\varepsilon^{\top}\beta_\varepsilon$
on $\Sigma$. Applying the Leibniz rule to this equation, we obtain
\begin{eqnarray}
	d\alpha_\varepsilon^*
	&=&
	dR_\varepsilon^{\top}\wedge\beta_\varepsilon
	+
	R_\varepsilon^{\top}d\beta_\varepsilon \nonumber
	\\
	&\stackrel{\eqref{eq80}}{=}&
	-R_\varepsilon^{\top}
	\mathcal A_\varepsilon\wedge\beta_\varepsilon
	+
	R_\varepsilon^{\top}d\beta_\varepsilon
	+\mathcal{O}(\varepsilon^2) 
	=
	R_\varepsilon^{\top}
	\left(
	d\beta_\varepsilon
	-
	\mathcal A_\varepsilon\wedge\beta_\varepsilon
	\right)
	+\mathcal{O}(\varepsilon^2). \label{dalphaepsilon*}
\end{eqnarray}

Here, $\mathcal A_\varepsilon\wedge\beta_\varepsilon$ denotes the following $\Im (\mathbb{H})$-valued two-form on $\Sigma$:
\begin{equation*}
	(\mathcal A_\varepsilon\wedge\beta_\varepsilon)(X,Y)
	=
	\mathcal A_\varepsilon(X)\beta_\varepsilon(Y)
	-
	\mathcal A_\varepsilon(Y)\beta_\varepsilon(X), 
\qquad X,Y\in T\Sigma.
\end{equation*}
Next we use two identities valid for every oriented immersion $h\colon\Sigma\to\mathbb S^3$. If
$\alpha_h=h^{-1}dh$, $J_h$ is its complex structure, and $H_h$ is its mean curvature function, then, for all $X,Y\in T\Sigma$,
\begin{align}
d(\alpha_h\circ J_h)(X,Y)
&=-2H_h\,\alpha_h(X)\times\alpha_h(Y),
\label{eq:dbeta-general}\\
\bigl(
[\phi\circ\alpha_h\circ(J_h-I)]\wedge(\alpha_h\circ J_h)
\bigr)(X,Y)
&=2\,\alpha_h(X)\times\alpha_h(Y).
\label{eq:A-beta-general}
\end{align}

Equation~\eqref{eq:dbeta-general} is precisely
\eqref{dalphaJ}, applied to the immersion \(h\).
Equation~\eqref{eq:A-beta-general} is algebraic; to check it, choose a
positively oriented \(g_h\)-orthonormal basis $\{e_1,e_2=J_he_1\}$, and set
$a:=\alpha_h(e_1)$, $b:=\alpha_h(e_2)$. Then, $(\alpha_h\circ J_h)(e_1)=b$, $(\alpha_h\circ J_h)(e_2)=-a$,
and
\[
\bigl[\phi\circ\alpha_h\circ(J_h-I)\bigr](e_1)
=\phi(b-a),
\qquad 
\bigl[\phi\circ\alpha_h\circ(J_h-I)\bigr](e_2)
=\phi(-a-b).
\]
Therefore,
\begin{eqnarray*}
\bigl(	[\phi\circ\alpha_h\circ(J_h-I)]
\wedge(\alpha_h\circ J_h)
\bigr)(e_1,e_2)
&=&
\phi(b-a)(-a)-\phi(-a-b)b
\\
&=&
(b-a)\times(-a)-(-a-b)\times b
\\
&=&
2a\times b,
\end{eqnarray*}
which proves~\eqref{eq:A-beta-general}.

We now apply~\eqref{eq:dbeta-general} to the actual variation
$f_s$. Writing
\(
\beta_s:=\alpha_s\circ J_s,
\)
we have, for every~$s$, $d\beta_s(X,Y)
=
-2H_s\,\alpha_s(X)\times\alpha_s(Y)$.
Taking first jets at $s=0$, and using the notation introduced
in~\eqref{alphaJAepsilon}, gives
\[
d\beta_\varepsilon(X,Y)
=
-2H_\varepsilon\,
\alpha_\varepsilon(X)\times\alpha_\varepsilon(Y)
+\mathcal O(\varepsilon^2).
\]
Since $H_\varepsilon
=
H+\varepsilon\dot H
=
0$, we conclude that
\begin{equation}\label{eq:d-beta-formal}
d\beta_\varepsilon(X,Y)
=
\mathcal O(\varepsilon^2).
\end{equation}

Arguing similarly with~\eqref{eq:A-beta-general} we obtain
\begin{equation}
(\mathcal A_\varepsilon\wedge\beta_\varepsilon)(X,Y)
=2\, \alpha_\varepsilon(X)\times\alpha_\varepsilon(Y)+\mathcal{O}(\varepsilon^2).
\label{eq:A-beta-formal}
\end{equation}
Now,
\begin{eqnarray}
(\alpha_\varepsilon^*\wedge\alpha_\varepsilon^*)(X,Y)
&=&
\alpha_\varepsilon^*(X)\cdot\alpha_\varepsilon^*(Y)
-
\alpha_\varepsilon^*(Y)\cdot\alpha_\varepsilon^*(X)\nonumber
\\
&\stackrel{(Q6)}{=}&
2\alpha_\varepsilon^*(X)\times
\alpha_\varepsilon^*(Y). \label{eq:alpha-star-wedge-cross}
\end{eqnarray}
We next show that $R_\varepsilon^{\top}$ preserves cross products to first order. 
For $\Omega\in\mathfrak{so}(3)$ and $p,q\in\mathbb R^3$, we have
\begin{equation}\label{eq:Omega-cross-derivation}
 \Omega(p\times q)=(\Omega p)\times q+p\times(\Omega q).
\end{equation}
Indeed, writing $\Omega x=\omega\times x$, both sides equal $p\langle\omega,q\rangle-q\langle\omega,p\rangle$. By bilinearity and~\eqref{eq:Omega-cross-derivation},
\[
 [(I-\varepsilon\Omega)p]\times[(I-\varepsilon\Omega)q]
 =(I-\varepsilon\Omega)(p\times q)+\varepsilon^2(\Omega p\times\Omega q).
\]
Since $R\in SO(3)$ and $R_\varepsilon^\top=(I-\varepsilon\Omega)R^\top$, it follows that $R_\varepsilon^\top p\times R_\varepsilon^\top q=R_\varepsilon^\top(p\times q)+\mathcal O(\varepsilon^2)$. Applying this identity to the definition~\eqref{eq:formal-alpha-star} gives
\begin{equation}\label{eq:alpha-star-cross-second}
 \alpha_\varepsilon^*(X)\times\alpha_\varepsilon^*(Y)
=R_\varepsilon^\top\bigl[\alpha_\varepsilon(J_\varepsilon X)
       \times\alpha_\varepsilon(J_\varepsilon Y)\bigr]\\
+\mathcal O(\varepsilon^2).
\end{equation}

We next compare $\alpha_\varepsilon(J_\varepsilon X)\times
\alpha_\varepsilon(J_\varepsilon Y)$ (which appears in~\eqref{eq:alpha-star-cross-second}) with
$\alpha_\varepsilon(X)\times\alpha_\varepsilon(Y)$. Taking determinants in~\eqref{eqJepsilon2} and using $\det J=1$ gives
\begin{equation}\label{eq:det-J-epsilon}
\det J_\varepsilon=1+\mathcal{O}(\varepsilon^2).
\end{equation}
Since $J_\varepsilon\colon T_p\Sigma\to T_p\Sigma$ is linear, we get
\begin{equation}\label{eq:J-preserves-alpha-cross}
\alpha_\varepsilon(J_\varepsilon X)
\times
\alpha_\varepsilon(J_\varepsilon Y)
=
(\det J_\varepsilon)\,	\alpha_\varepsilon(X)\times\alpha_\varepsilon(Y)
\stackrel{\eqref{eq:det-J-epsilon}}{=}
\alpha_\varepsilon(X)\times\alpha_\varepsilon(Y)	+\mathcal{O}(\varepsilon^2).
\end{equation}

Putting together \eqref{eq:alpha-star-wedge-cross},
\eqref{eq:alpha-star-cross-second} and~\eqref{eq:J-preserves-alpha-cross}, we get
\begin{equation}\label{eq:alpha-star-wedge}
\begin{split}
(\alpha_\varepsilon^*\wedge\alpha_\varepsilon^*)(X,Y)
&=
2R_\varepsilon^{\top}
\bigl(
\alpha_\varepsilon(X)\times\alpha_\varepsilon(Y)
\bigr)
+\mathcal{O}(\varepsilon^2).
\end{split}
\end{equation}

We can now verify the Maurer--Cartan equation for
$\alpha_\varepsilon^*$ up to second order:
\begin{align*}
d\alpha_\varepsilon^*(X,Y)
&\stackrel{\eqref{dalphaepsilon*}}{=}
R_\varepsilon^{\top}
\left(
d\beta_\varepsilon(X,Y)\right)
-
R_\varepsilon^{\top}
\left((\mathcal A_\varepsilon\wedge\beta_\varepsilon)(X,Y)\right)
+\mathcal{O}(\varepsilon^2)
\\
&\stackrel{\eqref{eq:d-beta-formal},\eqref{eq:A-beta-formal}}{=}
-2R_\varepsilon^{\top}
\bigl(
\alpha_\varepsilon(X)\times\alpha_\varepsilon(Y)
\bigr)
+\mathcal{O}(\varepsilon^2).
\end{align*}
Adding \eqref{eq:alpha-star-wedge}, we conclude that
\begin{equation}\label{eq:formal-MC-alpha-star}
d\alpha_\varepsilon^*
+
\alpha_\varepsilon^*\wedge\alpha_\varepsilon^*
=
\mathcal{O}(\varepsilon^2).
\end{equation}

With this Maurer--Cartan equation for $\alpha_\varepsilon^*$ up to order two, we are now ready to prove \eqref{eq:linearized-MC-dotalpha*}.
Expanding the left-hand side of
\eqref{eq:formal-MC-alpha-star} gives
\begin{eqnarray*}
d\alpha_\varepsilon^*
+
\alpha_\varepsilon^*\wedge\alpha_\varepsilon^*
&\stackrel{\eqref{alphaepsilon*}}{=}&
d(\alpha^*+\varepsilon\dot\alpha^*)+(\alpha^*+\varepsilon\dot\alpha^*)
\wedge (\alpha^*+\varepsilon\dot\alpha^*)+\mathcal{O}(\varepsilon^2)
\\
	&=&
	d\alpha^*+\alpha^*\wedge\alpha^*	+\varepsilon
	\left(
	d\dot\alpha^*
	+\alpha^*\wedge\dot\alpha^*
	+\dot\alpha^*\wedge\alpha^*
	\right)
	+\mathcal{O}(\varepsilon^2).
\end{eqnarray*}
The zeroth-order term of the last expression vanishes because $\alpha^*$ is the Maurer--Cartan form of $f^*$. \eqref{eq:formal-MC-alpha-star} implies that the coefficient of $\varepsilon$ must vanish, which proves
\eqref{eq:linearized-MC-dotalpha*} and completes the proof of item~2.

We next prove item~3. We start by analyzing the left-hand-side of the PDE in~\eqref{eq:q-linearized-system}: consider the spaces 
$\Omega^0(\Sigma,\Im (\mathbb{H}))=C^{\infty}(\Sigma,\Im (\mathbb{H}))$
and $\Omega^1(\Sigma,\Im (\mathbb{H}))$ of smooth 1-forms on $\Sigma$ taking values on $\Im (\mathbb{H})$ and define the operator
\begin{equation}\label{Dalpha*}
D_{\alpha^*}\colon
\Omega^0(\Sigma,\Im (\mathbb{H}))
\to
\Omega^1(\Sigma,\Im (\mathbb{H})),
\qquad
D_{\alpha^*}q=dq+\alpha^*\cdot q-q\cdot\alpha^*.
\end{equation}
This means that, for every $p\in\Sigma$ and $X\in T_p\Sigma$,
\[
(D_{\alpha^*}q)_p(X)
=
dq_p(X)
+\alpha_p^*(X)\cdot q(p)
-q(p)\cdot\alpha_p^*(X)=X(q)+[\alpha^*(X),q].
\]
(Recall that the commutator of two imaginary quaternions is imaginary).

For each smooth vector field $X$ on $\Sigma$, evaluation of $D_{\alpha^*}$ in the $X$-direction gives a degree-zero operator on sections:
\begin{equation}\label{Dalpha*X}
D_{\alpha^*,X}\colon \Omega^0(\Sigma,\Im (\mathbb{H}))
\to
\Omega^0(\Sigma,\Im (\mathbb{H})),
\qquad
D_{\alpha^*,X}q
:=X(q)+[\alpha^*(X),q].
\end{equation}
$D_{\alpha^*,X}$ satisfies the Leibniz rule $D_{\alpha^*,X}(hq)
=
X(h)q+hD_{\alpha^*,X}q$,  for all $h\in C^\infty(\Sigma,\mathbb R)$.

The curvature of $D_{\alpha^*}$ acts on a section $q$ of the trivial bundle $\Sigma \times \Im (\mathbb{H})$ by
\begin{equation}\label{curvDalpha*}
R^{D_{\alpha^*}}(X,Y)q
:=
D_{\alpha^*,X}D_{\alpha^*,Y}q-
D_{\alpha^*,Y}D_{\alpha^*,X}q-
D_{\alpha^*,[X,Y]}q.
\end{equation}
\par\vspace{.2cm}\noindent
{\bf Claim 1.} {\em $D_{\alpha^*}$ is flat, i.e., $R^{D_{\alpha^*}}=0$.}
\par\vspace{.2cm}\noindent
{\em Proof of Claim 1.}
Using~\eqref{Dalpha*X} twice, we have
\[
\begin{aligned}
D_{\alpha^*,X}D_{\alpha^*,Y}q
&=X(Y(q))+[X(\alpha^*(Y)),q]+[\alpha^*(Y),X(q)]\\
&\qquad+[\alpha^*(X),Y(q)]+[\alpha^*(X),[\alpha^*(Y),q]].
\end{aligned}
\]
Analogously,
\[
\begin{aligned}
D_{\alpha^*,Y}D_{\alpha^*,X}q
&=Y(X(q))+[Y(\alpha^*(X)),q]+[\alpha^*(X),Y(q)]\\
&\qquad+[\alpha^*(Y),X(q)]+[\alpha^*(Y),[\alpha^*(X),q]].
\end{aligned}
\]
On subtracting the last two expressions, the terms $[\alpha^*(Y),X(q)],
[\alpha^*(X),Y(q)]$ cancel.
The Jacobi identity for the commutator yields
\[
[\alpha^*(X),[\alpha^*(Y),q]]-[\alpha^*(Y),[\alpha^*(X),q]]
=[[\alpha^*(X),\alpha^*(Y)],q].
\]
Thus,
\begin{eqnarray*}
D_{\alpha^*,X}D_{\alpha^*,Y}q
-
D_{\alpha^*,Y}D_{\alpha^*,X}q
&=&[X,Y](q)
\\
& & \quad +[X(\alpha^*(Y)),q]-[Y(\alpha^*(X)),q]
\\
& & 
\qquad +[[\alpha^*(X),\alpha^*(Y)],q].
\end{eqnarray*}

Moreover,
\[
D_{\alpha^*,[X,Y]}q
=[X,Y](q)+[\alpha^*([X,Y]),q].
\]
Consequently,
\begin{eqnarray*}
R^{D_{\alpha^*}}(X,Y)q
&=&
\bigl[
X\bigl(\alpha^*(Y)\bigr)
-Y\bigl(\alpha^*(X)\bigr)
+[\alpha^*(X),\alpha^*(Y)]
-\alpha^*([X,Y]),\ q
\bigr]
\\
&=&
[d\alpha^*(X,Y)+[\alpha^*(X),\alpha^*(Y)],\ q]
\\
&=&
[d\alpha^*(X,Y)+\alpha^*(X)\cdot \alpha^*(Y)-\alpha^*(Y)\cdot \alpha^*(X),\ q]
\\
&=&
\bigl[
(d\alpha^*+\alpha^*\wedge\alpha^*)(X,Y),q
\bigr].
\end{eqnarray*}

Since $\alpha^*=(f^*)^{-1}df^*$ satisfies the Maurer--Cartan
equation $d\alpha^*+\alpha^*\wedge\alpha^*=0$, we deduce that 
$R^{D_{\alpha^*}}=0$, and Claim 1 is proved.

Recall that the PDE that we want to solve in~\eqref{eq:q-linearized-system} is 
\begin{equation}\label{PDEq}
D_{\alpha^*}q=\dot\alpha^* \quad \mbox{on $\Sigma$.}
\end{equation}
We next deduce the compatibility condition for~\eqref{PDEq}. To do so,
first we extend $D_{\alpha^*}$ to $\Im (\mathbb{H})$-valued one-forms by the usual procedure:
\[
(D_{\alpha^*}\eta)(X,Y)
=
D_{\alpha^*,X}\bigl(\eta(Y)\bigr)
-
D_{\alpha^*,Y}\bigl(\eta(X)\bigr)
-
\eta([X,Y]), 
\qquad \eta \in \Omega^1(\Sigma,\Im (\mathbb{H})).
\]
Now take $\eta=D_{\alpha^*}q$. Since $
(D_{\alpha^*}q)(Y)=D_{\alpha^*,Y}q$,
we have
\[
(D_{\alpha^*}^{\,2}q)(X,Y)
=
D_{\alpha^*,X}D_{\alpha^*,Y}q
-
D_{\alpha^*,Y}D_{\alpha^*,X}q
-
D_{\alpha^*,[X,Y]}q\\
\stackrel{\eqref{curvDalpha*}}{=}
R^{D_{\alpha^*}}(X,Y)q.
\]
As $D_{\alpha^*}$ is flat, the last expression vanishes. This implies
that the compatibility condition for~\eqref{PDEq} is precisely 
\begin{equation}\label{compatibilityq}
D_{\alpha^*}\dot\alpha^*=0.
\end{equation}
\par\vspace{.2cm}\noindent
{\bf Claim 2.} {\em The last equation is just~\eqref{eq:linearized-MC-dotalpha*}. }
\par\vspace{.2cm}\noindent
{\em Proof of Claim 2.}
Given $\eta\in\Omega^1(\Sigma,\Im (\mathbb{H}))$ and smooth vector fields $X,Y$ on $\Sigma$, \eqref{Dalpha*X} implies
\[
D_{\alpha^*,X}\bigl(\eta(Y)\bigr)
=
X\bigl(\eta(Y)\bigr)
+\alpha^*(X)\cdot\eta(Y)
-\eta(Y)\cdot\alpha^*(X).
\]
Interchanging $X$ and $Y$ and subtracting $\eta([X,Y])$, we obtain
\begin{eqnarray*}
(D_{\alpha^*}\eta)(X,Y)
&=&
D_{\alpha^*,X}\bigl(\eta(Y)\bigr)
-D_{\alpha^*,Y}\bigl(\eta(X)\bigr)
-\eta([X,Y])
\\
&=&
X\bigl(\eta(Y)\bigr)-Y\bigl(\eta(X)\bigr)-\eta([X,Y])
\\
& &\quad  +\alpha^*(X)\cdot\eta(Y)-\alpha^*(Y)\cdot\eta(X)
\\
& & 
\qquad +\eta(X)\cdot\alpha^*(Y)-\eta(Y)\cdot\alpha^*(X)
\\
&=&
d\eta(X,Y)+(\alpha^*\wedge\eta)(X,Y)+(\eta\wedge\alpha^*)(X,Y)
\\
&=&
\bigl(d\eta+\alpha^*\wedge\eta+\eta\wedge\alpha^*\bigr)(X,Y),
\end{eqnarray*}
or equivalently,
\begin{equation*}
D_{\alpha^*}\eta
=
d\eta+\alpha^*\wedge\eta+\eta\wedge\alpha^*.
\end{equation*}
From here we deduce that \eqref{compatibilityq} is precisely~\eqref{eq:linearized-MC-dotalpha*}, which proves Claim 2.

Since $\Sigma$ is simply connected, the Frobenius theorem applied to the PDE system~\eqref{PDEq} gives a
unique global solution of \eqref{eq:q-linearized-system} for every
prescribed initial value $q(p_0)=q_0\in\Im (\mathbb{H})$.

Equivalently, global existence can be seen directly by trivializing
the connection. Define the imaginary-quaternion-valued one-form
$ \zeta:=f^*\cdot\dot\alpha^*\cdot(f^*)^{-1}$.
The product rule and~\eqref{eq:linearized-MC-dotalpha*} give $d\zeta=0$ and $d\bigl(f^*\cdot q\cdot(f^*)^{-1}\bigr)
 =f^*\cdot D_{\alpha^*}q\cdot(f^*)^{-1}$. Thus the unique solution is
\begin{equation*}
 q(p)=(f^*(p))^{-1}\cdot
 \left[
 f^*(p_0)\cdot q_0\cdot(f^*(p_0))^{-1}
 +\int_{p_0}^{p}\zeta
 \right]\cdot f^*(p).
\end{equation*}
The integral is independent of the path because $\Sigma$ is simply
connected. This also shows that no completeness or growth hypothesis
is required for the global solution.

Since both $\alpha^*$ and $\dot\alpha^*$ are imaginary-valued, taking real parts in \eqref{eq:q-linearized-system} gives
\[
0=\Re (\dot\alpha^*)
=\Re(dq+\alpha^*\cdot q-q\cdot \alpha^*)
=\Re(dq)=d(\Re(q)),
\]
hence $\Re(q)$ is constant. Since $q_0\in \Im (\mathbb{H})$, it follows that
$q\in C^{\infty}(\Sigma, \Im (\mathbb{H}))$ and hence, defining $X^*:=f^*\cdot q $, 
we produce a smooth tangent vector field to $\mathbb S^3$ along $f^*$.

To finish the proof, it remains to prove that
$u^*:=\langle X^*,N^*\rangle$ is Jacobi. To do this, we introduce the smooth, first-order deformation of $f^*$
\begin{equation}\label{fvarepsilon*}
f_\varepsilon^*:=f^*\cdot \left(\cos\bigl(\varepsilon |q(p)|\bigr)\,{\bf e}
+
\frac{\sin\bigl(\varepsilon |q(p)|\bigr)}{|q(p)|}\,q(p)\right)\colon \Sigma \to \mathbb{S}^3.
\end{equation}
At $q(p)=0$, the factor $\sin(\varepsilon|q(p)|)/|q(p)|$ is interpreted by its smooth extension, with value $\varepsilon$; equivalently, 
$f^*_\varepsilon=f^*\exp(\varepsilon q)$.
Clearly, the variational field of $\varepsilon \mapsto f^*_{\varepsilon}$ at $\varepsilon=0$ is $X^*=f^*\cdot q$. 

If $\Sigma$ is compact, the fact that $f^*$ is an immersion ensures that we can choose $\varepsilon_0>0$ so that for $|\varepsilon|<
\varepsilon_0$, $f_{\varepsilon}^*$ is an immersion. If $\Sigma$ is noncompact, the existence of such a global $\varepsilon_0$ is not guaranteed. Nevertheless, to conclude that $u^*$ is Jacobi, we do not need a global $\varepsilon_0$: around every point $p\in\Sigma$ we can choose a neighborhood $U$ and an $\varepsilon_0(U)>0$ such that $f_\varepsilon^*|_U$ is an immersion whenever $|\varepsilon|<\varepsilon_0(U)$. Since the Jacobi equation is local, this suffices for the argument below. Thus, in what follows, we assume
that $f_\varepsilon^*$ is an immersion of $\Sigma$ into $\mathbb{S}^3$.
Denote its left Maurer--Cartan form by
\begin{equation}\label{MCfepsilon*}
\theta_\varepsilon:=(f_\varepsilon^*)^{-1}\cdot df_\varepsilon^*.
\end{equation}
We next expand $\theta_{\varepsilon}$ up to order two around $\varepsilon=0$. To do this, denote by $h_{\varepsilon}$ the factor that
accompanies $f^*$ in equation~\eqref{fvarepsilon*}. Hence,
\[
h_\varepsilon
={\bf e}+\varepsilon q+\mathcal O(\varepsilon^2),
\qquad 
h_\varepsilon^{-1}
={\bf e}-\varepsilon q+\mathcal O(\varepsilon^2),
\qquad 
dh_\varepsilon
=\varepsilon\,dq+\mathcal O(\varepsilon^2).
\]
Since
$f_\varepsilon^*=f^*h_\varepsilon$, the product rule gives
$d(f^*h_\varepsilon)
=df^*\cdot h_\varepsilon+f^*\cdot dh_\varepsilon$, and 
\begin{eqnarray}
\theta_\varepsilon
&=&
h_\varepsilon^{-1}\cdot (f^*)^{-1}\cdot 
\bigl(df^*\cdot h_\varepsilon+f^*\cdot dh_\varepsilon\bigr)
=h_\varepsilon^{-1}\cdot\alpha^*\cdot h_\varepsilon
+h_\varepsilon^{-1}\cdot dh_\varepsilon \nonumber
\\
&=&
({\bf e}-\varepsilon q)\cdot \alpha^*\cdot ({\bf e}+\varepsilon q)
+({\bf e}-\varepsilon q)\cdot \varepsilon\,dq 
+\mathcal O(\varepsilon^2)   \nonumber
\\
&=&
\alpha^*
+\varepsilon\bigl(\alpha^*\cdot q-q\cdot\alpha^*+dq\bigr)
+\mathcal O(\varepsilon^2)  \nonumber
\\
&\stackrel{\eqref{Dalpha*}}{=}&
\alpha^*
+\varepsilon\, D_{\alpha^*}q
+\mathcal O(\varepsilon^2)
\stackrel{\eqref{PDEq}}{=}
\alpha^*
+\varepsilon\, \dot\alpha^*
+\mathcal O(\varepsilon^2)  \nonumber
\\
&\stackrel{\eqref{eq:formal-alpha-star},\eqref{alphaepsilon*}}{=}&
R_\varepsilon^{\top}\alpha_\varepsilon\circ J_\varepsilon
+\mathcal O(\varepsilon^2). \label{thetavarepsilon}
\end{eqnarray}
In particular,
\begin{eqnarray*}
\langle\theta_\varepsilon(X),\theta_\varepsilon(Y)\rangle
&=&
\langle
R_\varepsilon^{\top}(\alpha_\varepsilon\circ J_\varepsilon)(X),
R_\varepsilon^{\top}(\alpha_\varepsilon\circ J_\varepsilon)(Y) 
\rangle
+\mathcal O(\varepsilon^2)
\\
&=&
\langle
R_\varepsilon R_\varepsilon^{\top}(\alpha_\varepsilon\circ J_\varepsilon)(X),
(\alpha_\varepsilon\circ J_\varepsilon)(Y) 
\rangle
+\mathcal O(\varepsilon^2)
\\
&\stackrel{\eqref{eq:R-epsilon-orthogonal}}{=}&
\langle
(\alpha_\varepsilon\circ J_\varepsilon)(X),
(\alpha_\varepsilon\circ J_\varepsilon )(Y)
\rangle
+\mathcal O(\varepsilon^2)
\\
&\stackrel{\eqref{eq:Jepsilon-compatible},\eqref{eq:alphaepsilon-isometric}}{=}&
g_\varepsilon (X,Y)
+\mathcal O(\varepsilon^2).
\end{eqnarray*}

We now compare the first-order geometry of the (non-minimal) normal variation $f_s$ with the geometry of the immersion $f_\varepsilon^*$ given by~\eqref{fvarepsilon*}. Throughout
this comparison, the subscript $\varepsilon$ without an
asterisk denotes a first jet of the variation $f_s$, whereas
$f_\varepsilon^*$ and its geometric objects are actual
immersions and objects associated with them.

Recall that $N_s$ is the unit normal of $f_s$ and that
$N_1=f^{-1}\cdot N$. Define the left-trivialized unit normal
of $f_s$ and its first jet by
\begin{equation}\label{N1sN1varepsilon}
N_{1,s}:=f_s^{-1}\cdot N_s,
\qquad
N_{1,\varepsilon}
:=N_1+\varepsilon\dot N_1,
\qquad
\mbox{where }\ \dot N_1
:=\left.\frac{d}{ds}\right|_{s=0}N_{1,s}.
\end{equation}
In particular, $N_{1,\varepsilon}$ is not asserted to be an
exact unit normal for any immersion. Differentiating at $s=0$
the identity $|N_{1,s}|^2=1$ gives $\langle N_1,\dot{N}_1\rangle=0$. Thus,  
\begin{equation}\label{eq:N1-first-jet}
|N_{1,\varepsilon}|^2=1+2\varepsilon \langle N_1,\dot{N}_1\rangle
+\mathcal{O}(\varepsilon^2)=1+\mathcal O(\varepsilon^2), 
\end{equation}
Analogously, differentiating  $\langle N_{1,s},\alpha_s(X)\rangle=0$ 
at $s=0$ (here $X\in T\Sigma$ is arbitrary) we have $
\langle \dot{N}_1,\alpha(X)\rangle + \langle N_1,\dot{\alpha}(X)\rangle=0$. Thus,  
\begin{eqnarray}
\langle N_{1,\varepsilon},\alpha_\varepsilon(X)\rangle
&\stackrel{\eqref{alphaJAepsilon}}{=}&
\langle N_1+\varepsilon\dot N_1,(\alpha+\varepsilon\dot\alpha)(X)\rangle \nonumber
\\
&=&
\langle N_1,\alpha(X)\rangle
+\varepsilon \bigl(
\langle N_1,\dot\alpha(X)\rangle
+ \langle \dot N_1,\alpha(X)\rangle 
\bigr)
+ \mathcal O(\varepsilon^2)
=\mathcal O(\varepsilon^2).\qquad \label{eq:N1-first-jet1}
\end{eqnarray}

Let $N_\varepsilon^*$ be the unit normal of  the immersion $f_\varepsilon^*$, chosen with the orientation for which
$N_0^*=N^*$, and set
\(
N_{1,\varepsilon}^*
:=(f_\varepsilon^*)^{-1}\cdot N_\varepsilon^*;
\)
this is a unit normal to the plane $\theta_\varepsilon(T\Sigma)$ in
$\Im (\mathbb{H})$. 
\par\vspace{.2cm}\noindent
{\bf Claim 3.} {\em In the situation above, it holds (compare to~\eqref{eq152}):
	\begin{equation}\label{eq:N1-star-comparison}
		N_{1,\varepsilon}^*
		=
		R_{\varepsilon}^{\top}N_{1,\varepsilon}
		+\mathcal{O}(\varepsilon^2).  
	\end{equation}
}
\par\vspace{.2cm}\noindent
{\em Proof of Claim 3.}
By~\eqref{eq:R-epsilon-orthogonal}, $R_\varepsilon^{\top}R_\varepsilon
=I+\mathcal O(\varepsilon^2)$. This property and~\eqref{eq:N1-first-jet} imply that 
\begin{equation}\label{RepsilonTN1unitary}
|R_\varepsilon^{\top}N_{1,\varepsilon}|=1+\mathcal O(\varepsilon^2).
\end{equation}
Furthermore,
\begin{eqnarray*}
\bigl\langle
R_\varepsilon^{\top}N_{1,\varepsilon},
\theta_\varepsilon(X)
\bigr\rangle
&\stackrel{\eqref{thetavarepsilon}}{=}&
\bigl\langle
R_\varepsilon^{\top}N_{1,\varepsilon},
R_\varepsilon^{\top}
(\alpha_\varepsilon\circ J_\varepsilon)(X)
\bigr\rangle
+\mathcal O(\varepsilon^2)
\\
&\stackrel{\eqref{eq:R-epsilon-orthogonal}}{=}&
\bigl\langle
N_{1,\varepsilon},
(\alpha_\varepsilon\circ J_\varepsilon)(X)
\bigr\rangle
+\mathcal O(\varepsilon^2)
\stackrel{\eqref{eq:N1-first-jet1}}{=}
\mathcal O(\varepsilon^2).
\end{eqnarray*}
By definition, $N_{1,\varepsilon}^*$ is unit and orthogonal to the plane $\theta_\varepsilon(T\Sigma)$ inside $\Im (\mathbb{H})$. \eqref{RepsilonTN1unitary} and the last displayed equality show that the same properties hold for $R_\varepsilon^{\top}N_{1,\varepsilon}$ up to second order. Since the codimension of  $\theta_\varepsilon(T\Sigma)$ in $\Im (\mathbb{H})$ is one, we deduce that $N_{1,\varepsilon}^*$ and $R_\varepsilon^{\top}N_{1,\varepsilon}$ coincide up to second order and up to sign. In fact, the sign relating
$N_{1,\varepsilon}^*,R_\varepsilon^{\top}N_{1,\varepsilon}$ must be positive since at $\varepsilon=0$ we have the equality $N_1^*=R^{\top}N_1$ by~\eqref{eq152}. This proves
\eqref{eq:N1-star-comparison} and completes the proof of Claim 3.

Write $A_s$ for the shape operator of $f_s$, with
respect to $N_s$, and put
\[
A_\varepsilon:=A+\varepsilon\dot A,
\qquad
A:=A_0,
\qquad
\dot A:=\left.\frac{d}{ds}\right|_{s=0}A_s.
\]
Thus $A_\varepsilon$, like $N_{1,\varepsilon}$,
is a first jet of the original variation. By contrast,
let $A_\varepsilon^*$ denote the actual shape
operator of $f_\varepsilon^*$, with respect to
$N_\varepsilon^*$.
\par\vspace{.2cm}\noindent
{\bf Claim 4.} {\em 
The left-trivialized Weingarten equations for the immersions $f_s$ and $f_{\varepsilon}^*$ are respectively:
\begin{eqnarray}
dN_{1,s}(X)+\alpha_s(X)\times N_{1,s}
&=&
-\alpha_s(A_sX), \label{eq:Weingarten-fs}
\\
dN_{1,\varepsilon}^*(X)
+\theta_\varepsilon(X)\times N_{1,\varepsilon}^*
&=&
-\theta_\varepsilon(A_\varepsilon^*X).  \label{eq:Weingarten-actual-star}
\end{eqnarray}

}
\par\vspace{.2cm}\noindent
{\em Proof of Claim 4.} We will only prove~\eqref{eq:Weingarten-fs}, as the proof of~\eqref{eq:Weingarten-actual-star} is completely analogous.
Fix $s$ and $X\in T_p\Sigma$. Recall that $\overline{\nabla}$ denotes the standard connection of
$\mathbb R^4\equiv \mathbb H$. 
\begin{equation}\label{eq:left-trivialized-normal-derivative}
f_s^{-1}\cdot \overline{\nabla}_{df_s(X)}N_s
\stackrel{\eqref{N1sN1varepsilon}}{=}
f_s^{-1}\cdot \overline{\nabla}_{df_s(X)}(f_s\cdot N_{1,s})
=
\alpha_s(X)\cdot N_{1,s}+dN_{1,s}(X).
\end{equation}
Differentiating $\langle N_s,f_s\rangle=0$ in the direction of $df_s(X)$ gives 
\[
\langle \overline{\nabla}_{df_s(X)}N_s,f_s\rangle
=-\langle N_s,df_s(X)\rangle=0,
\]
so $\overline{\nabla}_{df_s(X)}N_s$ is already tangent to $\mathbb S^3$. Thus, the Weingarten equation of $f_s$ in $\mathbb{S}^3$ yields
\begin{equation}\label{eq110}
f_s^{-1}\cdot \overline{\nabla}_{df_s(X)}N_s
=
-f_s^{-1}\cdot df_s(A_sX)
=
-\alpha_s(A_sX).
\end{equation}
Since $\alpha_s(X)$ and $N_{1,s}$ are orthogonal
imaginary quaternions, (Q2) gives $\alpha_s(X)\cdot N_{1,s}=
\alpha_s(X)\times N_{1,s}$. Substituting this and~\eqref{eq110} into
\eqref{eq:left-trivialized-normal-derivative} yields~\eqref{eq:Weingarten-fs}. This proves Claim 4.
\par\vspace{.2cm}
Taking the first-order expansion of~\eqref{eq:Weingarten-fs} gives
\begin{equation}\label{eq:Weingarten-first-jet}
	dN_{1,\varepsilon}(X)
	+\alpha_\varepsilon(X)\times N_{1,\varepsilon}
	=
	-\alpha_\varepsilon(A_\varepsilon X)
	+\mathcal O(\varepsilon^2).
\end{equation}

We next compute the left-hand side of
\eqref{eq:Weingarten-actual-star}:
\begin{eqnarray}
dN_{1,\varepsilon}^*(X)
+\theta_\varepsilon(X)\times N_{1,\varepsilon}^*
&\stackrel{\eqref{eq:N1-star-comparison}}{=}&
d(R_{\varepsilon}^{\top}N_{1,\varepsilon})(X)
+\theta_\varepsilon(X)\times N_{1,\varepsilon}^*
+\mathcal{O}(\varepsilon^2) \nonumber
\\
&=&
dR_{\varepsilon}^{\top}(X) N_{1,\varepsilon}
+R_{\varepsilon}^{\top}dN_{1,\varepsilon}(X)
+\theta_\varepsilon(X)\times N_{1,\varepsilon}^*
+\mathcal{O}(\varepsilon^2)  \nonumber
\\
&\stackrel{\eqref{eq:formal-R},\eqref{thetavarepsilon}}{=}&
R_\varepsilon^{\top}
		\Bigl(
		-\mathcal A_\varepsilon(X)N_{1,\varepsilon}
		+dN_{1,\varepsilon}(X)
		+\alpha_\varepsilon(J_\varepsilon X)
		\times N_{1,\varepsilon}
		\Bigr) \nonumber
\\
& & \qquad +\mathcal O(\varepsilon^2) \label{eq112}
\end{eqnarray}
Observe that
\begin{eqnarray*}
\mathcal A_\varepsilon(X)N_{1,\varepsilon}
&\stackrel{\eqref{eq75}}{=}&
	\phi\bigl(
	\alpha_\varepsilon(J_\varepsilon X)
	-\alpha_\varepsilon(X)
	\bigr)N_{1,\varepsilon}
	+\mathcal O(\varepsilon^2)\\
	&=&
	\bigl(
	\alpha_\varepsilon(J_\varepsilon X)
	-\alpha_\varepsilon(X)
	\bigr)\times N_{1,\varepsilon}
	+\mathcal O(\varepsilon^2),
\end{eqnarray*}
where in the last equality we have used that  $\phi(v)w=v\times w$. Consequently,
\begin{equation}\label{eq113}
-\mathcal A_\varepsilon(X)N_{1,\varepsilon}
 +\alpha_\varepsilon(J_\varepsilon X)\times N_{1,\varepsilon} =\alpha_\varepsilon(X)\times N_{1,\varepsilon}
                +\mathcal O(\varepsilon^2).
\end{equation}
Plugging~\eqref{eq113} into~\eqref{eq112} we get that the left-hand side of \eqref{eq:Weingarten-actual-star} is
\begin{eqnarray}
dN_{1,\varepsilon}^*(X)
+\theta_\varepsilon(X)\times N_{1,\varepsilon}^*
&=&
R_\varepsilon^{\top}
\bigl(
dN_{1,\varepsilon}(X)
+\alpha_\varepsilon(X)\times N_{1,\varepsilon}
\bigr)
+\mathcal O(\varepsilon^2)  \nonumber
\\
&\stackrel{\eqref{eq:Weingarten-first-jet}}{=}&
-R_\varepsilon^{\top}
\alpha_\varepsilon(A_\varepsilon X)
+\mathcal O(\varepsilon^2).  \label{eq114}
\end{eqnarray}
In turn, the right-hand side of~\eqref{eq:Weingarten-actual-star} is
\begin{equation}\label{eq115}
-\theta_\varepsilon(A_\varepsilon^*X)
\stackrel{\eqref{thetavarepsilon}}{=}
-R_\varepsilon^{\top}
\alpha_\varepsilon(J_\varepsilon A_\varepsilon^*X)
+\mathcal O(\varepsilon^2).
\end{equation}
Equating~\eqref{eq114} and~\eqref{eq115} and using that
$\alpha_\varepsilon$ is injective, we obtain
\(
J_\varepsilon A_\varepsilon^*
=
A_\varepsilon+\mathcal O(\varepsilon^2).
\)
Using~\eqref{eqJepsilon2}, we can rewrite the last equation as
\begin{equation}\label{eq:shape-comparison-star}
A_\varepsilon^*
=
-J_\varepsilon A_\varepsilon
+\mathcal O(\varepsilon^2).
\end{equation}

For every $s$, the complex structure $J_s$ is skew-adjoint
and the shape operator $A_s$ is self-adjoint with respect to
$g_s$. Denoting the $g_s$-adjoint endomorphism by ${}^\dagger$, we get
\[
\mbox{trace}(J_sA_s)=\mbox{trace}\bigl((J_sA_s)^\dagger\bigr)
=\mbox{trace}(A_s^\dagger J_s^\dagger)\\
=-\mbox{trace}(A_sJ_s)
=-\mbox{trace}(J_sA_s).
\]
Thus, 
\begin{equation}\label{eq117}
\mbox{trace}(J_sA_s)=0 \qquad \mbox{for every $s$.} 
\end{equation}

Next we derive the analogous property for
$\mbox{trace}(J_\varepsilon A_\varepsilon)$.
Since $J_\varepsilon=J+\varepsilon\dot J$ and $A_\varepsilon=A+\varepsilon\dot A$,
expanding their product and taking traces gives
\begin{equation}\label{eq118}
	\mbox{trace}(J_\varepsilon A_\varepsilon)
	=\mbox{trace}(JA)
	+\varepsilon\, \mbox{trace}(\dot J A+J\dot A)
	+\varepsilon^2\, \mbox{trace}(\dot J\dot A)\\
	=\mathcal O(\varepsilon^2),
\end{equation}
where in the last equality the constant term vanishes by~\eqref{eq117} at $s=0$, whereas the coefficient of $\varepsilon$ is
\[
\left.\frac{d}{ds}\right|_{s=0}
\mbox{trace}(J_sA_s)\stackrel{\eqref{eq117}}{=}0.
\]

Finally, taking the trace in~\eqref{eq:shape-comparison-star}
yields
\[
2H_\varepsilon^*
=\mbox{trace}(A_\varepsilon^*)
=-\mbox{trace}(J_\varepsilon A_\varepsilon)
+\mathcal O(\varepsilon^2)
\stackrel{\eqref{eq118}}{=}
\mathcal O(\varepsilon^2).
\]
Unlike $H_\varepsilon=H+\varepsilon\dot H$,
the function $H_\varepsilon^*$ is the actual mean curvature of the immersion $f_\varepsilon^*$. In particular,
\[
\left.\frac{d}{d\varepsilon}\right|_{\varepsilon=0}
H_\varepsilon^*=0.
\]

Finally, recall that the variational field of $f_\varepsilon^*$
is $X^*=f^*\cdot q$. Decompose it as
\(
X^*=df^*(Y)+u^*N^*.
\)
Because $f^*$ is minimal, the tangential component
contributes $Y(H^*)=0$ to the first variation of
mean curvature. With the convention for the Jacobi
operator used in \eqref{eq:minimal-to-first-order},
the normal variation formula gives
\[
0
=
2\left.\frac{d}{d\varepsilon}\right|_{\varepsilon=0}
H_\varepsilon^*
=
\mathcal L_{f^*}u^*,
\]
hence $u^*$ is a Jacobi function. Now the proof is complete.
\end{proof}

\begin{definition}\label{def:Jacobi-conjugate}
{\rm 
Given a Jacobi function $u\colon\Sigma\to\mathbb R$ for $f$, choose the
initial values $\Omega(p_0)\in\mathfrak{so}(3)$ for~\eqref{eq:Omega-IVP}
and $q(p_0)\in\Im (\mathbb{H})$ for~\eqref{eq:q-linearized-system}, and let $X^*$ be the vector field constructed in Theorem~\ref{propos8.9}. The Jacobi function 
\(
u^*=\langle X^*,N^*\rangle
\)
is called a {\em Jacobi conjugate} of $u$. Notice also that $\mathcal L_{f^*}=\mathcal L_f$ by~\eqref{eq146'}.
}
\end{definition}

The dependence on the two initial values is described explicitly in
Lemma~\ref{lem:initial-data-gauge} below; before stating and proving 
the lemma, we will need some preliminaries. Given 
$a,b\in \Im (\mathbb{H})$, the formula 
\begin{equation}\label{Kab}
K_{a,b}(x):=a\cdot x+x\cdot b,
\end{equation}
defines a smooth vector field on $\mathbb{S}^3$. $K_{a,b}$ is the infinitesimal generator of the one-parameter
group of ambient isometries $Q_t^{a,b}(x)
:=\exp(ta)\cdot x\cdot\exp(tb)$, $x\in \mathbb{S}^3$, where $\exp$ denotes the exponential map for the round metric in $\mathbb{S}^3$. Therefore, $K_{a,b}$ is a Killing field on $\mathbb{S}^3$.

Conversely, let $L_a(x)=a\cdot x$ and $R_b(x)=x\cdot b$ denote
the corresponding skew-symmetric linear maps of $\mathbb R^4\simeq\mathbb H$.
Their restrictions to $\mathbb S^3$ are infinitesimal generators of
isometries of the round metric. As a vector-space decomposition, one has
\[
\mathfrak{so}(4)
=
\{L_a+R_b:a,b\in\Im (\mathbb{H})\}
\simeq
\Im (\mathbb{H})\oplus\Im (\mathbb{H}).
\]
Thus, every Killing field of $\mathbb S^3$ is uniquely of the form~\eqref{Kab} for some $a,b\in \Im (\mathbb{H})$. We also have a standard isomorphism of Lie algebras
\[
\Phi \colon \Im (\mathbb{H})\oplus\Im (\mathbb{H})
\to \mathfrak{so}(4),\qquad \Phi(a,b)=\bigl(x\mapsto a\cdot x-x\cdot b\bigr);
\]
hence $K_{a,b}$ corresponds to the pair $(a,-b)$ under $\Phi$.

\begin{definition}\label{def6.6}
{\em 
Let $f\colon\Sigma\to\mathbb S^3$ be an oriented minimal immersion
with unit normal $N$. A function $u\colon\Sigma\to\mathbb R$ is
called a \emph{Jacobi--Killing function} if there exist $a,b\in\Im (\mathbb{H})$ such that
\[
u=u_{a,b}:=
\left\langle K_{a,b}\circ f,N\right\rangle
=
\left\langle
a\cdot f+f\cdot b,N
\right\rangle.
\]
Equivalently, a Jacobi--Killing function is the normal component along
$f$ of an ambient Killing field of $\mathbb S^3$. Since the immersions
$f_t:=Q_t^{a,b}\circ f$ are minimal whenever $f$ is minimal, the first variation of their mean curvature gives $\mathcal L_f u_{a,b}=0$. Thus, every Jacobi--Killing function is a Jacobi function.
}
\end{definition}

\begin{lemma}[Dependence on the initial conditions]
\label{lem:initial-data-gauge}
Fix a Jacobi function $u$ for $f$. Let $u^*$ and
$\widehat u^*$ be the Jacobi conjugates obtained from the initial
data $(\Omega_0,q_0),(\widehat\Omega_0,\widehat q_0)
\in \mathfrak{so}(3)\times\Im (\mathbb{H})$, respectively. Set
\begin{equation}\label{eq:left-Killing-initial-vector}
\begin{split}
 a&:=f^*(p_0)\cdot(\widehat q_0-q_0-b)\cdot f^*(p_0)^{-1}
      \in\Im(\mathbb H), \\
 b&:=\tfrac12\phi^{-1}(C)\in\Im(\mathbb H), \\
 C&:=\widehat\Omega_0-\Omega_0\in\mathfrak{so}(3).
\end{split}
\end{equation}
Then,
\begin{eqnarray}
\widehat q-q
&=&
(f^*)^{-1}\cdot a\cdot f^*+b,  \label{eq:q-initial-data-change}
\\
\widehat u^*-u^*
&=&
\left\langle
a\cdot f^*+f^*\cdot b,N^*
\right\rangle.  \label{eq:u-star-initial-data-change}
\end{eqnarray}
In particular, changing the initial conditions adds to $u^*$ the
normal component of the ambient Killing field
\[
K_{a,b}(x):=a\cdot x+x\cdot b,
\qquad x\in\mathbb S^3.
\]
More explicitly:
\begin{enumerate}[(a)]
\item If only $q_0$ is changed, then $b=0$ and
$a=f^*(p_0)\cdot(\widehat q_0-q_0)\cdot f^*(p_0)^{-1}$; in this case,
the added Killing field is $x\in \mathbb{S}^3\mapsto K_{a,0}(x)=a\cdot x$. 
\item If only $\Omega_0$ is changed while $q_0$ is kept fixed, then
$b=\frac12\phi^{-1}(\widehat\Omega_0-\Omega_0)$, $a=-f^*(p_0)\cdot b\cdot f^*(p_0)^{-1}$ and the corresponding Killing field $K_{a,b}$ vanishes at $f^*(p_0)$.
\end{enumerate}
Thus the Jacobi conjugate of $u$ is naturally defined modulo
Jacobi--Killing functions.
\end{lemma}

\begin{proof}
Let $(\Omega,\dot\alpha^*,q)$ and
$(\widehat\Omega,\widehat{\dot\alpha}^{\,*},\widehat q)$
be the objects determined by the two choices of initial data. Since
both $\Omega$ and $\widehat\Omega$ solve
\eqref{eq:Omega-IVP}, then $d\Omega=R^{\top}\dot{\mathcal A}R=
d\widehat\Omega$ (observe that $R,\dot{\mathcal{A}}$ do not depend on $\Omega(p_0),q(p_0)$ but only on $f,u$). Hence, $d(\widehat\Omega-\Omega)=0$. As $\Sigma$ is connected, there exists $C\in \mathfrak{so}(3)$ such that $\widehat\Omega-\Omega=C$ on $\Sigma$. It follows from \eqref{dotalpha*} that
\begin{equation}\label{eq:dotalpha-initial-data-change}
\widehat{\dot\alpha}^{\,*}-\dot\alpha^*
=-C\alpha^*.
\end{equation}

Subtracting the two equations~\eqref{eq:q-linearized-system} we obtain
\begin{equation}\label{eq:delta-q-system}
D_{\alpha^*}(\widehat q-q)
:=d(\widehat q-q)+\alpha^*\cdot(\widehat q-q)
-(\widehat q-q)\cdot\alpha^*
=\widehat{\dot\alpha}^{\,*}-\dot\alpha^*
\stackrel{\eqref{eq:dotalpha-initial-data-change}}{=}
-C\alpha^*.
\end{equation}
Since $C=2\phi(b)$, for every $X\in T\Sigma$ we have
\begin{equation}\label{eq126}
\bigl(\alpha^*\cdot b-b\cdot\alpha^*\bigr)(X)
=\alpha^*(X)\cdot b-b\cdot\alpha^*(X)\\
=2\alpha^*(X)\times b
=-2b\times\alpha^*(X)=-C\alpha^*(X).
\end{equation}
Therefore
\begin{equation}\label{eq:D-constant-b}
D_{\alpha^*}b\stackrel{\eqref{Dalpha*}}{=}
db+\alpha^*\cdot b-b\cdot \alpha ^*
=\alpha^*\cdot b-b\cdot\alpha^*\stackrel{\eqref{eq126}}{=}
-C\alpha^*.
\end{equation}

For any constant $a\in\Im (\mathbb{H})$, define
\(
c:=(f^*)^{-1}\cdot a\cdot f^*.
\)
Using $df^*=f^*\cdot\alpha^*$ and $d(f^*)^{-1}=-\alpha^*\cdot(f^*)^{-1}$, we find
$dc=-\alpha^*\cdot c+c\cdot\alpha^*$,
and hence~\eqref{Dalpha*} yields
\begin{equation}\label{eq:D-conjugated-a}
D_{\alpha^*}c=0.
\end{equation}
Equations \eqref{eq:D-constant-b} and
\eqref{eq:D-conjugated-a} show that $q_{a,b}:=(f^*)^{-1}\cdot a\cdot f^*+b$ solves \eqref{eq:delta-q-system}, as well as $\widehat q-q$. Regarding their initial conditions at $p_0$, we have
\begin{eqnarray*}
q_{a,b}(p_0)
&=&
f^*(p_0)^{-1}\cdot a\cdot f^*(p_0)+b\\
&\stackrel{\eqref{eq:left-Killing-initial-vector}}{=}&
f^*(p_0)^{-1}\cdot
\left[
f^*(p_0)\cdot
\bigl(\widehat q_0-q_0-b\bigr)
\cdot f^*(p_0)^{-1}
\right]
\cdot f^*(p_0)+b\\
&=&
\widehat q_0-q_0=(\widehat q-q)(p_0).
\end{eqnarray*}

By~\eqref{eq:left-Killing-initial-vector},
$q_{a,b}(p_0)=\widehat q_0-q_0=(\widehat q-q)(p_0)$.
The uniqueness statement in item~3 of Theorem~\ref{propos8.9} therefore implies \eqref{eq:q-initial-data-change}.

Finally, the corresponding change in the vector field along $f^*$
is
\[
\widehat X^*-X^*
=f^*\cdot(\widehat q-q)
\stackrel{\eqref{eq:q-initial-data-change}}{=}
f^*\cdot ((f^*)^{-1}\cdot a\cdot f^*+b)
=a\cdot f^*+f^*\cdot b
=K_{a,b}\circ f^*.
\]
Taking the normal component proves~\eqref{eq:u-star-initial-data-change}. This completes the proof.
\end{proof}

\section{Conjugate Jacobi functions, gauge freedom and exchange of Dirichlet and Neumann conditions}
\label{sec:integrable-gauge}

We now turn to the boundary behavior of the Jacobi conjugate. The
construction in Section~\ref{secJacconj} applies to every Jacobi
function and requires no integrability assumption through a family of minimal immersions; as in that section,
we choose an arbitrary normal variation
\[
f_s\colon\Sigma\to\mathbb S^3,
\qquad f_0=f,
\qquad
\left.\frac{d}{ds}\right|_{s=0}f_s=uN.
\]
The immersions $f_s$ need not be minimal for $s\ne0$; they are used
only to represent the first jets
$g_\varepsilon,J_\varepsilon,\alpha_\varepsilon$, and
$A_\varepsilon$ for the metric, complex structure Maurer-Cartan form and shape operator. The conjugate Jacobi field $X^*=f^*\cdot q$ and its
normal component $u^*=\langle X^*,N^*\rangle$
are those constructed in Theorem~\ref{propos8.9}.

By Lemma~\ref{lem:initial-data-gauge}, changing the initial values $\Omega(p_0)\in\mathfrak{so}(3)$ for~\eqref{eq:Omega-IVP}
and $q(p_0)\in\Im (\mathbb{H})$ for~\eqref{eq:q-linearized-system} changes $X^*$ by the restriction of an arbitrary
ambient Killing field. Thus, for $\Lambda\in\mathfrak{so}(4)$, the
corresponding gauge change is
\begin{equation}\label{eq:ambient-gauge}
u^*\mapsto u^*+u^*_\Lambda,
\qquad
u^*_\Lambda:=\langle\Lambda f^*,N^*\rangle.
\end{equation}

The notation in \eqref{eq:ambient-gauge} is equivalent to that used in
Lemma~\ref{lem:initial-data-gauge}. Indeed, under the standard
isomorphism
\[
\Im (\mathbb{H})\oplus\Im (\mathbb{H})\to \mathfrak{so}(4),
\quad
(a,b)\mapsto \Lambda_{a,b},\qquad 
\mbox{where $\Lambda_{a,b}(x):=a\cdot x+x\cdot b,
\qquad x\in\mathbb H$,}
\]
the Killing field $K_{a,b}$ introduced in~\eqref{Kab} is precisely
$K_{a,b}(x)=\Lambda_{a,b}x$. Consequently, along $f^*$, it holds
$K_{a,b}\circ f^*=
a\cdot f^*+f^*\cdot b
=
\Lambda_{a,b}f^*$, and the corresponding Jacobi--Killing function can be written
equivalently as
\[
u_{a,b}^*
:=
\left\langle K_{a,b}\circ f^*,N^*\right\rangle
=
\left\langle \Lambda_{a,b}f^*,N^*\right\rangle
=
u_{\Lambda_{a,b}}^*.
\]
Thus the notation $u_\Lambda^*$ used here is simply the coordinate-free version of the notation $u_{a,b}^*$ from Definition~\ref{def6.6}; every $\Lambda\in\mathfrak{so}(4)$ is uniquely of the form $\Lambda=\Lambda_{a,b}$.

The function $u^*_\Lambda$ is Jacobi--Killing. The two boundary
results below (Theorems~\ref{thm:D-to-N} and~\ref{thm:N-to-D}) amount to choosing this gauge appropriately.

\begin{remark}[Relation with almost-even and almost-odd fields]
\label{rem:almost-even-almost-odd}
{\rm
A related boundary phenomenon appears in
\cite[Section~3.4, Proposition~3.8]{GKKRS2009}
for the conjugate-cousin correspondence between CMC $1$-surfaces in $\mathbb R^3$ and minimal surfaces in $\mathbb S^3$.
There, a variational field is called almost even if, along each
boundary component, it differs from an even field by a translation,
whereas a cousin field is called almost odd if, along each conjugate
boundary component, it differs from an odd field by a left
translation. The cited proposition shows that the cousin
correspondence exchanges these two weakened parity conditions.
As the authors point out, the uncorrected even--odd exchange need not hold.
Their result is closely related to the boundary mechanism underlying
our Theorems~\ref{thm:D-to-N} and~\ref{thm:N-to-D}, but the statements
are different. Proposition~3.8 of~\cite{GKKRS2009} concerns
vector-valued Jacobi fields in the CMC--minimal cousin
correspondence and allows a separate translation correction along
each boundary component. Here we consider scalar Jacobi functions
for conjugate minimal immersions in $\mathbb S^3$ along one
prescribed arc, and obtain the desired homogeneous boundary
condition by choosing a single global ambient Killing gauge.
}
\end{remark}

\subsection{From Dirichlet to Neumann data}
\label{sec:D-to-N}

\begin{theorem}\label{thm:D-to-N}
Let $f\colon\Sigma\to\esf^3$ be a simply connected minimal immersion
and let $\gamma\subset\Sigma$ be such that $f(\gamma)$ is a
great-circle arc. If $u$ is a Jacobi function satisfying $u|_\gamma=0$,
then the initial conditions in Theorem~\ref{propos8.9} can be
chosen so that the resulting Jacobi conjugate $u^*$ satisfies
$ \partial_{\eta}u^*=0$  along $\gamma$ (here $\eta$ denotes the conormal in the common parameter surface; under
$f^*$ it is the conormal to the reflection curve $f^*(\gamma)$).
\end{theorem}

\begin{proof}
Use $\gamma(t)$ also for the arc-length parametrization of the source arc. Then
$f(\gamma (t))=\cos t\,{\bf p}+\sin t\,{\bf v}$, $\langle{\bf p},{\bf v}\rangle=0$, 
and put $T=\gamma '$ and $\eta=JT$.

Choose any normal variation $f_s$ with variational field $uN$. We use
only the first jets of objects associated to $f_s$, denoted as in Section~\ref{secJacconj}. Set
\begin{equation}\label{etavarepsilon}
 \eta_\varepsilon:=J_\varepsilon \gamma ',
 \qquad
 {\bf b}_\varepsilon:=(\alpha_\varepsilon)_\gamma (\gamma '),
 \qquad
 {\bf a}_\varepsilon:=d(f_\varepsilon^*)_\gamma (\eta_\varepsilon),
\end{equation}
where $J_{\varepsilon},\alpha_{\varepsilon}$ are given by~\eqref{alphaJAepsilon}, and
$f_\varepsilon^*$ is the actual variation constructed in
\eqref{fvarepsilon*} (in particular, we have made a choice of 
the initial values $\Omega(p_0)\in\mathfrak{so}(3)$ for~\eqref{eq:Omega-IVP} and $q(p_0)\in\Im (\mathbb{H})$ for~\eqref{eq:q-linearized-system}). 

Observe that although $\gamma $ is parametrized by arc length with respect to $g=g_0$, the fixed vector field $\gamma '$, and hence $\eta_\varepsilon=J_\varepsilon \gamma '$, need not be unit with respect to the deformed metric $g_\varepsilon$ given by~\eqref{alphaJAepsilon} ($\eta_\varepsilon$ is $g_\varepsilon$-unit up to second order). Also note that for each fixed $\varepsilon$, both ${\bf b}_{\varepsilon}$ and ${\bf a}_{\varepsilon}$ depend smoothly on $t$, and ${\bf b}_{\varepsilon}$ takes values in $\Im (\mathbb{H})$ while ${\bf a}_{\varepsilon}$ is tangent to $f_{\varepsilon}^*$ along the curve $f_{\varepsilon}^*(\gamma )$.
Thus, we can expand 
\[
{\bf b}_\varepsilon = {\bf b} + \varepsilon\dot{\bf b}, \quad
{\bf a}_\varepsilon = {\bf a} + \varepsilon\dot{\bf a}
+ \mathcal O(\varepsilon^2), \quad 
\mbox{where }
\dot{\bf b}
:=
\left.\frac{d}{d\varepsilon}\right|_{\varepsilon=0}
{\bf b}_\varepsilon, 
\quad 
\dot{\bf a}
:=
\left.\frac{d}{d\varepsilon}\right|_{\varepsilon=0}
{\bf a}_\varepsilon,
\]
where the first expansion is exact (without term $\mathcal{O}(\varepsilon^2)$) because $\alpha_\varepsilon=\alpha+\varepsilon\dot\alpha$ and ${\bf b}_\varepsilon:=(\alpha_\varepsilon)_\gamma (\gamma ')$.
\par\vspace{.2cm}\noindent
{\bf Claim.} {\em 
As functions of $t$, both ${\bf b}$, ${\bf a}$ and $\dot{\bf a}$ are a constant vector along $\gamma $. Moreover, $\dot{\bf b}=0$ along $\gamma $.
}
\par\noindent
{\em Proof of the claim.}  Note that 
\(
{\bf b}:=\alpha_\gamma (T)=\overline{\bf p}\cdot{\bf v},
\)
which is constant along $\gamma $ by item~1 of Lemma~\ref{lem8.14}. Also,
\(
{\bf a}:=df_\gamma ^*(\eta)=-f^*(\gamma )\cdot R_\gamma ^\top{\bf b}
\)
is constant in $t$ by item~2 of Lemma~\ref{lem8.14}.

Since $u=0$ on $\gamma$, the variational field $uN$ vanishes
there, and differentiation of ${\bf b}_\varepsilon$ at $\varepsilon=0$ gives
\begin{equation*}
 \dot{\bf b}
 =
 -f(\gamma )^{-1}\cdot (uN)_\gamma \cdot f(\gamma )^{-1}\cdot df_\gamma (\gamma ')
 +f(\gamma )^{-1}\cdot d(uN)_\gamma (\gamma ')
 =0 \qquad \mbox{along $\gamma $.}
\end{equation*}

Thus, in order to finish the proof of the claim, it only remains to demonstrate that $\dot{\bf a}$ is constant along $\gamma $. To do this, define
\begin{equation}\label{hvarepsilon}
{\bf h}_\varepsilon
:=(\alpha_\varepsilon)_\gamma (\eta_\varepsilon),
\qquad
B_\varepsilon:=(R_{\varepsilon})_\gamma ^{\top}{\bf b}_\varepsilon,
\qquad
H_\varepsilon:=(R_{\varepsilon})_\gamma ^{\top}{\bf h}_\varepsilon.
\end{equation}
Then
\begin{eqnarray}
d(f_\varepsilon^*)_\gamma (\gamma ')
&\stackrel{\eqref{MCfepsilon*}}{=}&
f_\varepsilon^*(\gamma )\cdot (\theta_\varepsilon)_\gamma (\gamma ')
\stackrel{\eqref{thetavarepsilon}}{=}
f_\varepsilon^*(\gamma )\cdot
(R_\varepsilon)_\gamma ^{\top}(\alpha_\varepsilon\circ J_\varepsilon)(\gamma ')
+\mathcal O(\varepsilon^2)  \nonumber
\\
&\stackrel{\eqref{etavarepsilon}}{=}&
f_\varepsilon^*(\gamma )\cdot
(R_\varepsilon)_\gamma ^{\top}(\alpha_\varepsilon)_\gamma (\eta _\varepsilon)
+\mathcal O(\varepsilon^2)
\stackrel{\eqref{hvarepsilon}}{=}
f_\varepsilon^*(\gamma )\cdot
(R_\varepsilon)_\gamma ^{\top}{\bf h}_\varepsilon
+\mathcal O(\varepsilon^2)  \nonumber
\\
&\stackrel{\eqref{hvarepsilon}}{=}&
f_\varepsilon^*(\gamma )\cdot H_\varepsilon
+\mathcal O(\varepsilon^2),  \label{df*varepsilonc'}
\\
{\bf a}_\varepsilon
&=&
d(f_\varepsilon^*)_\gamma (\eta_\varepsilon)\stackrel{\eqref{MCfepsilon*}}{=}
f_\varepsilon^*(\gamma )\cdot (\theta_\varepsilon)_\gamma (\eta_\varepsilon)
\stackrel{\eqref{thetavarepsilon}}{=}
f_\varepsilon^*(\gamma )\cdot
(R_\varepsilon)_\gamma ^{\top}(\alpha_\varepsilon\circ J_\varepsilon)
(\eta_\varepsilon)
+\mathcal O(\varepsilon^2)  \nonumber
\\
&\stackrel{\eqref{etavarepsilon}}{=}&
f_\varepsilon^*(\gamma )\cdot
(R_\varepsilon)_\gamma ^{\top}(\alpha_\varepsilon)_\gamma (J_\varepsilon^2 \gamma ')
+\mathcal O(\varepsilon^2)
\stackrel{\eqref{eqJepsilon2}}{=}
-f_\varepsilon^*(\gamma )\cdot
(R_\varepsilon)_\gamma ^{\top}(\alpha_\varepsilon)_\gamma (\gamma ')
+\mathcal O(\varepsilon^2)  \nonumber
\\
&\stackrel{\eqref{etavarepsilon}}{=}&
-f_\varepsilon^*(\gamma )\cdot
(R_\varepsilon)_\gamma ^{\top} {\bf b}_\varepsilon
+\mathcal O(\varepsilon^2)
\stackrel{\eqref{hvarepsilon}}{=}
-f_\varepsilon^*(\gamma )\cdot B_\varepsilon
+\mathcal O(\varepsilon^2).   \label{avarepsilon}
\end{eqnarray}
Differentiating~\eqref{avarepsilon} in the direction $\gamma '$ and
using the Leibniz rule, we obtain
\begin{eqnarray}
d({\bf a}_\varepsilon)_\gamma (\gamma ')
&=&
-d(f_\varepsilon^*)_\gamma (\gamma ')\cdot B_\varepsilon
-f_\varepsilon^*(\gamma )\cdot
d(B_\varepsilon)_\gamma (\gamma ')  
+\mathcal O(\varepsilon^2)  \nonumber
\\
&\stackrel{\eqref{df*varepsilonc'}}{=}&
-f_\varepsilon^*(\gamma )\cdot H_\varepsilon\cdot B_\varepsilon
-f_\varepsilon^*(\gamma )\cdot
d(B_\varepsilon)_\gamma (\gamma ')  
+\mathcal O(\varepsilon^2). \label{eq:da-product-rule}
\end{eqnarray}
To compute the second term in the right-hand side of~\eqref{eq:da-product-rule}, we differentiate the definition of $B_\varepsilon$ along $\gamma $:
\begin{eqnarray*}
d(B_\varepsilon)_\gamma (\gamma ')
&=&
d(R_\varepsilon^{\top})_\gamma (\gamma ')\,{\bf b}_\varepsilon
+
(R_\varepsilon)_\gamma ^{\top}
d({\bf b}_\varepsilon)_\gamma (\gamma ')
\\
&\stackrel{\eqref{eq:formal-R}}{=}&
(\mathcal A_\varepsilon R_\varepsilon)_\gamma ^{\top}
(\gamma ')\,{\bf b}_\varepsilon
+
(R_\varepsilon)_\gamma ^{\top}
d({\bf b}_\varepsilon)_\gamma (\gamma ')
+\mathcal{O}(\varepsilon^2)
\\
&=&
-(R_\varepsilon)_\gamma ^{\top}(\mathcal A_\varepsilon)_\gamma 
(\gamma ')\,{\bf b}_\varepsilon
+
(R_\varepsilon)_\gamma ^{\top}
d({\bf b}_\varepsilon)_\gamma (\gamma ')
+\mathcal{O}(\varepsilon^2)
\\
&\stackrel{\eqref{eq75}}{=}&
-(R_\varepsilon)_\gamma ^{\top}
\phi\left( (\alpha_\varepsilon)_\gamma (J_\varepsilon \gamma '-\gamma ')\right)
\,{\bf b}_\varepsilon
+
(R_\varepsilon)_\gamma ^{\top}
d({\bf b}_\varepsilon)_\gamma (\gamma ')
+\mathcal{O}(\varepsilon^2)
\\
&\stackrel{\eqref{etavarepsilon},\eqref{hvarepsilon}}{=}&
-(R_\varepsilon)_\gamma ^{\top}\phi\left(
{\bf h}_\varepsilon-{\bf b}_\varepsilon \right) {\bf b}_\varepsilon
+
(R_\varepsilon)_\gamma ^{\top}
d({\bf b}_\varepsilon)_\gamma (\gamma ')
+ \mathcal O(\varepsilon^2)
\\
&\stackrel{\eqref{eq138'}}{=}&
-(R_\varepsilon)_\gamma ^{\top}
[({\bf h}_\varepsilon-{\bf b}_\varepsilon)
\times{\bf b}_\varepsilon]
+
(R_\varepsilon)_\gamma ^{\top}
d({\bf b}_\varepsilon)_\gamma (\gamma ')
+ \mathcal O(\varepsilon^2)
\\
&=&
-(R_\varepsilon)_\gamma ^{\top}
({\bf h}_\varepsilon\times{\bf b}_\varepsilon)
+
(R_\varepsilon)_\gamma ^{\top}
d({\bf b}_\varepsilon)_\gamma (\gamma ')
+ \mathcal O(\varepsilon^2).
\end{eqnarray*}
Recall that $R_\varepsilon$ preserves the Euclidean
cross product up to terms of order $\mathcal{O}(\varepsilon^2)$. Thus,
\begin{eqnarray}
d(B_\varepsilon)_\gamma (\gamma ')
&=&
-\bigl((R_\varepsilon)_\gamma ^{\top}{\bf h}_\varepsilon\bigr)
\times
\bigl((R_\varepsilon)_\gamma ^{\top}{\bf b}_\varepsilon\bigr)
+(R_\varepsilon)_\gamma ^{\top}d({\bf b}_\varepsilon)_\gamma (\gamma ')
+ \mathcal O(\varepsilon^2)  \nonumber
\\
&\stackrel{\eqref{hvarepsilon}}{=}&
-H_\varepsilon\times B_\varepsilon
+(R_\varepsilon)_\gamma ^{\top}d({\bf b}_\varepsilon)_\gamma (\gamma ')
+\mathcal O(\varepsilon^2).  \label{dBvarepsilonc'}
\end{eqnarray}
Plugging~\eqref{dBvarepsilonc'} into~\eqref{eq:da-product-rule} we get
\[
d({\bf a}_\varepsilon)_\gamma (\gamma ')
=
f_\varepsilon^*(\gamma )\cdot \big( -H_\varepsilon\cdot B_\varepsilon
+H_\varepsilon\times B_\varepsilon\big)
-f_\varepsilon^*(\gamma )\cdot 
(R_\varepsilon)_\gamma ^{\top}d({\bf b}_\varepsilon)_\gamma (\gamma ')
+\mathcal O(\varepsilon^2). 
\]

Since both $H_\varepsilon,B_\varepsilon$ lie in $\Im (\mathbb{H})$ and
\begin{eqnarray*}
\langle H_\varepsilon,B_\varepsilon\rangle
&\stackrel{\eqref{hvarepsilon}}{=}&
\langle (R_{\varepsilon})_\gamma ^{\top}{\bf h}_\varepsilon,
(R_{\varepsilon})_\gamma ^{\top}{\bf b}_\varepsilon\rangle 
=
\langle (R_{\varepsilon}R_{\varepsilon}^{\top})_\gamma {\bf h}_\varepsilon,
{\bf b}_\varepsilon\rangle
\\
&\stackrel{\eqref{eq:R-epsilon-orthogonal}}{=}&
\langle {\bf h}_\varepsilon,{\bf b}_\varepsilon\rangle
+\mathcal O(\varepsilon^2)
\stackrel{\eqref{etavarepsilon},\eqref{hvarepsilon}}{=}
\langle (\alpha_\varepsilon)_\gamma (\eta_\varepsilon),
(\alpha_\varepsilon)_\gamma (\gamma ')\rangle
+\mathcal O(\varepsilon^2)
\\
&=&
g_\varepsilon(J_\varepsilon \gamma ',\gamma ')
+\mathcal O(\varepsilon^2)
=\mathcal O(\varepsilon^2),
\end{eqnarray*}
then we conclude that $H_\varepsilon\cdot B_\varepsilon
=H_\varepsilon\times B_\varepsilon+\mathcal O(\varepsilon^2)$.
Thus,
\begin{equation}\label{eq:a-derivative-along-c}
d({\bf a}_\varepsilon)_\gamma (\gamma ')
=
-f_\varepsilon^*(\gamma )\cdot 
(R_\varepsilon)_\gamma ^{\top}d({\bf b}_\varepsilon)_\gamma (\gamma ')
+\mathcal O(\varepsilon^2). 
\end{equation}

Since both ${\bf b}$ and $\dot{\bf b}=0$ are constant along $\gamma $, it follows that their respective tangential derivatives along $\gamma $ also vanish:
\begin{equation}\label{eq:d-bdot-zero-along-c}
d{\bf b}_\gamma (\gamma ')=0, \qquad d\dot{\bf b}_\gamma (\gamma ')=0.
\end{equation}
Differentiating ${\bf b}_\varepsilon={\bf b}+\varepsilon\dot{\bf b}$ along $\gamma $ and using~\eqref{eq:d-bdot-zero-along-c}, we obtain
\begin{equation*}
d({\bf b}_\varepsilon)_\gamma (\gamma ')=d{\bf b}_\gamma (\gamma ')
+\varepsilon\,d\dot{\bf b}_\gamma (\gamma ')=0.
\end{equation*}
It follows from \eqref{eq:a-derivative-along-c} that
\begin{equation}\label{davarepsilonc'}
d({\bf a}_\varepsilon)_\gamma (\gamma ')
=
\mathcal O(\varepsilon^2). 
\end{equation}

Differentiating ${\bf a}_\varepsilon = {\bf a} + \varepsilon\dot{\bf a}
+ \mathcal O(\varepsilon^2)$ along $\gamma $ gives
\begin{equation}\label{eq:da-epsilon-expansion}
	d({\bf a}_\varepsilon)_\gamma (\gamma ')
	=
	d{\bf a}_\gamma (\gamma ')
	+
	\varepsilon\,d\dot{\bf a}_\gamma (\gamma ')
	+
	\mathcal O(\varepsilon^2).
\end{equation}
Since ${\bf a}$ is constant along $\gamma $, we have $d{\bf a}_\gamma (\gamma ')=0$. Using this together with~\eqref{davarepsilonc'} and~\eqref{eq:da-epsilon-expansion}, we have $d\dot{\bf a}_\gamma (\gamma ')=0$. Thus, $\dot{\bf a}$ is a constant vector along $\gamma $ and the claim is proved.
\par\vspace{.2cm}

We next show that the constant vectors ${\bf a},\dot{\bf a}$ are orthogonal.
Let $g_\varepsilon^*$ denote the metric induced by $f_\varepsilon^*$. By the construction in Section~\ref{secJacconj}, we have 
$g_\varepsilon^*=g_\varepsilon+\mathcal O(\varepsilon^2)$. Now,
\begin{eqnarray*}
|{\bf a}_\varepsilon|^2
&\stackrel{\eqref{etavarepsilon}}{=}&
|d(f_\varepsilon^*)_\gamma (\eta_\varepsilon)|^2
=
g_\varepsilon^*(\eta_\varepsilon,\eta_\varepsilon)=
g_\varepsilon(\eta_\varepsilon,\eta_\varepsilon)
+\mathcal O(\varepsilon^2)
\\
&\stackrel{\eqref{etavarepsilon}}{=}&
g_\varepsilon(J_\varepsilon \gamma ',J_\varepsilon \gamma ')
+\mathcal O(\varepsilon^2)
=
g_\varepsilon(\gamma ',\gamma ')
+\mathcal O(\varepsilon^2),
\end{eqnarray*}
Differentiating at $\varepsilon=0$ gives
$2\langle\dot{\bf a},{\bf a}\rangle
=
\dot g(\gamma ',\gamma ')\stackrel{\eqref{dotg}}{=}-2u\,\sigma(\gamma ',\gamma ')=0$, where in the last equality we have used the hypothesis $u=0$ along $\gamma$. Therefore, $\langle\dot{\bf a},{\bf a}\rangle=0$. Since both ${\bf a},\dot{\bf a}$ are constant along $\gamma $,
$|{\bf a}|=1$ and $\dot{\bf a}\perp{\bf a}$, there exists $\Lambda\in\mathfrak{so}(4)$ such that
\begin{equation}\label{Lambdaadota}
\Lambda{\bf a}=-\dot{\bf a}.
\end{equation}
For example, one may take $\Lambda x
=
-\langle{\bf a},x\rangle\dot{\bf a}
+
\langle\dot{\bf a},x\rangle{\bf a}$,
 $x\in\mathbb H$.

By Lemma~\ref{lem:initial-data-gauge}, we may change the initial values $\Omega(p_0)\in\mathfrak{so}(3)$ for~\eqref{eq:Omega-IVP} and $q(p_0)\in\Im (\mathbb{H})$ for~\eqref{eq:q-linearized-system} so that $X^*$ is replaced by 
\[
\widehat X^*:=X^*+\Lambda f^*.
\]
We now explain how this gauge change affects ${\bf a}_{\varepsilon}$ and $\dot{\bf a}$. The vector field $\eta_\varepsilon$ depends only on the original variation $f_s$, and remains therefore unaffected by the change of initial conditions in the linearized conjugate system. 

Let $f_\varepsilon^*$ and $\widehat f_\varepsilon^*$ be smooth variations of $f^*$ whose variational fields are $X^*$ and
$\widehat X^*$, respectively (the fact that 
$f_\varepsilon^*,\widehat f_\varepsilon^*$ are given by  formula~\eqref{fvarepsilon*}) is not relevant in what follows). Define
${\bf a}_\varepsilon
=
d(f_\varepsilon^*)_\gamma (\eta_\varepsilon)$,
$\widehat{\bf a}_\varepsilon
=
d(\widehat f_\varepsilon^*)_\gamma (\eta_\varepsilon)$. Differentiating these equalities at $\varepsilon=0$ we obtain $\dot{\bf a}=dX_\gamma ^*(\eta)+df_\gamma ^*(\dot\eta)$, $\widehat{\dot{\bf a}}=d\widehat X_\gamma ^*(\eta)+df_\gamma ^*(\dot\eta)$. Subtracting the two identities, the terms containing $\dot\eta$ cancel, and
\begin{equation}\label{changedota}
\widehat{\dot{\bf a}}-\dot{\bf a}=d(\widehat X^*-X^*)_\gamma (\eta)=
d(\Lambda f^*)_\gamma (\eta)=\Lambda\, df_\gamma ^*(\eta)=\Lambda\, {\bf a}.
\end{equation}

Now, \eqref{Lambdaadota} and~\eqref{changedota} imply that after
making this change of initial conditions, we have $\widehat{\dot{\bf a}}=0$. To simplify the notation, we shall henceforth omit the hat and we may therefore assume that
\begin{equation}\label{eq:adot-zero}
\dot{\bf a}=0 \qquad\text{along }\gamma.
\end{equation}
Decompose the conjugate vector field into tangent and normal parts to $f^*$: 
\(
X^*=df^*(Y)+u^*N^*.
\)
By definition of ${\bf a}_\varepsilon$, we have $\left\langle f_\varepsilon^*(\gamma ),{\bf a}_\varepsilon \right\rangle
=
\left\langle f_\varepsilon^*(\gamma ),d(f_\varepsilon^*)_\gamma (\eta_\varepsilon) \right\rangle
=
0$.  Differentiating at $\varepsilon=0$ gives
\(
\langle X^*,{\bf a}\rangle
+
\langle f^*(\gamma ),\dot{\bf a}\rangle
=
0.
\)
By \eqref{eq:adot-zero}, the second term vanishes, and therefore
$\langle X^*,{\bf a}\rangle=0$. Since ${\bf a}=df_\gamma ^*(\eta)$ and $N^*$ is orthogonal to $df^*(T\Sigma)$, it follows that
\begin{equation}\label{eq:Y-eta-zero}
0=\langle X^*,{\bf a}\rangle
=
\langle df^*(Y)+u^*N^*,df_\gamma ^*(\eta)\rangle
=
g(Y,\eta).
\end{equation}

On the other hand, differentiating ${\bf a}_\varepsilon
=d(f_\varepsilon^*)_\gamma (\eta_\varepsilon)$ at $\varepsilon=0$ yields
\[
\dot{\bf a}
=\overline{\nabla}_{df^*_\gamma (\eta)}X^*+df^*_\gamma (\dot{\eta})
=\nabla^{\mathbb S^3}_{df^*_\gamma (\eta)}X^*
-\langle df^*_\gamma (\eta),X^*\rangle f^*(\gamma )
+df^*_\gamma (\dot\eta).
\]
Taking the scalar product with $N^*\circ \gamma $, and using
\eqref{eq:adot-zero}, we obtain
\begin{equation}\label{eq:normal-derivative-X-star}
\left\langle
\nabla^{\mathbb S^3}_{df^*_\gamma (\eta)}X^*,
N^*
\right\rangle
=
0.
\end{equation}

Finally, differentiating $u^*=\langle X^*,N^*\rangle$ in the direction $\eta$, and using \eqref{eq:normal-derivative-X-star} together with the Weingarten equation, gives
\[
\partial_\eta u^*
=	\left\langle
\nabla^{\mathbb S^3}_{df^*_\gamma (\eta)}X^*,
N^*
\right\rangle
+
\left\langle
X^*,
\nabla^{\mathbb S^3}_{df^*_\gamma (\eta)}N^*
\right\rangle\\
=
-\langle Y,A^*\eta\rangle,
\]
where $A^*$ is the shape operator of $f^*$. Since the reflection curve $f^*(\gamma)$ is a curvature line of $f^*$, $A^*\eta$ is a multiple of $\eta$. This fact together with~\eqref{eq:Y-eta-zero} imply that 
$\partial_\eta u^*=0$ along $\gamma$, which completes the proof.
\end{proof}

\begin{remark}\label{rem:D-to-N-gauge}
{\rm 
The conclusion of Theorem~\ref{thm:D-to-N} is gauge-dependent. Keeping the conjugate immersion $f^*$ fixed, one can choose the initial conditions of the linearized system so that $\dot{\bf a}=0$ along the chosen arc. Thus the conormal of the reflection sphere is fixed to first order. This statement concerns the first jet of the variation: the immersions $f^*_\varepsilon$ need not have reflection symmetry for $\varepsilon\ne0$. An arbitrary normalization of the linearized system need not give zero Neumann data.
}
\end{remark}

\subsection{From Neumann to Dirichlet data}
\label{sec:N-to-D}
\begin{theorem}\label{thm:N-to-D}
Let $f\colon\Sigma\to\esf^3$ be a simply connected minimal immersion
and let $\gamma\subset\Sigma$ be a reflection geodesic. If $u$ is a
Jacobi function satisfying $\partial_\eta u=0$  along $\gamma$, then, for any initially chosen conjugate $u^*$, there exists
$\Lambda\in\mathfrak{so}(4)$ such that
\begin{equation}\label{eq:N-to-D-conclusion}
 (u^*-u_\Lambda^*)|_\gamma=0,
 \qquad \mbox{where } u_\Lambda^*=\langle\Lambda f^*,N^*\rangle.
\end{equation}

More precisely, if we write the conjugate great circle as
\[
 f^*(\gamma (t))=\cos t\,{\bf p}+\sin t\,{\bf v},
 \qquad
 \Pi=\operatorname{Span}\{{\bf p},{\bf v}\},
\]
then there exists $\Lambda_0\in\mathfrak{so}(4)$, given by~\eqref{eq:Lambda-zero} below, such that every matrix $\Lambda$ in the affine
two-plane
\begin{equation}\label{eq:affine-gauge-space}
\Lambda_0+\mathfrak{so}(\Pi)\oplus\mathfrak{so}(\Pi^\perp)
\end{equation}
satisfies \eqref{eq:N-to-D-conclusion}.
\end{theorem}
\begin{proof}
Parametrize $\gamma $ by its arc length and choose the sign of $\eta$ so that, by the converse statement in Lemma~\ref{lem8.14}, ${\bf a}:=df_\gamma (\eta)$ is constant and $f^*(\gamma )$ is a great-circle arc. 

Choose any normal variation $f_s$ of $f$ with variational field $uN$. Let $g_s,J_s$ be the induced metric and conformal structure on $\Sigma$ by $f_s$. Put 
\begin{equation}\label{dihedron}
T=\gamma ',\qquad \tau_s:=T/|T|_{g_s},\qquad 
\eta_s:=J_s\tau_s,
\end{equation}
hence $\tau_s,\eta_s$ form a positively oriented $g_s$-orthonormal basis of $T_{\gamma}\Sigma$. The variation $f_s$ will only be used to compute first derivatives and need not consist of minimal immersions. 

Define 
\(
{\bf a}_s:=(df_s)_\gamma (\eta_s).
\)
 For $s=0$, ${\bf a}_0={\bf a}$ is
constant. Since $\gamma$ is a curvature line of $f$, we can write
$A T=\kappa T$, where $A$ is the shape operator of $f$ and $\kappa$ denotes the corresponding principal curvature. As $f$ is minimal, $A\eta=-\kappa\, \eta$.

We now compute the first variation $\dot\eta$ of $\eta_s$, which is a tangent vector field to $\Sigma$ along $\gamma$. By~\eqref{dihedron},
we have $g_s(\eta_s,T)=0$, $g_s(\eta_s,\eta_s)=1$. Differentiating the first identity at $s=0$,
\begin{eqnarray}
0 & = &
\dot g(\eta,T)+g(\dot\eta,T)\stackrel{\eqref{dotg}}{=}
-2u\,\sigma(\eta,T)+g(\dot\eta,T)
=
-2u\,g(A\eta,T)+g(\dot\eta,T) \nonumber
\\
&=&
2u\kappa\,g(\eta,T)+g(\dot\eta,T)
=
g(\dot\eta,T). \label{eq153}
\end{eqnarray}
Since at $s=0$, the vector fields $T,\eta$ form a $g$-orthonormal basis of $T\Sigma$ along $\gamma $, \eqref{eq153} gives that $\dot{\eta}$ is everywhere proportional to $\eta$. To compute this proportionality factor, we differentiate $g_s(\eta_s,\eta_s)=1$ at $s=0$, obtaining
\[
\begin{aligned}
0&=\dot g(\eta,\eta)+2g(\dot\eta,\eta)\stackrel{\eqref{dotg}}{=}2\bigl[-u\,\sigma(\eta,\eta)+g(\dot\eta,\eta)\bigr]\\
 &=2\bigl[-u\,g(A\eta,\eta)+g(\dot\eta,\eta)\bigr]=2\bigl[u\kappa+g(\dot\eta,\eta)\bigr].
\end{aligned}
\]
Therefore, $g(\dot\eta,\eta)=-u\kappa$ and thus, $\dot\eta=-u\kappa\eta$. 
Consequently, the first variation of ${\bf a}_s=(df_s)_\gamma (\eta_s)$ is
\[
\dot{\bf a}
 =
 \overline{\nabla}_{df_\gamma (\eta)}(uN)+df_\gamma (\dot\eta)
 =
 (\partial_\eta u)(N\circ \gamma )
 +
u \overline{\nabla}_{df_\gamma (\eta)}N-u\kappa \, df_\gamma (\eta).
\]
The first term vanishes since we are assuming that $u$ has zero Neumann data along $\gamma$. The other two terms cancel since
\[
\overline{\nabla}_{df_\gamma (\eta)}N-\kappa \, df_\gamma (\eta)
=-df_\gamma (A\eta)-\kappa\, df_\gamma (\eta)=-df_\gamma (A\eta+\kappa \, \eta)=0.
\]
Therefore,
\begin{equation}\label{eq:adot-Neumann}
	\dot{\bf a}=0 \qquad \mbox{along $\gamma$.}
\end{equation}

From now on, choose initial values $\Omega(p_0)\in\mathfrak{so}(3)$ for~\eqref{eq:Omega-IVP} and $q(p_0)\in\Im(\mathbb H)$ for~\eqref{eq:q-linearized-system}.
Let $X^*$ be the vector field constructed in
Theorem~\ref{propos8.9}, and let $u^*:=\langle X^*,N^*\rangle$
be its normal component, which is a Jacobi function for $f^*$.
\par\vspace{.2cm}\noindent
{\bf Claim.} 
{\em The normal component $Z^{\perp}$ of $Z(t):=X_\gamma ^*(t)$
is a Jacobi field along the great circle $f^*(\gamma )$.
}
\par\noindent
{\em Proof of the claim.}
Let $\nabla^s$ denote the Levi--Civita connection of $g_s$, and
let $\kappa_{g,s}$ be the geodesic curvature of $\gamma $ with respect
to $g_s$, with the convention
\begin{equation}\label{eq:geodesic-curvature-convention}
\nabla^s_{\tau_s}\tau_s
=
\kappa_{g,s}\, \eta_s.
\end{equation}
Since $\nabla^s J_s=0$, we have 
\begin{equation}\label{eq:conormal-derivative}
\nabla^s_{\tau_s}\eta_s
=
\nabla^s_{\tau_s}(J_s\tau_s)
=
J_s\nabla^s_{\tau_s}\tau_s
\stackrel{\eqref{eq:geodesic-curvature-convention}}{=} 
J_s(\kappa_{g,s}\, \eta_s)
=
\kappa_{g,s}\, J_s\eta_s
=
-\kappa_{g,s}\, \tau_s.
\end{equation}

We now differentiate ${\bf a}_s=(df_s)_\gamma (\eta_s)$ 
along the curve in the direction $\tau_s$. The Gauss formula for
the immersion $f_s\colon(\Sigma,g_s)\to\mathbb S^3$ gives
\begin{eqnarray*}
\nabla^{\mathbb S^3}_{df_s(\tau_s)}{\bf a}_s
&=&
\nabla^{\mathbb S^3}_{(df_s)_\gamma (\tau_s)}
(df_s)_\gamma (\eta_s)
=
(df_s)_\gamma \bigl(\nabla^s_{\tau_s}\eta_s\bigr)
+
\sigma_s(\tau_s,\eta_s)N_s
\\
&\stackrel{\eqref{eq:conormal-derivative}}{=}&
-\kappa_{g,s}\,(df_s)_\gamma (\tau_s)
+
\sigma_s(\tau_s,\eta_s)N_s.
\end{eqnarray*}
Since $\langle (df_s)_\gamma (\tau_s),(df_s)_\gamma (\eta_s)\rangle =
g_s(\tau_s,\eta_s)=0$, the Gauss formula for $\mathbb S^3\subset\mathbb R^4$ gives
\[
\overline{\nabla}_{(df_s)_\gamma (\tau_s)}{\bf a}_s
=
\nabla^{\mathbb S^3}_{(df_s)_\gamma (\tau_s)}{\bf a}_s
-
\langle (df_s)_\gamma (\tau_s),(df_s)_\gamma (\eta_s)\rangle (f_s\circ \gamma )
=
\nabla^{\mathbb S^3}_{(df_s)_\gamma (\tau_s)}{\bf a}_s,
\]
where $\sigma_s$ is the second fundamental form of $f_s$ with respect to the unit normal $N_s$. Hence, 
\begin{equation}\label{eq:Gauss-a-s}
\overline{\nabla}_{(df_s)_\gamma (\tau_s)}{\bf a}_s
=
-\kappa_{g,s}\,(df_s)_\gamma (\tau_s)
+
\sigma_s(\tau_s,\eta_s)N_s.
\end{equation}

 At $s=0$, ${\bf a}_0={\bf a}$ is constant along $\gamma $, and
\eqref{eq:adot-Neumann} says that its first variation also vanishes
there. Therefore, $\overline{\nabla}_{df_s(\tau_s)}{\bf a}_s=\mathcal O(s^2)$.
Since the two terms on the right-hand side of
\eqref{eq:Gauss-a-s} are mutually orthogonal, we conclude that
\begin{equation}\label{eq:boundary-coefficients-first-order}
\kappa_{g,s}=\mathcal O(s^2),
\qquad
\sigma_s(\tau_s,\eta_s)=\mathcal O(s^2).
\end{equation}

We now pass to the first-order conjugate data. 
Again, we will use first-order jets, hence the formula~\eqref{fvarepsilon*} for the variation $f_{\varepsilon}^*$ is not relevant in what follows, but only its first-order jet.
Let $g_\varepsilon^*$ denote the metric induced by
$f_\varepsilon^*$. For any $X,Y\in T\Sigma$ we have
\[
g_\varepsilon^*(X,Y)
=
\left\langle
df_\varepsilon^*(X),df_\varepsilon^*(Y)
\right\rangle
\stackrel{\eqref{MCfepsilon*}}{=}
\left\langle
\theta_\varepsilon(X),\theta_\varepsilon(Y)
\right\rangle.
\]
Hence, by
\eqref{thetavarepsilon},
\eqref{eq:R-epsilon-orthogonal},
\eqref{eq:Jepsilon-compatible}, and
\eqref{eq:alphaepsilon-isometric}, we deduce
\begin{equation}\label{eq:metric-comparison-star}
g_\varepsilon^*
=
g_\varepsilon+\mathcal O(\varepsilon^2).
\end{equation}
Here and below, the comparison~\eqref{eq:metric-comparison-star}
is understood in $C^1$ on compact subsets of $\Sigma$. Hence the corresponding Levi--Civita connections of $g_\varepsilon,g_\varepsilon^*$, and therefore their geodesic curvatures along $\gamma $, agree up to second order in $\varepsilon$:
\begin{equation}\label{eq:star-geodesic-curvature}
\kappa_{g,\varepsilon}^*=\kappa_{g,\varepsilon}
+\mathcal O(\varepsilon^2)
\stackrel{\eqref{eq:boundary-coefficients-first-order}}{=}
\mathcal O(\varepsilon^2).
\end{equation}
Now,
\begin{eqnarray}
\sigma_\varepsilon^*(\tau_\varepsilon,\tau_\varepsilon)
&=&
g_\varepsilon^*(A_\varepsilon^*\tau_\varepsilon,\tau_\varepsilon)
\stackrel{\eqref{eq:shape-comparison-star}}{=}
g_\varepsilon^*(
-J_\varepsilon A_\varepsilon\tau_\varepsilon,
\tau_\varepsilon )
+\mathcal O(\varepsilon^2) 
\stackrel{\eqref{eq:metric-comparison-star}}{=}
g_\varepsilon(
-J_\varepsilon A_\varepsilon\tau_\varepsilon,
\tau_\varepsilon )
+\mathcal O(\varepsilon^2)  \nonumber
\\
&\stackrel{\eqref{eq:Jepsilon-skew-adjoint}}{=}&
g_\varepsilon (
A_\varepsilon\tau_\varepsilon,
J_\varepsilon\tau_\varepsilon )
+\mathcal O(\varepsilon^2) 
\stackrel{\eqref{dihedron}}{=}
\sigma_\varepsilon(\tau_\varepsilon,\eta_\varepsilon)
+\mathcal O(\varepsilon^2)
\stackrel{\eqref{eq:boundary-coefficients-first-order}}{=}
\mathcal O(\varepsilon^2),  \label{eq:star-normal-curvature-first-order}
\end{eqnarray}
Next we compute the covariant acceleration
of the curve $f_\varepsilon^*(\gamma )$ with respect to $g_\varepsilon^*$.
To do this, let $\tau_\varepsilon^*:=T/|T|_{g_\varepsilon^*}$, 
$\eta_\varepsilon^*:=J_\varepsilon^*\tau_\varepsilon^*$, be the unit tangent vector 
to $\gamma $ and the corresponding unit conormal 
with respect to $g_\varepsilon^*$. Observe that \eqref{eq:metric-comparison-star} implies that 
\begin{equation}\label{eq:star-unit-frame-comparison}
	\tau_\varepsilon^*
	=
	\tau_\varepsilon+\mathcal O(\varepsilon^2),
	\qquad
	\eta_\varepsilon^*
	=
	\eta_\varepsilon+\mathcal O(\varepsilon^2).
\end{equation}

The Gauss formula for the immersion
$f_\varepsilon^*\colon(\Sigma,g_\varepsilon^*)\to\mathbb S^3$
gives
\[
\nabla^{\mathbb S^3}_{
(df_\varepsilon^*)_\gamma (\tau_\varepsilon^*)}
(df_\varepsilon^*)_\gamma (\tau_\varepsilon^*)
=
(df_\varepsilon^*)_\gamma 
\left(
\nabla^{g_\varepsilon^*}_{\tau_\varepsilon^*}
\tau_\varepsilon^*
\right)
+
\sigma_\varepsilon^*
(\tau_\varepsilon^*,\tau_\varepsilon^*)N_\varepsilon^*
=
\kappa_{g,\varepsilon}^*\,
df_\varepsilon^*(\eta_\varepsilon^*)
+
\sigma_\varepsilon^*
(\tau_\varepsilon^*,\tau_\varepsilon^*)N_\varepsilon^*.
\]
Since both $df_\varepsilon^*(\eta_\varepsilon^*)$ and
$N_\varepsilon^*$ remain bounded as $\varepsilon\to 0$, 
\eqref{eq:star-geodesic-curvature}, \eqref{eq:star-normal-curvature-first-order}, and~\eqref{eq:star-unit-frame-comparison} imply
\begin{equation}\label{eq:formal-geodesic-first-order-unit}
\nabla^{\mathbb S^3}_{
(df_\varepsilon^*)_\gamma (\tau_\varepsilon^*)}
(df_\varepsilon^*)_\gamma (\tau_\varepsilon^*)
=
\mathcal O(\varepsilon^2),
\end{equation}
which is the desired expression of the covariant acceleration
of $f_\varepsilon^*(\gamma )$ with respect to $g_\varepsilon^*$.
Now~\eqref{eq:star-unit-frame-comparison} and~\eqref{eq:formal-geodesic-first-order-unit} ensure that
\begin{equation}\label{eq:formal-geodesic-first-order}
\nabla^{\mathbb S^3}_{
(df_\varepsilon^*)_\gamma (\tau_\varepsilon)}
(df_\varepsilon^*)_\gamma (\tau_\varepsilon)
=
\mathcal O(\varepsilon^2).
\end{equation}

Consider the two-parameter smooth map $F(\varepsilon,t):=f_\varepsilon^*(\gamma (t))$, which is a variation of the 
unit-speed great circle $\Gamma^*(t):=F(0,t)=f^*(\gamma (t))$ with variational field 
\[
Z(t):=\left.\partial_\varepsilon F(\varepsilon,t)\right|_{\varepsilon=0}
=X^*_{\gamma (t)}.
\]
Denote by $W_\varepsilon(t):=\partial_tF(\varepsilon,t)$ the partial derivative of $F$ tangent to each $\varepsilon$-curve. Thus,
\begin{equation}\label{W0}
W_0(t)=\frac{\partial F}{\partial t}(0,t)=df^*_{\gamma (t)}(\gamma '(t))
=(\Gamma^*)'(t).
\end{equation}
Since $\tau_\varepsilon$ is the $g_\varepsilon$-unit tangent
vector to $\gamma $, there exists a positive function $\lambda_\varepsilon=\lambda_\varepsilon(t)$ such that
$\tau_\varepsilon=\lambda_\varepsilon \gamma '$ and thus,
$(df_\varepsilon^*)_\gamma (\tau_\varepsilon)
=\lambda_\varepsilon W_\varepsilon$. Moreover,
$\lambda_0$ is constant equal to one.

Since $\Gamma^*$ is a great circle parameterized by arc length, we have 
\begin{equation}\label{gamma*geod}
\nabla^{\mathbb S^3}_{(\Gamma^*)'} (\Gamma^*)'=0.
\end{equation}
Equation~\eqref{eq:formal-geodesic-first-order} now reads
\[
\mathcal O(\varepsilon^2)
=\nabla^{\mathbb S^3}_{\lambda_\varepsilon W_\varepsilon}
(\lambda_\varepsilon W_\varepsilon)
=
\lambda_\varepsilon W_\varepsilon(\lambda_\varepsilon)W_\varepsilon
+
\lambda_\varepsilon^2
\nabla^{\mathbb S^3}_{W_\varepsilon}W_\varepsilon
=
\lambda_\varepsilon \frac{\partial \lambda_\varepsilon}{\partial t} W_\varepsilon
+
\lambda_\varepsilon^2
\nabla^{\mathbb S^3}_{W_\varepsilon}W_\varepsilon.
\]
Differentiating the last equality at $\varepsilon=0$ gives
\begin{eqnarray}
0&=&
\dot{\lambda}\frac{\partial \lambda_0}{\partial t}W_0
+
\lambda_0\frac{\partial \dot\lambda}{\partial t}W_0
+
\lambda_0\frac{\partial \lambda_0}{\partial t}\frac{\partial W_\varepsilon}{\partial \varepsilon}(0,t)
+
2\lambda_0 \dot\lambda \nabla_{W_0}^{\mathbb{S}^3}W_0
+
\lambda_0^2 \left. \frac{\partial}{\partial \varepsilon}\right|_{\varepsilon=0}
\nabla_{W_\varepsilon}^{\mathbb{S}^3}W_\varepsilon \nonumber
\\
&\stackrel{(\star)}{=}&
\frac{\partial \dot\lambda}{\partial t}W_0
+
2\dot\lambda \nabla_{W_0}^{\mathbb{S}^3}W_0
+
\left. \frac{\partial}{\partial \varepsilon}\right|_{\varepsilon=0}
\nabla_{W_\varepsilon}^{\mathbb{S}^3}W_\varepsilon \nonumber
\\
&\stackrel{\eqref{W0}}{=}&
\frac{\partial \dot\lambda}{\partial t}(\Gamma ^*)'
+
2\dot\lambda \nabla_{(\Gamma ^*)'}^{\mathbb{S}^3}(\Gamma ^*)'
+
\left. \frac{\partial}{\partial \varepsilon}\right|_{\varepsilon=0}
\nabla_{W_\varepsilon}^{\mathbb{S}^3}W_\varepsilon 
\stackrel{\eqref{gamma*geod}}{=}
\frac{\partial \dot\lambda}{\partial t}(\Gamma ^*)'
+
\left. \frac{\partial}{\partial \varepsilon}\right|_{\varepsilon=0}
\nabla_{W_\varepsilon}^{\mathbb{S}^3}W_\varepsilon,\qquad 
\label{eq:linearized-geodesic-equation}
\end{eqnarray}
where in $(\star)$ we have used that $\lambda_0$ is constant one.

In the computations below, covariant derivatives in the direction
$Z$ are understood through the parameterized surface $F$; more precisely, if $V_\varepsilon(t)$ is a smooth vector field
along the longitudinal curves $F(\varepsilon,\cdot)$, we use the notation:
\[
\nabla_Z^{\mathbb S^3}V_0
:=
\left.
\nabla_{\partial_\varepsilon F}^{\mathbb S^3}
V_\varepsilon
\right|_{\varepsilon=0}.
\]
In particular, $W_\varepsilon=\partial_tF$ provides the extension
through $F$ of $W_0=(\Gamma ^*)'$, and hence
\[
\nabla_Z^{\mathbb S^3}(\Gamma ^*)'
:=
\left.
\nabla_{\partial_\varepsilon F}^{\mathbb S^3}
W_\varepsilon
\right|_{\varepsilon=0}.
\]
Thus, every covariant derivative in the $Z$-direction appearing
below is well defined, and no arbitrary extension of $Z$ or $(\Gamma ^*)'$ to a neighborhood in $\mathbb S^3$ is involved.

By the Gauss formula of $\mathbb{S}^3$ in $\mathbb{R}^4$, the last term in~\eqref{eq:linearized-geodesic-equation} can be written as follows:
\[
\left. \frac{\partial}{\partial \varepsilon}\right|_{\varepsilon=0}
\nabla_{W_\varepsilon}^{\mathbb{S}^3}W_\varepsilon
=
\left. \overline{\nabla}_Z \nabla_{W_\varepsilon}^{\mathbb{S}^3}W_\varepsilon
\right|_{\varepsilon=0}
=
\nabla_{Z}^{\mathbb{S}^3} \nabla_{(\Gamma ^*)'}^{\mathbb{S}^3}(\Gamma ^*)'
-
\langle Z,\nabla^{\mathbb S^3}_{(\Gamma ^*)'} (\Gamma ^*)'\rangle f^*
\stackrel{\eqref{gamma*geod}}{=}
\nabla_{Z}^{\mathbb{S}^3} \nabla_{(\Gamma ^*)'}^{\mathbb{S}^3}(\Gamma ^*)'.
\]
The usual commutation formula for first and second covariant derivatives in a parameterized surface allows us to write
\begin{eqnarray*}
\nabla_{Z}^{\mathbb{S}^3}(\Gamma ^*)'
&=&
\nabla_{(\Gamma ^*)'}^{\mathbb{S}^3} Z,
\\
\nabla_{Z}^{\mathbb{S}^3} \nabla_{(\Gamma ^*)'}^{\mathbb{S}^3}(\Gamma ^*)'
&=&
\nabla_{(\Gamma ^*)'}^{\mathbb{S}^3} \nabla_{Z}^{\mathbb{S}^3}(\Gamma ^*)'
+
R^{\mathbb S^3}(Z,(\Gamma ^*)')(\Gamma ^*)',
\end{eqnarray*}
where $R^{\mathbb S^3}$ stands for the Riemannian curvature tensor of $\mathbb{S}^3$. Hence,
\[
R^{\mathbb S^3}(Z,(\Gamma ^*)')(\Gamma ^*)'
=
|(\Gamma ^*)'|^2 Z-\langle Z,(\Gamma ^*)'\rangle (\Gamma ^*)'
=
Z-\langle Z,(\Gamma ^*)'\rangle (\Gamma ^*)',
\]
and \eqref{eq:linearized-geodesic-equation} becomes
\begin{equation}\label{eq:linearized-geodesic-equation-expanded}
0=\frac{\partial \dot\lambda}{\partial t}(\Gamma ^*)'
+
\nabla_{(\Gamma ^*)'}^{\mathbb S^3}\nabla_{(\Gamma ^*)'}^{\mathbb S^3}Z
+
Z-\langle Z,(\Gamma ^*)'\rangle (\Gamma ^*)'.
\end{equation}
The first and fourth terms in~\eqref{eq:linearized-geodesic-equation-expanded}
are tangent to the great circle $f^*(\gamma )$, hence they disappear after projecting onto the normal bundle of $\Gamma ^*$ in $\mathbb S^3$:
\begin{equation*}
0=
\left( \nabla_{(\Gamma ^*)'}^{\mathbb S^3}\nabla_{(\Gamma ^*)'}^{\mathbb S^3}Z
+Z\right)^{\perp}
=
\nabla_{(\Gamma ^*)'}^{\mathbb S^3}\nabla_{(\Gamma ^*)'}^{\mathbb S^3}Z^{\perp}
+Z^\perp.
\end{equation*}
This is precisely the Jacobi equation for $Z^{\perp}$ along the geodesic
$\Gamma ^*=f^*(\gamma )$, which finishes the proof of the claim.
\par\vspace{.2cm}
We continue with the proof of Theorem~\ref{thm:N-to-D}. Recall that 
$\Pi:=\operatorname{Span}\{{\bf p},{\bf v}\}$. Thus, the normal bundle in $\mathbb{S}^3$ of the great circle $\Gamma ^*$ is the constant plane $\Pi^{\perp}$. Since $Z^{\perp}$ is a Jacobi field along $\Gamma ^*$, 
there exist vectors $A_0,B_0\in\Pi^\perp$ such that
\begin{equation}\label{eq:Z-perp}
Z^\perp(t)=\cos t\,A_0+\sin t\,B_0.
\end{equation}
Define $\Lambda_0\in\mathfrak{so}(4)$ by
\begin{equation}\label{eq:Lambda-zero}
 \Lambda_0x=
 \langle{\bf p},x\rangle A_0-\langle A_0,x\rangle{\bf p}
 +\langle{\bf v},x\rangle B_0-\langle B_0,x\rangle{\bf v},
\end{equation}
which clearly satisfies $\Lambda_0{\bf p}=A_0$ and $\Lambda_0{\bf v}=B_0$. 
Thus, the orthogonal projection of $\Lambda_0 f^*(\gamma (t))$ on $\Pi^{\perp}$ is equal to
\[
 \pi_{\Pi^\perp}\bigl(\Lambda_0 f^*(\gamma (t))\bigr)
=
\pi_{\Pi^\perp}\bigl(\cos t\, \Lambda _0{\bf p}+\sin t\, \Lambda_0{\bf v}\bigr)
=
\cos t\,A_0+\sin t\,B_0
\stackrel{\eqref{eq:Z-perp}}{=}
Z^\perp(t).
\]
Since $N^*|_{\gamma}$ takes values in $\Pi^\perp$,
$ \pi_{\Pi^\perp}(N^*)=N^*$ along $\gamma$ and
\[
u_{\Lambda_0}^*|_\gamma
=\langle\Lambda_0 f^*,N^*\rangle|_\gamma
=
\langle \pi_{\Pi^\perp}(\Lambda_0 f^*),N^*\rangle |_\gamma
=
\langle Z^\perp,N^*|_\gamma\rangle
 =\langle X^*,N^*\rangle|_\gamma
 =u^*|_\gamma,
\]
which proves~\eqref{eq:N-to-D-conclusion}.

Finally, every element of $\mathfrak{so}(\Pi)\oplus\mathfrak{so}(\Pi^\perp)$
preserves the splitting $\mathbb R^4=\Pi\oplus\Pi^\perp$. Since
$f^*(\gamma (t))\in\Pi$ and $N^*(\gamma (t))\in\Pi^\perp$, adding such an
element to $\Lambda_0$ does not change the function
$\langle\Lambda_0 f^*,N^*\rangle|_\gamma$.
Thus every matrix in the affine two-plane
\eqref{eq:affine-gauge-space} satisfies
\eqref{eq:N-to-D-conclusion}.
\end{proof}

\begin{remark}\label{rem:scalar-boundary-stabilizer}
{\rm The affine two-plane in~\eqref{eq:affine-gauge-space} is a sufficient family of corrections. If
\[
\mathcal T_\gamma\colon\mathfrak{so}(4)\to  C^\infty(\gamma),
\qquad
\mathcal T_\gamma(\Lambda)=\langle\Lambda f^*,N^*\rangle|_\gamma,
\]
then the complete set of solutions of~\eqref{eq:N-to-D-conclusion} is $\Lambda_0+\ker\mathcal T_\gamma$. The proof gives
$\mathfrak{so}(\Pi)\oplus\mathfrak{so}(\Pi^\perp)\subseteq\ker\mathcal T_\gamma$, but does not assert equality.}
\end{remark}

\section{Extension to the associated family}
\label{sec:associated-family}
The conjugate immersion is the member of the associated family corresponding to a rotation by $\pi/2$ of the second fundamental form. In this section we first show that the Frobenius formulation of Section~\ref{secconjminS3}, as well as its linearization in Section~\ref{secJacconj}, extends to every member of the associated family. We keep the notation of those sections. Thus, $f\colon(\Sigma,g)\to\mathbb S^3$ is an oriented immersion, $\alpha=f^{-1}\cdot df$, 
$N_1=f^{-1}\cdot N$, and $\phi(v)w=v\times w$ for $v,w\in\Im(\mathbb H)$.

For $\theta\in\mathbb R$, set $P_\theta:=\cos\theta\,I+\sin\theta\,J$. Since $J^2=-I$, the endomorphism $P_\theta$ is the rotation through the angle $\theta$ in every oriented tangent plane, and $P_\theta^{-1}=P_{-\theta}$. 

Define the $\Im(\mathbb H)$-valued one-form $\beta_\theta$ and the $\mathfrak{so}(3)$-valued one-form $\mathcal A_\theta$ by
\begin{equation*}
\beta_\theta:=\alpha\circ P_\theta,
\qquad
\mathcal A_\theta
:=
\phi\circ\alpha\circ(P_\theta-I).
\end{equation*}
For $\theta=\pi/2$, these are precisely $\beta_{\pi/2}=\alpha\circ J$, $\mathcal A_{\pi/2}=\phi\circ\alpha\circ(J-I)$, which are the forms used in Section~\ref{secconjminS3}.

\begin{lemma}\label{lem:associated-curvature}
Let $f\colon(\Sigma,g)\to\mathbb S^3$ be an oriented immersion, not necessarily minimal, and let $H$ be its mean curvature function. Then,
\begin{eqnarray}
d\mathcal A_\theta-
\mathcal A_\theta\wedge\mathcal A_\theta
&=&
-2H\sin\theta\,\phi(N_1)\,da, \label{eq:curvature-A-theta}
\\
d\beta_\theta-
\mathcal A_\theta\wedge\beta_\theta
&=&
-2(1+H\sin\theta)N_1\,da.  \label{eq:beta-theta-identity}
\end{eqnarray}
Here, as in Section~\ref{secconjminS3}, $da$ denotes the area form of $g$.
\end{lemma}

\begin{proof}
Write $c=\cos\theta$ and $s=\sin\theta$, and choose a positively oriented local 
orthonormal frame $\{E_1,E_2=JE_1\}$. Define the smooth functions $a,b$ as 
in~\eqref{ab}, so that $a\times b=N_1$. Then,
\begin{equation}\label{eq:d-Q-theta}
d\bigl(\alpha\circ(P_\theta-I)\bigr) 
=(c-1)d\alpha+s\,d(\alpha\circ J)
\stackrel{\eqref{dalpha},\eqref{dalphaJ}}{=}
2(1-c-Hs)N_1\,da.
\end{equation}
Moreover, $\alpha((P_\theta-I)E_1)=(c-1)a+sb$ and 
$\alpha((P_\theta-I)E_2)=-sa+(c-1)b$, and hence
\begin{equation}\label{eq:cross-Q-theta}
\alpha((P_\theta-I)E_1)
\times
\alpha((P_\theta-I)E_2)
=
\bigl((c-1)^2+s^2\bigr)N_1
=
2(1-c)N_1.
\end{equation}
Using the identity
$[\phi(v),\phi(w)]=\phi(v\times w)$, equations~\eqref{eq:d-Q-theta} and~\eqref{eq:cross-Q-theta} give
\[
\bigl(d\mathcal A_\theta-
\mathcal A_\theta\wedge\mathcal A_\theta\bigr)(E_1,E_2)
=
\phi\bigl(2(1-c-Hs)N_1-2(1-c)N_1\bigr)
=-2Hs\,\phi(N_1),
\]
which proves~\eqref{eq:curvature-A-theta}.

Similarly,
\[
d\beta_\theta
=
c\,d\alpha+s\,d(\alpha\circ J)
\stackrel{\eqref{dalpha},\eqref{dalphaJ}}{=}
-2(c+Hs)N_1\,da.
\]
On the chosen frame, $\beta_\theta(E_1)=ca+sb$, $\beta_\theta(E_2)=-sa+cb$. A direct calculation, using $\phi(v)w=v\times w$, yields
\(
(\mathcal A_\theta\wedge\beta_\theta)(E_1,E_2)
=
2(1-c)N_1.
\)
Combining the last two displayed identities gives
\[
(d\beta_\theta-
\mathcal A_\theta\wedge\beta_\theta)(E_1,E_2)
=
-2(1+Hs)N_1,
\]
and proves~\eqref{eq:beta-theta-identity}.
\end{proof}

We can now recover the entire associated family by a Frobenius system.

\begin{theorem}[Frobenius formulation of the associated family]\label{thm:Frobenius-associated-family}
Let $f\colon(\Sigma,g)\to\mathbb S^3$ be a minimal immersion of a simply connected oriented surface. Fix $\theta\in\mathbb R$. For every choice of $R_\theta(p_0)\in SO(3)$, the system
\begin{equation}\label{eq:R-theta-system}
dR_\theta=\mathcal A_\theta R_\theta
\end{equation}
has a unique solution $R_\theta\colon\Sigma\to SO(3)$. The one-form
\begin{equation}\label{eq:alpha-theta}
\alpha_\theta
:=
R_\theta^{\mathsf T}\beta_\theta
=
R_\theta^{\mathsf T}\alpha\circ P_\theta
\end{equation}
satisfies
\begin{equation}\label{eq:MC-alpha-theta}
d\alpha_\theta+\alpha_\theta\wedge\alpha_\theta=0.
\end{equation}
Therefore, for every prescribed value $f_\theta(p_0)\in\mathbb S^3$, there is a unique immersion $f_\theta\colon\Sigma\to\mathbb S^3$ satisfying
\begin{equation}\label{eq:f-theta-Frobenius}
f_\theta^{-1}\cdot df_\theta=\alpha_\theta.
\end{equation}
The immersion $f_\theta$ induces the metric $g$, its left-translated unit normal is
\begin{equation}\label{eq:N-theta}
f_\theta^{-1}\cdot N_\theta
=
R_\theta^{\mathsf T}N_1,
\end{equation}
and its shape operator is
\begin{equation}\label{eq:shape-associated-family}
A_\theta
=
P_\theta^{-1}A
=
\cos\theta\,A-\sin\theta\,JA.
\end{equation}
In particular, $f_\theta$ is minimal, $|A_\theta|=|A|$, $\mathcal L_{f_\theta}=\mathcal L_f$, and, up to the choices of the initial values, $\{f_\theta\}_{\theta\in\mathbb R}$ is the associated family of
$f$, with $f_0=f$ and $f_{\pi/2}=f^*$.
\end{theorem}

\begin{proof}
Since $H=0$, equation~\eqref{eq:curvature-A-theta} says that $\mathcal A_\theta$ is flat. Therefore, equation~\eqref{eq:R-theta-system} is compatible and has a unique global solution on the simply connected surface $\Sigma$ once an initial value $R_\theta(p_0)\in SO(3)$ has been chosen. Since $\mathcal A_\theta$ takes values in $\mathfrak{so}(3)$, the values of such a solution $R$ remain in $SO(3)$.

From $dR_\theta^{\mathsf T}=-R_\theta^{\mathsf T}\mathcal A_\theta$, equations~\eqref{eq:alpha-theta} and~\eqref{eq:beta-theta-identity} give
\begin{equation}\label{eq:d-alpha-theta}
d\alpha_\theta
=
R_\theta^{\mathsf T}
\bigl(d\beta_\theta-
\mathcal A_\theta\wedge\beta_\theta\bigr)
=
-2R_\theta^{\mathsf T}N_1\,da.
\end{equation}
On the other hand, because both $P_\theta$ and $R_\theta$ preserve orientation and the metric,
\[
(\alpha_\theta\wedge\alpha_\theta)(E_1,E_2)
=
2\alpha_\theta(E_1)\times\alpha_\theta(E_2)
=
2R_\theta^{\mathsf T} \bigl(\alpha(P_\theta E_1)\times \alpha(P_\theta E_2)\bigr)
=
2R_\theta^{\mathsf T}N_1.
\]
Together with~\eqref{eq:d-alpha-theta}, this proves~\eqref{eq:MC-alpha-theta}. The flat system $df_\theta=f_\theta\cdot\alpha_\theta$ can therefore be integrated along paths. Since $\Sigma$ is simply connected, the result is independent of the path and gives the unique solution with prescribed value $f_\theta(p_0)\in\mathbb S^3$; see also~\cite{konu}. Moreover,
\[
\langle\alpha_\theta(X),\alpha_\theta(Y)\rangle
=
\langle\alpha(P_\theta X),\alpha(P_\theta Y)\rangle
=
g(P_\theta X,P_\theta Y)
=
g(X,Y),
\]
so $f_\theta$ is an isometric immersion. Equation~\eqref{eq:N-theta} follows immediately from orientation and orthogonality.

It remains to compute the shape operator. Put $N_{1,\theta}:=R_\theta^{\mathsf T}N_1$. Using the left-trivialized Weingarten equation~\eqref{eq:Weingarten-fs} at $s=0$, we obtain
\begin{equation*}
dN_{1,\theta}(X)+\alpha_\theta(X)\times N_{1,\theta}
=
R_\theta^{\mathsf T}
\bigl(-\mathcal A_\theta(X)N_1+dN_1(X)
+\beta_\theta(X)\times N_1\bigr).
\end{equation*}
By definition of $\mathcal A_\theta$ and the identity $\phi(v)w=v\times w$, we have
\[
-\mathcal A_\theta(X)N_1 = -\alpha\bigl((P_\theta-I)X\bigr)\times N_1.
\]
Moreover,
\(
\beta_\theta(X)
=
\alpha(P_\theta X)
=
\alpha(X)+\alpha\bigl((P_\theta-I)X\bigr),
\)
while the left-trivialized Weingarten equation for \(f\) gives
\(
dN_1(X)
=
-\alpha(AX)-\alpha(X)\times N_1.
\)
Consequently,
\begin{eqnarray*}
-\mathcal A_\theta(X)N_1+dN_1(X)+\beta_\theta(X)\times N_1
&=&
-\alpha\bigl((P_\theta-I)X\bigr)\times N_1
-\alpha(AX)-\alpha(X)\times N_1 
\\
& & 
+\alpha(X)\times N_1
+\alpha\bigl((P_\theta-I)X\bigr)\times N_1
\\
&=&
-\alpha(AX),
\end{eqnarray*}
and thus,
\begin{equation*}
dN_{1,\theta}(X)+\alpha_\theta(X)\times N_{1,\theta}
=
-R_\theta^{\mathsf T}
\alpha(AX).
\end{equation*}
Comparing the last identity with the left-trivialized Weingarten equation for $f_\theta$,
\[
dN_{1,\theta}(X)+\alpha_\theta(X)\times N_{1,\theta}
=
-\alpha_\theta(A_\theta X),
\]
and using~\eqref{eq:alpha-theta}, we find
\(
\alpha(P_\theta A_\theta X)=\alpha(AX).
\)
Since $\alpha$ is injective, $P_\theta A_\theta=A$, which proves~\eqref{eq:shape-associated-family}. Finally, $A$ is traceless and $\operatorname{trace}(JA)=0$, so $\operatorname{trace}(A_\theta)=0$. Since $P_\theta$ is an isometry, $|A_\theta|=|A|$, and therefore $\mathcal L_{f_\theta}=\mathcal L_f$. The assertions for $\theta=0$ and $\theta=\pi/2$ follow from $P_0=I$ and $P_{\pi/2}=J$, after choosing the corresponding initial conditions.
\end{proof}

\subsection{Jacobi functions along the associated family}

We now linearize Theorem~\ref{thm:Frobenius-associated-family}. Let $u$ be a Jacobi function for $f$, and choose, as in Section~\ref{secJacconj}, an arbitrary normal variation $f_s$ with variational field $uN$. The immersions $f_s$ are not assumed to be minimal for $s\ne0$. We use the notation $g_s,J_s,\alpha_s,N_{1,s},H_s$ introduced there, and put
\begin{equation*}
P_{\theta,s}:=\cos\theta\,I+\sin\theta\,J_s,
\qquad
\beta_{\theta,s}:=\alpha_s\circ P_{\theta,s},
\qquad
\mathcal A_{\theta,s}
:=
\phi\circ\alpha_s\circ(P_{\theta,s}-I).
\end{equation*}
At $s=0$, differentiation gives
\begin{equation*}
\begin{split}
\dot P_\theta&=\sin\theta\,\dot J,\\
\dot\beta_\theta
&=
\dot\alpha\circ P_\theta
+\sin\theta\,\alpha\circ\dot J,\\
\dot{\mathcal A}_\theta
&=
\phi\circ
\left(
\dot\alpha\circ(P_\theta-I)
+\sin\theta\,\alpha\circ\dot J
\right).
\end{split}
\end{equation*}

\begin{theorem}[Jacobi functions along the associated family]\label{thm:Jacobi-associated-family}
Let $f\colon(\Sigma,g)\to\mathbb S^3$ be a minimal immersion of a simply connected surface, let $u$ be a Jacobi function for $f$, and let $f_\theta$ and $R_\theta$ be given by Theorem~\ref{thm:Frobenius-associated-family}. Then:
\begin{enumerate}
\item The $\mathfrak{so}(3)$-valued one-form $R_\theta^{\mathsf T}\dot{\mathcal A}_\theta R_\theta$
is closed. Hence, for every $\Omega_{\theta,0}\in\mathfrak{so}(3)$, the system
\begin{equation}\label{eq:Omega-theta-system}
d\Omega_\theta
=
R_\theta^{\mathsf T}
\dot{\mathcal A}_\theta R_\theta,
\qquad
\Omega_\theta(p_0)=\Omega_{\theta,0},
\end{equation}
has a unique global solution on $\Sigma$.

\item The one-form
\begin{equation*}
\dot\alpha_\theta
:=
-\Omega_\theta\alpha_\theta
+R_\theta^{\mathsf T}\dot\beta_\theta
\end{equation*}
satisfies the linearized Maurer--Cartan equation
\begin{equation}\label{eq:linearized-MC-theta}
d\dot\alpha_\theta
+\alpha_\theta\wedge\dot\alpha_\theta
+\dot\alpha_\theta\wedge\alpha_\theta
=0.
\end{equation}

\item For every $q_{\theta,0}\in\Im(\mathbb H)$, the initial-value problem
\begin{equation}\label{eq:q-theta-system}
dq_\theta
+\alpha_\theta\cdot q_\theta
-q_\theta\cdot\alpha_\theta
=
\dot\alpha_\theta,
\qquad
q_\theta(p_0)=q_{\theta,0},
\end{equation}
has a unique global solution $q_\theta\colon\Sigma\to\Im(\mathbb H)$. The smooth vector field $X_\theta:=f_\theta\cdot q_\theta$ is a vector field along $f_\theta$, and its normal component $u_\theta:=\langle X_\theta,N_\theta\rangle$ is a Jacobi function for $f_\theta$.
\end{enumerate}
Changing the initial values in~\eqref{eq:Omega-theta-system} and~\eqref{eq:q-theta-system} changes $u_\theta$ by a Jacobi--Killing function of $f_\theta$. Thus, $u_\theta$ is naturally defined modulo Jacobi--Killing functions. For $\theta=\pi/2$, the construction is exactly the Jacobi conjugation of Theorem~\ref{propos8.9} and Definition~\ref{def:Jacobi-conjugate}.
\end{theorem}
\begin{proof}
Applying~\eqref{eq:curvature-A-theta} to $f_s$ gives
\[
d\mathcal A_{\theta,s}
-\mathcal A_{\theta,s}\wedge\mathcal A_{\theta,s}
=
-2H_s\sin\theta\,
\phi(N_{1,s})\,da_s.
\]
Since $H_0=0$ and $2\dot H=\mathcal L_fu$, differentiation at $s=0$ yields
\begin{equation*}
d\dot{\mathcal A}_\theta
-\dot{\mathcal A}_\theta\wedge\mathcal A_\theta
-\mathcal A_\theta\wedge\dot{\mathcal A}_\theta
=
-\sin\theta\,(\mathcal L_fu)\phi(N_1)\,da
=0.
\end{equation*}
Using $dR_\theta=\mathcal A_\theta R_\theta$ and
$dR_\theta^{\mathsf T}=-R_\theta^{\mathsf T}\mathcal A_\theta$, we therefore obtain
\[
d\left(
R_\theta^{\mathsf T}
\dot{\mathcal A}_\theta R_\theta
\right)
=
R_\theta^{\mathsf T}
\left(
d\dot{\mathcal A}_\theta
-\dot{\mathcal A}_\theta\wedge\mathcal A_\theta
-\mathcal A_\theta\wedge\dot{\mathcal A}_\theta
\right)R_\theta
=0.
\]
This proves the first assertion.

Set
\[
\dot R_\theta:=R_\theta\Omega_\theta,
\qquad
R_{\theta,\varepsilon}:=R_\theta+\varepsilon\dot R_\theta.
\]
As in equations~\eqref{eq:R-epsilon-orthogonal}--\eqref{eq80},
\begin{equation}\label{eq:R-theta-first-jet}
R_{\theta,\varepsilon}^{\mathsf T}R_{\theta,\varepsilon}
=I+\mathcal O(\varepsilon^2),
\qquad
dR_{\theta,\varepsilon}
=
\mathcal A_{\theta,\varepsilon}R_{\theta,\varepsilon}
+\mathcal O(\varepsilon^2),
\end{equation}
where $\mathcal A_{\theta,\varepsilon}:=\mathcal A_\theta+\varepsilon\dot{\mathcal A}_\theta$. If
\begin{equation}\label{eq163}
\beta_{\theta,\varepsilon}
:=
\beta_\theta+\varepsilon\dot\beta_\theta,
\qquad
\alpha_{\theta,\varepsilon}
:=
R_{\theta,\varepsilon}^{\mathsf T}
\beta_{\theta,\varepsilon},
\end{equation}
then
\(
\alpha_{\theta,\varepsilon}
=
\alpha_\theta+\varepsilon\dot\alpha_\theta
+\mathcal O(\varepsilon^2).
\)

We now verify the Maurer--Cartan equation for \(\alpha_{\theta,\varepsilon}\) to first order. On the one hand, by~\eqref{eq:R-theta-first-jet} we have
\(
dR_{\theta,\varepsilon}^{\mathsf T}
=
-
R_{\theta,\varepsilon}^{\mathsf T}
\mathcal A_{\theta,\varepsilon}
+
\mathcal O(\varepsilon^2).
\)
Consequently,
\begin{equation}
d\alpha_{\theta,\varepsilon}
\stackrel{\eqref{eq163}}{=}
dR_{\theta,\varepsilon}^{\mathsf T}
\wedge\beta_{\theta,\varepsilon}
+
R_{\theta,\varepsilon}^{\mathsf T}
d\beta_{\theta,\varepsilon}
=
R_{\theta,\varepsilon}^{\mathsf T}
\bigl(
d\beta_{\theta,\varepsilon}
-
\mathcal A_{\theta,\varepsilon}
\wedge\beta_{\theta,\varepsilon}
\bigr)
+
\mathcal O(\varepsilon^2).
\label{eq:d-alpha-theta-epsilon}
\end{equation}

Taking first jets in
\eqref{eq:beta-theta-identity} gives
\[
d\beta_{\theta,\varepsilon}
-
\mathcal A_{\theta,\varepsilon}
\wedge\beta_{\theta,\varepsilon}
=
-2\bigl(1+H_\varepsilon\sin\theta\bigr)
N_{1,\varepsilon}\,da_\varepsilon
+
\mathcal O(\varepsilon^2),
\]
where \(da_\varepsilon=da+\varepsilon\dot{da}\) denotes the
first jet of the area form. By \eqref{alphaJAepsilon}, $H_\varepsilon=H+\varepsilon\dot H=0$.
Thus, \eqref{eq:d-alpha-theta-epsilon} becomes
\begin{equation}
	d\alpha_{\theta,\varepsilon}
	=
	-2R_{\theta,\varepsilon}^{\mathsf T}
	N_{1,\varepsilon}\,da_\varepsilon
	+
	\mathcal O(\varepsilon^2).
	\label{eq:d-alpha-theta-epsilon-final}
\end{equation}

On the other hand, for $X,Y\in T\Sigma$ we have
\begin{eqnarray}
(\alpha_{\theta,\varepsilon}\wedge
\alpha_{\theta,\varepsilon})(X,Y)
&=&
2\alpha_{\theta,\varepsilon}(X)
\times
\alpha_{\theta,\varepsilon}(Y)
\stackrel{\eqref{eq163}}{=}
2 \bigl(
R_{\theta,\varepsilon}^{\mathsf T}
\beta_{\theta,\varepsilon}(X)
\bigr)
\times
\bigl(
R_{\theta,\varepsilon}^{\mathsf T}
\beta_{\theta,\varepsilon}(Y)
\bigr) \nonumber
\\
&\stackrel{(\star)}{=}&
2 R_{\theta,\varepsilon}^{\mathsf T}
\bigl(
\beta_{\theta,\varepsilon}(X)
\times
\beta_{\theta,\varepsilon}(Y)
\bigr)
+
\mathcal O(\varepsilon^2),  \label{eq:alpha-theta-cross-product-first-jet}
\end{eqnarray}
where in $(\star)$ we have used that $R_{\theta,\varepsilon}$ is orientation-preserving and orthogonal up to terms of order $\varepsilon^2$, which implies that it preserves the Euclidean cross product to the same order.

For every $s$, the endomorphism $P_{\theta,s}:=\cos\theta\,I+\sin\theta\,J_s$ is an orientation-preserving isometry of $(T\Sigma,g_s)$. Since $\beta_{\theta,s}=\alpha_s\circ P_{\theta,s}$, it follows that
\[
\beta_{\theta,s}(X)\times\beta_{\theta,s}(Y)
=
N_{1,s}\,da_s(X,Y).
\]
Taking first jets yields
\(
\beta_{\theta,\varepsilon}(X)
\times
\beta_{\theta,\varepsilon}(Y)
=
N_{1,\varepsilon}\,
da_\varepsilon(X,Y)
+
\mathcal O(\varepsilon^2).
\)
Combining this identity with \eqref{eq:alpha-theta-cross-product-first-jet}, we obtain
\begin{equation}
\alpha_{\theta,\varepsilon}
\wedge\alpha_{\theta,\varepsilon}
=
2R_{\theta,\varepsilon}^{\mathsf T}
N_{1,\varepsilon}\,da_\varepsilon
+
\mathcal O(\varepsilon^2).
\label{eq:alpha-theta-wedge}
\end{equation}

Finally, adding
\eqref{eq:d-alpha-theta-epsilon-final} and
\eqref{eq:alpha-theta-wedge} we get $d\alpha_{\theta,\varepsilon}
+
\alpha_{\theta,\varepsilon}
\wedge\alpha_{\theta,\varepsilon}
=
\mathcal O(\varepsilon^2)$, as desired.
Its coefficient of order one is precisely~\eqref{eq:linearized-MC-theta}.

Equation~\eqref{eq:linearized-MC-theta} is the compatibility condition for~\eqref{eq:q-theta-system};
to make the compatibility condition explicit, define $D_{\alpha_\theta}\colon
\Omega^0(\Sigma,\Im(\mathbb H)) \to \Omega^1(\Sigma,\Im(\mathbb H))$ by
\[
D_{\alpha_\theta}q
:=
dq+\alpha_\theta\cdot q-q\cdot\alpha_\theta.
\]
Thus, equation~\eqref{eq:q-theta-system} can be written as
\begin{equation}\label{eq:q-theta-covariant}
D_{\alpha_\theta}q_\theta=\dot\alpha_\theta.
\end{equation}
As in the proof of Theorem~\ref{propos8.9}, we extend
$D_{\alpha_\theta}$ to $\Im(\mathbb H)$-valued one-forms by setting
\[
D_{\alpha_\theta}\eta
:=
d\eta
+\alpha_\theta\wedge\eta
+\eta\wedge\alpha_\theta,
\qquad
\eta\in
\Omega^1(\Sigma,\Im(\mathbb H)).
\]
The curvature $R^{D_{\alpha_{\theta}}}$ of this connection is given by
\[
R^{D_{\alpha_{\theta}}}(X,Y)q
=
\bigl[
(d\alpha_\theta+
\alpha_\theta\wedge\alpha_\theta)(X,Y),
q
\bigr].
\]

Since \(\alpha_\theta\) satisfies the Maurer--Cartan equation~\eqref{eq:MC-alpha-theta}, the connection
$D_{\alpha_\theta}$ is flat. Consequently, $D_{\alpha_\theta}^{\,2}q=0$
for every $q\in C^\infty(\Sigma,\operatorname{Im}\mathbb H)$.
Applying $D_{\alpha_\theta}$ to both sides of
\eqref{eq:q-theta-covariant}, we obtain the necessary compatibility
condition
\[
0
=
D_{\alpha_\theta}^{\,2}q_\theta
=
D_{\alpha_\theta}\dot\alpha_\theta
=
d\dot\alpha_\theta
+\alpha_\theta\wedge\dot\alpha_\theta
+\dot\alpha_\theta\wedge\alpha_\theta,
\]
which is~\eqref{eq:linearized-MC-theta}.

Since $\Sigma$ is simply connected and $D_{\alpha_\theta}$ is
flat, this compatibility condition is also sufficient. Hence,
for every prescribed value $q_\theta(p_0)=q_{\theta,0}\in \Im(\mathbb H)$, the PDE system \eqref{eq:q-theta-system} has a unique global solution $q_\theta$, which lies in $\Im(\mathbb H)$ after taking real parts in~\eqref{eq:q-theta-system}, i.e., 
$q_\theta\colon\Sigma\to \Im(\mathbb H)$.

It remains to show that $u_\theta:=\langle X_\theta,N_\theta\rangle$ is Jacobi for the immersion $f_{\theta}$. First realize the vector field $X_\theta=f_\theta\cdot q_\theta$ by the variation $f_{\theta,\varepsilon}:=f_\theta\cdot\exp(\varepsilon q_\theta)$,
where $\exp $ denotes the usual exponential map in $\mathbb{S}^3$
(i.e., define as in~\eqref{fvarepsilon*} with $f^*,f_{\varepsilon}^*,q$ replaced by $f_\theta,f_{\theta,\varepsilon},q_\theta$ respectively).
If
\(
\vartheta_{\theta,\varepsilon}
:=
f_{\theta,\varepsilon}^{-1}\cdot
df_{\theta,\varepsilon},
\)
then the same expansion as in~\eqref{thetavarepsilon} gives
\[
\vartheta_{\theta,\varepsilon}
=
\alpha_\theta
+\varepsilon\bigl(
dq_\theta+\alpha_\theta\cdot q_\theta
-q_\theta\cdot\alpha_\theta
\bigr)
+\mathcal O(\varepsilon^2)
=
\alpha_\theta+\varepsilon\dot\alpha_\theta
+\mathcal O(\varepsilon^2)
=
\alpha_{\theta,\varepsilon}
+\mathcal O(\varepsilon^2).
\]
Consequently, the metric and the left-trivialized unit normal of
$f_{\theta,\varepsilon}$ agree, up to
$\mathcal O(\varepsilon^2)$, with the formal objects obtained from
$g_\varepsilon$, $R_{\theta,\varepsilon}$, and
$N_{1,\varepsilon}$, exactly as in the proof of
Theorem~\ref{propos8.9}.
The comparison of the first-order Weingarten equations carried out in the proof of Theorem~\ref{propos8.9} applies with $J_\varepsilon$ replaced by $P_{\theta,\varepsilon}
:=
\cos\theta\,I+\sin\theta\,J_\varepsilon$, and it gives (compare with~\eqref{eq:shape-comparison-star}) 
\begin{equation}\label{eq:shape-first-jet-theta}
A^{(\theta)}_\varepsilon
=
P_{\theta,\varepsilon}^{-1}A_\varepsilon
+\mathcal O(\varepsilon^2)
=
\bigl(\cos\theta\,I-\sin\theta\,J_\varepsilon\bigr)A_\varepsilon
+\mathcal O(\varepsilon^2),
\end{equation}
where $A^{(\theta)}_\varepsilon$ is the actual shape operator of $f_{\theta,\varepsilon}$, whereas $A_\varepsilon=A+\varepsilon\dot A$ is the first jet associated with the original variation $f_s$. For every $s$, $J_s$ is skew-adjoint and $A_s$ is self-adjoint with respect to $g_s$; hence $\mbox{trace}(J_sA_s)=0$ for any $s$. Taking first jets in this identity, or equivalently using~\eqref{eq118}, we obtain
$\mbox{trace}(J_\varepsilon A_\varepsilon)
=
\mathcal O(\varepsilon^2)$.
Moreover, $u$ is Jacobi, and therefore $\mbox{trace}(A_\varepsilon)
=2H_\varepsilon =\mathcal O(\varepsilon^2)$.
Taking traces in~\eqref{eq:shape-first-jet-theta}, we conclude that
\[
2H^{(\theta)}_\varepsilon
=
\cos\theta\,\operatorname{trace}(A_\varepsilon)
-\sin\theta\,\operatorname{trace}(J_\varepsilon A_\varepsilon)
+\mathcal O(\varepsilon^2)
=
\mathcal O(\varepsilon^2).
\]
Thus, the first variation of the mean curvature of $f_{\theta,\varepsilon}$ vanishes. Decomposing $X_\theta$ into its tangential and normal components and using the normal variation formula, as at the end of the proof of Theorem~\ref{propos8.9}, gives
$\mathcal L_{f_\theta}u_\theta=0$.

Finally, the proof of Lemma~\ref{lem:initial-data-gauge} depends only on the flat equations for $R$ and for the Maurer--Cartan form of the target immersion. Applied to $R_\theta$ and $\alpha_\theta$, it shows that changing the initial conditions adds the restriction of an ambient Killing field along $f_\theta$, whose normal component is a Jacobi--Killing function. This proves the theorem.
\end{proof}

\begin{remark}[Relation with the Sym--Bobenko formula]\label{rem:Sym-Bobenko}
{\rm 
Theorem~\ref{thm:Frobenius-associated-family} is a geometric Frobenius formulation of the same associated-family phenomenon that appears in the integrable-systems description of CMC surfaces. In the latter framework, the Gauss--Codazzi equations are encoded by a loop of flat $\mathfrak{su}(2)$-connections depending on a spectral parameter $\lambda\in\mathbb S^1$. After integrating this loop to an extended frame $F_\lambda$, an immersion into $\mathbb S^3\simeq SU(2)$ is recovered from the values of $F_\lambda$ at two Sym points by the Sym--Bobenko formula. Simultaneously rotating the two Sym points while keeping their angular separation fixed produces the associated family; in the minimal case the separation is the one for which the resulting mean curvature vanishes. See, for instance, Pinkall and Sterling~\cite{ps1}, Bobenko~\cite{bob}, and Dorfmeister, Pedit, and Wu~\cite{dpw}.

In our formulation the associated-family parameter is fixed geometrically by
\[
P_\theta=\cos\theta\,I+\sin\theta\,J,
\qquad A_\theta=P_\theta^{-1}A.
\]
Identifying $\theta$ with the phase of a particular loop parameter requires specifying the frame, gauge, and spectral-parameter conventions; no such identification is needed here. The immersion is reconstructed by the two successive Frobenius systems~\eqref{eq:R-theta-system} and~\eqref{eq:f-theta-Frobenius}, without loop-group factorization. Their direct linearization gives Theorem~\ref{thm:Jacobi-associated-family} for arbitrary Jacobi functions.
}
\end{remark}

\begin{remark}\label{rem:boundary-associated-family}
{\rm 
The construction of $u_\theta$ is valid for every $\theta$. The geometric interchange underlying the pure Dirichlet--Neumann rule occurs at the conjugate angles $\theta=\pi/2$ modulo $\pi$, where $A_\theta=\pm JA$. At these angles asymptotic and principal directions are interchanged. The case $\theta=\pi/2$ is proved in Section~\ref{sec:integrable-gauge}; the opposite conjugate angle differs by an ambient orientation-reversing congruence, with the corresponding change of unit normal, and gives the same homogeneous boundary conditions. In contrast, at $\theta=0$ modulo $\pi$ the two types of directions are preserved. Away from these special angles an asymptotic direction at a nonumbilic point is in general neither asymptotic nor principal for $f_\theta$. Thus the argument above provides no pure Dirichlet--Neumann exchange for a general associated member, and we make no claim here about a general mixed boundary formula.
}
\end{remark}

\section{Jacobi conjugation for CMC $1$ surfaces in $\R^3$}
\label{sec:CMC-Jacobi-conjugation}

We now return to case~(d) of the introduction. Let $f,f^*\colon(\Sigma,g)\to\R^3$ be a conjugate pair of oriented CMC~$1$ immersions as in Section~\ref{sec8.1}, with $\Sigma$ simply connected. We use the convention $df\circ A=-dN$, and write
\begin{equation}\label{eq170}
R\,df^*=df\circ J,\qquad dR=\mathcal C R,\qquad
\mathcal C=-\phi\bigl(df\circ(I+J)\bigr), \qquad N^*=R^{\top}\, N,
\end{equation}
where $\phi(v)w=v\times w$. The corresponding normals and shape operators are related by~\eqref{eq146} and~\eqref{eq:CMC-shape-conjugate}. In particular, $|A^*|^2=|A|^2$, and the two Jacobi operators coincide:
$\mathcal L_{f}=\Delta+|A|^2=\Delta+|A^*|^2=\mathcal L_{f^*}$.

\begin{definition}
{\rm 
Let $f\colon\Sigma\to\mathbb R^3$ be an oriented CMC immersion with unit normal $N$. A function $v\in C^\infty(\Sigma)$ is called a \emph{Euclidean Jacobi--Killing function} if there exists an ambient Euclidean Killing field of the form $K(x)=b-c\times x$, $x\in\mathbb R^3$, with $b,c\in\mathbb R^3$, such that $v=\langle K\circ f,N\rangle$.
Since the flow of $K$ consists of Euclidean isometries, every Euclidean Jacobi--Killing function belongs to the kernel of $\mathcal{L}_f$.
}
\end{definition}

We next state the parallel result to Theorem~\ref{propos8.9} and Lemma~\ref{lem:initial-data-gauge} in the context of conjugate CMC $H=1$ immersions into $\mathbb{R}^3$.
\begin{theorem}\label{thm:CMC-Jacobi-conjugation}
Let $f,f^*\colon (\Sigma,g)\to \mathbb R^3$ be a conjugate pair of oriented CMC $1$ immersions, with $\Sigma$ simply connected, and let $R\colon \Sigma \to SO(3)$ be the corresponding frame map. Fix a base point $p_0\in\Sigma$. If $u\in C^\infty(\Sigma)$ satisfies $\mathcal L_f u=0$, then the following hold.
\begin{enumerate}
\item The $\mathbb R^3$-valued one-form on $\Sigma$ defined by
\begin{equation}\label{eq:CMC-omega-form}
\zeta_u(Y)	:=	u\,df^*((I-J)AY)-du(Y+JY)N^*
\end{equation}
is closed. Consequently, for every $\omega_0\in\mathbb R^3$, the initial value problem
\begin{equation}\label{eq:CMC-omega-PDE}
d\omega=\zeta_u \quad \mbox{on $\Sigma$,} \qquad \omega(p_0)=\omega_0,
\end{equation}
has a unique global solution $\omega\colon\Sigma\to\mathbb R^3$.
	
\item Once $\omega$ has been determined, define the $\mathbb R^3$-valued one-form on $\Sigma$
\begin{equation}\label{eq:CMC-X-star-form}
\begin{split}
\xi_{u,\omega}(Y)
			:={}&
			-\omega\times df^*(Y)
			-u\,df^*(AY)
			+du(JY)N^*.
\end{split}
\end{equation}
Then $\xi_{u,\omega}$ is closed. Therefore, for every $X_0^*\in\mathbb R^3$, the
initial value problem 
\begin{equation}\label{eq:CMC-X-star-PDE}
dX^*=\xi_{u,\omega}\quad \mbox{on $\Sigma$,} \qquad	X^*(p_0)=X_0^*
\end{equation}
has a unique global solution $X^*\colon\Sigma\to \mathbb R^3$.
Moreover, $u^*:=\langle X^*,N^*\rangle$ is a Jacobi function for $f^*$, i.e.
$\mathcal L_{f^*}u^*=0$.
	
\item (Dependence on the initial conditions).
If $(\widehat\omega,\widehat X^*)$ is obtained from another choice
of initial values, then
\begin{equation}\label{eq:CMC-gauge-ambiguity}
\widehat\omega-\omega=c,
\qquad
\widehat X^*-X^*=-c\times f^*+b
\end{equation}
for some constant vectors $b,c\in\mathbb R^3$. Thus $u^*$ is well defined modulo Euclidean Jacobi--Killing functions. No integrability of $u$ through a family of CMC immersions is required.
\end{enumerate}
\end{theorem}
\begin{proof}
We first compute identities valid for an arbitrary oriented immersion into $\mathbb R^3$.
\par\vspace{.2cm}\noindent
{\bf Claim 1.} {\em 
Let $h\colon\Sigma\to \mathbb R^3$ be an oriented immersion, not necessarily of constant mean curvature. Denote by $J_h,N_h,H_h$, and $da_h$ its complex structure, unit normal, mean curvature, and area form, respectively. Set $\mathcal C_h:=-\phi\circ dh\circ(I+J_h)$.
Then:
\begin{equation}\label{eq:CMC-Frobenius-defects}
\begin{aligned}
d\mathcal C_h-\mathcal C_h\wedge\mathcal C_h
&=
2(H_h-1)\phi(N_h)\,da_h,
\\
d(dh\circ J_h)-\mathcal C_h\wedge(dh\circ J_h)
&=
2(1-H_h)N_h\,da_h.
\end{aligned}
\end{equation}
In addition,
\begin{equation}\label{eq:C-wedge-dh-general}
\mathcal C_h\wedge dh=-2N_h\,da_h.
\end{equation}
}
\par\vspace{.2cm}\noindent
{\em Proof of Claim 1.}
Fix a point $p\in\Sigma$, and choose a positively oriented $g_h$-orthonormal frame
$E_1,E_2=J_hE_1$ which is geodesic at $p$, i.e. $\nabla _{(E_i)_p}E_j=0$. 
At $p$, the Gauss formula gives $d(dh\circ J_h)(E_1,E_2)	=
-\bigl(\sigma_h(E_1,E_1)+\sigma_h(E_2,E_2)\bigr)N_h
=-2H_hN_h$, hence
\begin{equation}\label{eq177a}
d(dh\circ J_h)=-2H_hN_h\, da_h.
\end{equation}
Since $d(dh)=0$, it follows that 
\begin{equation}\label{eq177}
d\mathcal C_h	=
-\phi\bigl(d(dh\circ J_h)\bigr)	=2H_h\phi(N_h)\,da_h.
\end{equation}
Moreover,
\[
\mathcal C_h(E_1)=-\phi\bigl(dh(E_1)+dh(E_2)\bigr),
\qquad
\mathcal C_h(E_2)=\phi\bigl(dh(E_1)-dh(E_2)\bigr).
\]
Using $[\phi(v),\phi(w)]=\phi(v\times w)$, we obtain
\begin{equation}\label{eq178}	
(\mathcal C_h\wedge\mathcal C_h)(E_1,E_2)	=	2\phi(N_h).
\end{equation}
Now the first identity in~\eqref{eq:CMC-Frobenius-defects} follows directly from~\eqref{eq177}  and~\eqref{eq178}.

On the other hand,
\[
\begin{aligned}
	\bigl(\mathcal C_h\wedge(dh\circ J_h)\bigr)(E_1,E_2)
	&=
	\mathcal C_h(E_1)(-dh(E_1))
	-\mathcal C_h(E_2)dh(E_2)
	\\
	&=
	(dh(E_1)+dh(E_2))\times dh(E_1)
	-(dh(E_1)-dh(E_2))\times dh(E_2)
	\\
	&=-2N_h,
\end{aligned}
\]
together with~\eqref{eq177a}, proves the second identity in
\eqref{eq:CMC-Frobenius-defects}.
		
		Finally,
		\[
		\begin{aligned}
			(\mathcal C_h\wedge dh)(E_1,E_2)
			&=
			\mathcal C_h(E_1)dh(E_2)-\mathcal C_h(E_2)dh(E_1)
			\\
			&=
			-(dh(E_1)+dh(E_2))\times dh(E_2)
			-(dh(E_1)-dh(E_2))\times dh(E_1)
			\\
			&=-2N_h,
		\end{aligned}
		\]
which demonstrates \eqref{eq:C-wedge-dh-general}. This completes the proof of Claim 1.
\par\vspace{.2cm}

We continue with the proof of Theorem~\ref{thm:CMC-Jacobi-conjugation}. Choose an actual normal variation $f_s\colon\Sigma\to\mathbb R^3$ of $f_0=f$, so that
$\left.\frac{\partial f_s}{\partial s}\right|_{s=0}=uN$. The immersions $f_s$ are not assumed to have constant mean curvature for $s\neq0$. Let $g_s$, $J_s$, $N_s$, $H_s$, $da_s$ denote the corresponding geometric objects, and define
\begin{equation}\label{Cs}
\mathcal C_s:=-\phi\circ df_s\circ(I+J_s).
\end{equation}
Since $f$ has CMC $H=1$, the traceless shape operator of $f$ is $B=A-I$.
Although equation~\eqref{dotJ} was stated for immersions
into $\mathbb{S}^3$, its proof does not involve the curvature of the
ambient space and therefore applies without change to immersions into
$\mathbb{R}^3$. Hence the first variation of the complex structure is given by 
\begin{equation}\label{eq:Jdot-CMC-one}
	\dot J=-2uJ(A-I).
\end{equation}
\par\vspace{.2cm}\noindent
{\bf Claim 2} (Linearized equation for $\mathcal C$).
{\em 
For the above normal variation of the CMC\/ $1$ immersion $f$, the first variation of $\mathcal C$ is the following smooth one-form on $\Sigma$ taking values in the space of skew-symmetric matrices of order three:
	\begin{equation}\label{eq:CMC-C-variation}
	\dot{\mathcal C}
	=
	\phi\!\left(
	u\,df\circ(I+J)A
	-
	\bigl(du\circ(I+J)\bigr)N
	\right).
\end{equation}
Here the product of a scalar-valued one-form by $N$ is understood
pointwise; explicitly,
\[
\Bigl(\bigl(du\circ(I+J)\bigr)N\Bigr)(Y)
=
\bigl(du(Y)+du(JY)\bigr)N.
\]
}
\par\vspace{.2cm}\noindent
{\em Proof of Claim 2.} Differentiating~\eqref{Cs} at $s=0$ gives, for every $Y\in T\Sigma$,
\begin{eqnarray*}
\dot{\mathcal C}(Y)
&=&
-\phi\left(
\left.
\frac{\partial}{\partial s}
\right|_{s=0}
\bigl[df_s\bigl((I+J_s)Y\bigr)\bigr]
\right)
\\
&=&
-\phi\left(
d(uN)\bigl((I+J)Y\bigr)
+
df(\dot JY)
\right)
\\
&\stackrel{\eqref{eq:Jdot-CMC-one}}{=}&
-\phi\left(
\bigl(du(Y)+du(JY)\bigr)N
+
u\,dN\bigl((I+J)Y\bigr)
-2u\, df(J(A-I)Y)
\right)
\\
&=&
-\phi\Bigl(
\bigl(du(Y)+du(JY)\bigr)N
-u\,df\bigl(A(I+J)Y\bigr)
-2u\,df\bigl(J(A-I)Y\bigr)
\Bigr).
\end{eqnarray*}
We now simplify the two terms involving $df$. Since $AJ+JA=\mbox{trace}(A)J=2J$ because $f$ has mean curvature $H=1$, we have $A(I+J)+2J(A-I)=(I+J)A$.
Consequently,
\[
\begin{aligned}
	\dot{\mathcal C}(Y)
	&=
	-\phi\left(
	\bigl(du(Y)+du(JY)\bigr)N
	-u\,df\bigl((I+J)AY\bigr)
	\right)
	\\
	&=
	\phi\left(
	u\,df\bigl((I+J)AY\bigr)
	-\bigl(du(Y)+du(JY)\bigr)N
	\right),
\end{aligned}
\]
which is~\eqref{eq:CMC-C-variation}. This proves Claim~2.
\par\vspace{.2cm}\noindent
{\bf Claim 3}
{\em 
The one-form $R^\top\dot{\mathcal C}R$ is closed. Moreover,
$R^\top\dot{\mathcal C}R=\phi\circ\zeta_u$, and thus, $\zeta_u$ is also a closed one-form on $\Sigma$.
}
\par\vspace{.2cm}\noindent
{\em Proof of Claim 3.} The first identity in~\eqref{eq:CMC-Frobenius-defects} applies to the immersion $f_s$ gives $d\mathcal C_s-\mathcal C_s\wedge\mathcal C_s=2(H_s-1)\phi(N_s)\,da_s$. Differentiating this equation at $s=0$ and using $H=1$ gives
\begin{equation}\label{eq:CMC-linearized-curvature-defect}
d\dot{\mathcal C}
	-\mathcal C\wedge\dot{\mathcal C}
	-\dot{\mathcal C}\wedge\mathcal C
=
2\dot H\,\phi(N)\,da
=
(\mathcal L_fu)\phi(N)\,da
=0.
\end{equation}
Therefore,
\begin{eqnarray}
d\bigl(R^\top\dot{\mathcal C}R\bigr) 
&=&
d\bigl(R^\top\dot{\mathcal C}\bigr) R
-R^{\top}\dot{\mathcal C}\wedge dR \nonumber
\\
&=&
\bigl( d(R^\top)\wedge \dot{\mathcal C}+R^{\top}d\dot{\mathcal C}\bigr) R
-R^{\top}\dot{\mathcal C}\wedge dR \quad \nonumber
\\
&\stackrel{\eqref{eq170}}{=}&
\bigl( -R^{\top}\mathcal C\wedge \dot{\mathcal C}
+R^{\top} d\dot{\mathcal C}\bigr) R
-R^{\top}\dot{\mathcal C}\wedge \mathcal C R \nonumber
\\
&=&
R^\top
\bigl( -\mathcal C\wedge\dot{\mathcal C}
+d\dot{\mathcal C}
-\dot{\mathcal C}\wedge\mathcal C
\bigr)R
\nonumber\\
&\stackrel{\eqref{eq:CMC-linearized-curvature-defect}}{=}&
0.
\label{eq:closed-RCdotR}
\end{eqnarray}
We next show that $R^\top\dot{\mathcal C}R=\phi\circ\zeta_u$, where $\zeta_u$ is the one-form introduced in~\eqref{eq:CMC-omega-form}. Since $R\in SO(3)$, for any $w\in \mathbb{R}^3$ we have $(R^\top\phi(v)R)w=R^\top\bigl( \phi(v)(Rw)\bigr) =
R^\top\bigl( v\times Rw\bigr) =(R^\top v)\times w=\phi(R^{\top}v)w$.
Thus, $R^\top\phi(v)R=\phi(R^\top v)$. Applying this equality to 
$v=u\, df((I+J)AY)-du(Y+JY)N$ where $Y\in T\Sigma$, we get
\begin{eqnarray*}
R^\top\dot{\mathcal C}(Y)R
&=&
\phi\left(
u\,R^\top df\bigl((I+J)AY\bigr)
-du(Y+JY)R^\top N\right) 
\\
&=&
\phi\left(
u\,R^\top df\bigl(J(I-J)AY\bigr)
-du(Y+JY)R^\top N
\right)
\\
&\stackrel{\eqref{eq170}}{=}&
\phi\left(
u\,df^*\bigl((I-J)AY\bigr)
-du(Y+JY)N^*
\right)
\\
&\stackrel{\eqref{eq:CMC-omega-form}}{=}&
\phi\bigl(\zeta_u(Y)\bigr).
\end{eqnarray*}
Since $\phi\circ\zeta_u=R^\top\dot{\mathcal C}R$ is closed and 
$\phi\colon\mathbb R^3\to\mathfrak{so}(3)$ is a linear
isomorphism, then $\zeta_u$ is also closed, and Claim~3 is proved.
\par\vspace{.2cm}
As $\Sigma$ is simply connected, for every prescribed
$\omega(p_0)=\omega_0\in \mathbb{R}^3$ there exists a unique global solution $\omega\colon\Sigma\to\mathbb R^3$ of~\eqref{eq:CMC-omega-PDE}. This proves item~1 of the theorem.

Regarding item 2, we shall first need the first variation of $s\mapsto \beta_s:=df_s\circ J_s\in \Omega^1(\Sigma,\mathbb{R}^3)$.
Given $Y\in T\Sigma$, 
\begin{eqnarray*}
\dot{\beta}(Y)
&:=&
\left. \frac{d}{ds}\right|_{s=0}\beta_s(Y)=
\left. \frac{d}{ds}\right|_{s=0}(df_s\circ J_s)(Y)=d(uN)(JY)+df(\dot{J}Y)
\\
&\stackrel{\eqref{eq:Jdot-CMC-one}}{=}&
du(JY)N+u\, dN(JY)-2u\, df(J(A-I)Y).
\end{eqnarray*}
Using the Weingarten equation for $f$ and that $AJ+JA=2J$, we can write the last expression as $du(JY)N-u\, df(JAY)$, and thus,
\begin{equation}\label{eq:CMC-beta-variation}
\dot{\beta}
=
(du\circ J)N-u\,df\circ JA\in \Omega^1(\Sigma,\mathbb{R}^3).
\end{equation}
\par\vspace{.2cm}\noindent
{\bf Claim 4}
{\em 
Set $\Omega:=\phi(\omega)$, where $\omega\colon \Sigma\to \mathbb{R}^3$ is the solution to~\eqref{eq:CMC-omega-PDE}. Then:
\begin{enumerate}
\item $R^{\top}\dot{\mathcal{C}}R=d\Omega$ in $\Sigma$.
\item $d\Omega\wedge df^*=R^\top(\dot{\mathcal C}\wedge\beta)$ in $\Sigma$.
\item Define (compare with~\eqref{dotalpha*})
\begin{equation}\label{eq187}
	\dot{\alpha}^*:=-\Omega \, df^*+R^{\top}\dot{\beta}\in \Omega^1(\Sigma,\mathbb{R}^3).
\end{equation}
Then, $\dot{\alpha}^*$ is closed. Moreover, $\dot{\alpha}^*=\xi_{u,\omega}$.
\end{enumerate}
}
\par\vspace{.2cm}\noindent
{\em Proof of Claim 4.} By Claim~3,
\begin{equation}
R^{\top}\dot{\mathcal{C}}R=\phi(\zeta_u)\stackrel{\eqref{eq:CMC-omega-PDE}}{=}
\phi(d\omega)=d\phi(\omega)=d\Omega,
\end{equation}
which is item~1 of the claim. 

Given $X,Y\in T\Sigma$, using~\eqref{eq170} and item 1 we get
\[
\begin{aligned}
(d\Omega\wedge df^*)(X,Y)
&=
d\Omega(X)df^*(Y)-d\Omega(Y)df^*(X) 
\\
&=
R^\top\dot{\mathcal C}(X)R
R^\top\beta(Y)
-
R^\top\dot{\mathcal C}(Y)R
R^\top\beta(X)
	\\
	&=
	R^\top\left(
	\dot{\mathcal C}(X)\beta(Y)
	-
	\dot{\mathcal C}(Y)\beta(X)
	\right)
	\\
	&=
	R^\top(\dot{\mathcal C}\wedge\beta)(X,Y),
\end{aligned}
\]
and item~2 follows. 

As for item~3, the second identity in~\eqref{eq:CMC-Frobenius-defects} applied to $f_s$ gives $d\beta_s-\mathcal C_s\wedge \beta_s=2(1-H_s)N_s\,da_s$. Differentiating this equation at $s=0$ and using $H=1$ we obtain
\begin{eqnarray}
0&=&-(\mathcal L_fu)N\,da=-2\dot H\,N\,da=\left. \frac{d}{ds}\right|_{s=0}
2(1-H_s)N_s\,da_s \nonumber
\\
&=&
\left. \frac{d}{ds}\right|_{s=0}(d\beta_s-\mathcal C_s\wedge \beta_s)
=
d\dot{\beta}-\dot{\mathcal{C}}\wedge \beta-\mathcal{C}\wedge \dot\beta.
\label{eq:CMC-linearized-immersion-defect}
\end{eqnarray}

Taking exterior differentials on~\eqref{eq187} and using item 2 and~\eqref{eq170},
\begin{eqnarray*}
d\dot {\alpha}^*
&=&
-d\Omega\wedge df^*
+dR^\top\wedge\dot{\beta}
+R^\top d\dot{\beta}
\\
&=&
-R^\top(\dot{\mathcal C}\wedge\beta)
-R^\top(\mathcal C\wedge\dot{\beta})
+R^\top d\dot{\beta}
\\
&=&
R^\top
\bigl(
d\dot\beta
-\mathcal C\wedge\dot\beta
-\dot{\mathcal C}\wedge\beta
\bigr)
\stackrel{\eqref{eq:CMC-linearized-immersion-defect}}{=}0,
\end{eqnarray*}
so $\dot {\alpha}^*$ is closed. 

Finally, given $Y\in T\Sigma$,
\begin{eqnarray*}
\dot{\alpha}^*(Y)
&\stackrel{\eqref{eq187}}{=}&
-\Omega \, df^*(Y)+R^{\top}\dot{\beta}(Y)
\\
&\stackrel{\eqref{eq:CMC-beta-variation}}{=}&
-\Omega\,df^*(Y)
+du(JY)R^\top N
-u\,R^\top df(JAY)
\\
&\stackrel{(\star)}{=}&
-\omega\times df^*(Y)
+du(JY)N^*
-u\,df^*(AY)
\\
&\stackrel{\eqref{eq:CMC-X-star-form}}{=}&
\xi_{u,\omega}(Y),
\end{eqnarray*}
where in $(\star)$  we have used that $\Omega=\phi(\omega)$,
$\phi(\omega)v=\omega\times v$, and~\eqref{eq170}. This finishes the proof of Claim~4.
\par\vspace{.2cm}
We continue with the proof of Theorem~\ref{thm:CMC-Jacobi-conjugation}. As $\xi_{u,\omega}$ is closed and $\Sigma$ is simply connected, for every prescribed value
$X^*(p_0)=X^*_0\in \mathbb{R}^3$ there exists a unique global solution $X^*\colon\Sigma\to\mathbb R^3$ of~\eqref{eq:CMC-X-star-PDE}. 

To finish the proof of item~2 of the theorem, it remains to show that $\mathcal L_{f^*}u^*=0$ for $u^*:=\langle X^*,N^*\rangle$. We follow the first-jet argument used in the proof of Theorem~\ref{propos8.9}.

Define
\begin{equation}\label{eq:CMC-Rs-formal}
R_\varepsilon :=R\exp(\varepsilon \Omega),
\end{equation}
where $\exp\colon \mathfrak{so}(3)\to SO(3)$ is the usual exponential of matrices.
$\varepsilon \mapsto R_\varepsilon $ is an $SO(3)$-valued family; we do not claim that, for $\varepsilon \neq0$, $R_\varepsilon $
is the frame map of a conjugate pair. Clearly, 
\[
R_\varepsilon =R+\varepsilon \dot{R}+\mathcal{O}(\varepsilon ^2),\qquad \mbox{where $\dot{R}:=R\Omega$.}
\]

All $\mathcal O(\varepsilon ^2)$ statements below are understood locally in
$C^1$ on compact subsets of $\Sigma$, so that exterior differentiation preserves the
order of the remainder. Expanding~\eqref{eq:CMC-Rs-formal} and using item~1 of Claim 4, we obtain
\begin{eqnarray}
dR_\varepsilon 
&=&
dR
+\varepsilon \,d(R\Omega)
+\mathcal O(\varepsilon ^2)
\stackrel{\eqref{eq170}}{=}
\mathcal CR
+\varepsilon \bigl(\mathcal CR\Omega+R\,d\Omega\bigr)
+\mathcal O(\varepsilon ^2) \nonumber
\\
&=&
\mathcal CR
+\varepsilon \bigl(\mathcal C\dot{R}+\dot{\mathcal C}R\bigr)
+\mathcal O(\varepsilon ^2)
=\mathcal C_\varepsilon R_\varepsilon +\mathcal O(\varepsilon ^2). \label{eq:CMC-dRs}
\end{eqnarray}

Now define an actual variation of the conjugate CMC $H=1$ immersion $f^*$ by
(compare to~\eqref{fvarepsilon*})
\begin{equation}\label{eq:CMC-target-actual-variation}
	f_\varepsilon ^*:=f^*+\varepsilon X^*.
\end{equation}
Around every point of $\Sigma$, $f_\varepsilon ^*$ is an immersion for all
sufficiently small $\varepsilon $. Since the Jacobi equation is local, this is
enough for the argument. Now,
\begin{eqnarray}
R_{\varepsilon}^{\top}(df_{\varepsilon}\circ J_{\varepsilon})
&=&
R_{\varepsilon}^{\top}\beta_{\varepsilon}=[R^{\top}-\varepsilon \, \Omega R^{\top}+\mathcal{O}(\varepsilon^2)][df\circ J+\varepsilon\dot{\beta}+\mathcal{O}(\varepsilon^2)] \nonumber
\\
&=&
R^{\top}(df\circ J)+\varepsilon \bigl( R^{\top}\dot{\beta}-\Omega R^{\top}(df\circ J)\bigr) +\mathcal{O}(\varepsilon^2)  \nonumber
\\
&\stackrel{\eqref{eq170}}{=}&
df^*+\varepsilon\bigl( R^{\top}\dot{\beta}-\Omega \, df^*\bigr) +\mathcal{O}(\varepsilon^2)  \nonumber
\\
&\stackrel{\eqref{eq187}}{=}&
df^*+\varepsilon\, \dot{\alpha}^* +\mathcal{O}(\varepsilon^2)
\stackrel{(\star)}{=}
df^*+\varepsilon\, \xi_{u,\omega}+\mathcal{O}(\varepsilon^2)  \nonumber
\\
&\stackrel{\eqref{eq:CMC-X-star-PDE}}{=}&
df^*+\varepsilon\, dX^*+\mathcal{O}(\varepsilon^2)  \nonumber
\\
&=&
df_{\varepsilon}^*+\mathcal{O}(\varepsilon^2), \label{eq:CMC-first-jet-conjugation}
\end{eqnarray}
where in $(\star)$ we have used item~3 of Claim~4.

Let $\widehat g_\varepsilon$, $\widehat J_\varepsilon$, $\widehat N_\varepsilon$,
$d\widehat a_\varepsilon$, $\widehat H_\varepsilon$ denote, respectively, the metric, complex structure, chosen unit normal, area form, and mean curvature of the immersion
$f_\varepsilon ^*$, with $\widehat N_0=N^*$. Thus,
\[
\widehat g_0=g,\qquad
\widehat J_0=J,\qquad
\widehat N_0=N^*,\qquad
\widehat H_0=H^*=1.
\]
The hat notation is used to distinguish these actual geometric objects from the quantities formally transported from the source variation $f_\varepsilon$ (where conjugation does not necessarily make sense).

Since $R_\varepsilon \in SO(3)$, equation~\eqref{eq:CMC-first-jet-conjugation} gives
$\widehat g_\varepsilon =g_\varepsilon +\mathcal O(\varepsilon ^2)$. 
Consequently,
\begin{equation}\label{eq:CMC-target-geometry-first-jet}
\widehat J_\varepsilon 
=
J_\varepsilon +\mathcal O(\varepsilon ^2),
\qquad
\widehat N_\varepsilon 
=
R_\varepsilon ^\top N_\varepsilon +\mathcal O(\varepsilon ^2),
\qquad
d\widehat a_\varepsilon 
=
da_\varepsilon +\mathcal O(\varepsilon ^2).
\end{equation}
The sign in the second relation is positive because at $\varepsilon =0$ it
reduces to $N^*=R^\top N$.

We can now compute the mean curvature of $f_\varepsilon ^*$ to first order.
Applying~\eqref{eq177a} to the immersion $f_{\varepsilon}^*$, we deduce
\begin{eqnarray*}
2\widehat H_\varepsilon \widehat N_\varepsilon \,d\widehat a_\varepsilon 
&=&
-d(df_\varepsilon ^*\circ\widehat J_\varepsilon )
\\
&\stackrel{\eqref{eq:CMC-first-jet-conjugation},
\eqref{eq:CMC-target-geometry-first-jet}}{=}&
-d\bigl(
R_\varepsilon ^\top (df_\varepsilon \circ J_\varepsilon^2)
\bigr)
+\mathcal O(\varepsilon ^2)
\\
&=&
d(R_\varepsilon ^\top df_\varepsilon )+\mathcal O(\varepsilon ^2)
\\
&\stackrel{\eqref{eq:CMC-dRs}}{=}&
-R_\varepsilon ^\top
(\mathcal C_\varepsilon \wedge df_\varepsilon )
+\mathcal O(\varepsilon ^2).
\end{eqnarray*}
Applying~\eqref{eq:C-wedge-dh-general} to the immersion $f_\varepsilon $,
we have $\mathcal C_\varepsilon \wedge df_\varepsilon 
=
-2N_\varepsilon \,da_\varepsilon $.
Hence
\[
2\widehat H_\varepsilon \widehat N_\varepsilon \,d\widehat a_\varepsilon 
=
2R_\varepsilon ^\top N_\varepsilon \,da_\varepsilon +\mathcal O(\varepsilon ^2).
\]
Using~\eqref{eq:CMC-target-geometry-first-jet} on the left-hand side of the last expression, we conclude that
\begin{equation}\label{eq:CMC-target-H-first-jet}
	\widehat H_\varepsilon =1+\mathcal O(\varepsilon ^2).
\end{equation}
In particular, the first jet of $\widehat H_\varepsilon$ vanishes, i.e.,
$\left.
\frac{d}{d\varepsilon }
\right|_{\varepsilon =0}
\widehat H_\varepsilon =0$. 

Decompose the variational field of~\eqref{eq:CMC-target-actual-variation} in its tangent and normal parts, $X^*=df^*(Z)+u^*N^*$.
For a general variation of a CMC immersion, the first variation of
the mean curvature is
\[
2\left.
\frac{d}{d\varepsilon }
\right|_{\varepsilon =0}\widehat H_\varepsilon 
=
\mathcal L_{f^*}u^*.
\]
Together with \eqref{eq:CMC-target-H-first-jet}, this gives $\mathcal L_{f^*}u^*=0$, as desired. This completes the proof of item~2 of Theorem~\ref{thm:CMC-Jacobi-conjugation}.

It remains to analyze the dependence on the initial conditions.
Let $(\widehat\omega,\widehat X^*)$ be obtained from another pair of
initial values. Since both $\omega$ and $\widehat\omega$ solve the
same first-order PDE~\eqref{eq:CMC-omega-PDE}, we deduce that $d(\widehat\omega-\omega)=0$
on $\Sigma$. As $\Sigma$ is connected, there exists a constant vector
$c\in\mathbb R^3$ such that $\widehat\omega-\omega=c$.

Subtracting the two equations~\eqref{eq:CMC-X-star-PDE} gives
\[
d(\widehat X^*-X^*)=\xi_{u,\widehat{\omega}}-\xi_{u,\omega}
\stackrel{\eqref{eq:CMC-X-star-form}}{=}
-(\widehat\omega-\omega)\times df^*=-c\times df^*.
\]
Therefore, $d\bigl( \widehat X^*-X^*+c\times f^* \bigr)=0$, hence there exists a constant vector $b\in\mathbb R^3$ such that $\widehat X^*-X^*=-c\times f^*+b$. This proves~\eqref{eq:CMC-gauge-ambiguity}.

The vector field $x\in \mathbb{R}^3\mapsto -c\times x+b$ is a Killing field on $\mathbb{R}^3$. Consequently, the corresponding normal
components $u^*, \widehat{u}^*$ differ by a Euclidean Jacobi--Killing function.
This proves item~3 and completes the proof of the theorem.
\end{proof}

\subsection{Boundary behavior}
The boundary behavior for a pair of conjugate CMC $H=1$ immersions $f,f^*\colon \Sigma \to \mathbb{R}^3$ differs from that in Section~\ref{section5.2}. The geometric distinction already appears in~\eqref{eq:CMC-shape-conjugate}: Let $\gamma\subset \Sigma$ be a unit-speed curve, and set $\{T=\gamma',\eta=JT\}$ the corresponding positively oriented orthonormal frame of $T\Sigma$. Then,
\begin{eqnarray*}
\sigma^*(T,\eta)&=&
g(A^*T,\eta)=g(T-J(A-I)T,\eta )=1-g(AT,T)=1-\sigma(T,T),
\\
\sigma^*(T,T)&=&g(T-J(A-I)T,T)=1+g(AT,\eta )=1+\sigma(T,\eta).
\end{eqnarray*}

As usual, let us denote by $\kappa_n:=\sigma(T,T)$, $\tau_g:=\sigma(T,\eta)$ the normal 
curvature and geodesic torsion of $f(\gamma)$, respectively. Analogously, let
$\kappa^*_n:=\sigma^*(T,T)$, $\tau^*_g:=\sigma^*(T,\eta)$ be the normal 
curvature and geodesic torsion of $f^*(\gamma)$, respectively. Then, the above equations read
\[
\tau_g^*=1-\kappa _n,\qquad \kappa_n^*=1+\tau_g.
\]
Thus, conjugation exchanges curvature lines for $f$ (i.e., those where 
$\tau_g=\sigma(T,\eta)=0$) with curves for $f^*$ whose normal curvature equals the mean curvature $H^*=1$, and vice versa.
Since $f,f^*$ induce the same metric, their geodesic curvatures also agree. Hence geodesics which are also curvature lines for $f$ correspond to geodesics for $f^*$ whose normal curvature is $\kappa_n^*=1$, and conversely. 
Notice that this is a correspondence between surface-adapted families of spatial
curves; unlike the minimal case, it does not in general reduce to
an exchange between two elementary ambient classes such as straight
lines and planar reflection curves.

To identify the corresponding boundary behavior at the linearized level (for Jacobi functions), take a function $u\in C^{\infty}(\Sigma)$ satisfying $\mathcal L_f u=0$, 
initial conditions $(\omega_0,X_0^*)\in\mathbb R^3\times \mathbb R^3$ and 
decompose the smooth $\mathbb{R}^3$-valued vector fields $\omega,X^*\in C^{\infty}(\Sigma,\mathbb R^3)$ given respectively by~\eqref{eq:CMC-omega-PDE}, \eqref{eq:CMC-X-star-PDE} into their tangent and normal components to $f^*$:
\begin{equation}\label{eq196}
X^*=df^*(Z)+u^*N^*,\qquad \omega=df^*(W)+\omega_nN^*,
\end{equation}
where $Z,W$ are tangent vector fields to $\Sigma$. 

Take any $Y\in T\Sigma$. Using~\eqref{eq:CMC-X-star-form} and~\eqref{eq:CMC-X-star-PDE} we have
\[
dX^*(Y)=\xi_{u,\omega}(Y)=-\omega\times df^*(Y) -u\,df^*(AY) +du(JY)N^*.
\]
for the left-hand side, the Gauss formula yields
\[
dX^*(Y)\stackrel{\eqref{eq196}}{=}
	\overline{\nabla}^{\mathbb{R}^3}_Y\bigl(df^*(Z)+u^*N^*\bigr)
=
\overline{\nabla}^{\mathbb{R}^3}_Ydf^*(Z)
	+du^*(Y)N^*
	+u^*dN^*(Y).
\]
Taking normal components to $f^*$ gives
\begin{eqnarray}
\sigma^*(Y,Z)+du^*(Y)
&=&
-\langle \omega \times df^*(Y),N^*\rangle +du(JY)  \nonumber
\\
&\stackrel{\eqref{eq196}}{=}&
-\langle df^*(W)\times df^*(Y),N^*\rangle +du(JY)  \nonumber
\\
&=&
g(W,JY)+du(JY). \label{eq197}
\end{eqnarray}
Taking $Y=T=\gamma'$ and $Y=\eta $ in \eqref{eq197}, we respectively get
\begin{eqnarray*}
\sigma^*(\gamma',Z)+(u^*\circ \gamma)'
&=&
g(W,\eta)+\partial_{\eta}u, 
\\
\sigma^*(\eta,Z)+\partial_{\eta}u^*
&=&-g(W,\gamma')-(u\circ \gamma)'.
\end{eqnarray*}

The following example shows that the additional terms in the last two equations cannot, in general, be eliminated along the arc by changing the initial values (compare to Theorems~\ref{thm:D-to-N} and~\ref{thm:N-to-D}).

\begin{example}
\label{ex:CMC-cylinder-boundary}
{\rm 
Consider the oriented CMC~$1$ cylinder
\[
f(s,t)=\left(\tfrac12\cos(2t),\tfrac12\sin(2t),s\right),\qquad
N(s,t)=(-\cos(2t),-\sin(2t),0)
\]
on a simply connected neighborhood $\Sigma$ of a segment of $\{t=0\}$. Here $g=ds^2+dt^2$, $J\partial_s=\partial_t$, $A=\operatorname{diag}(0,2)$, and
$\mathcal L_f=\partial_s^2+\partial_t^2+4$. 

With the notation of Theorem~\ref{thm8.3}, take 
\[
p_0=(0,0),\qquad 
f^*(p_0)=\left(\frac12,0,0\right),\qquad
R(p_0)=
\begin{pmatrix}
1&0&0\\
0&1/\sqrt2&1/\sqrt2\\
0&-1/\sqrt2&1/\sqrt2
\end{pmatrix}.
\]
Then, equations~\eqref{eq147}, \eqref{eq148} produce the conjugate immersion
\[
f^*(s,t)=f\left(\frac{s-t}{\sqrt2},\frac{s+t}{\sqrt2}\right);
\]
this can be checked by the following three facts: $f,f^*$ induce the same metric, the oriented shape operator of $f^*$ is
\begin{equation}\label{eq198}
A^*=\begin{pmatrix}
1&1\\
1&1
\end{pmatrix}=I-J(A-I)
\end{equation} 
(which matches with~\eqref{eq:CMC-shape-conjugate}), and equation~\eqref{eq147} holds at $p_0$.

The source arc $\gamma(s)=(s,0)$ on $\Sigma$ produces via $f$ a straight generatrix which is also a reflective geodesic. Its conjugate on $f^*$ is a helix with curvature and torsion both equal to $1$; to see this, note that along $f^*(\gamma)$ we have the positive orthonormal frame of $\mathbb{R}^3$ $\{(\gamma^*)',\eta^*,N^*\}:=\{df^*(\partial_s),df^*(\partial_t),N^*\}$.
Since $(s,t)$ are orthonormal coordinates for the flat metric $g=ds^2+dt^2$,
we have $\nabla_{\partial_s}\partial_s=0$, $\nabla_{\partial_s}\partial_t=0$ for the Levi-Civita connection associated to $g$. Moreover, the matrix of $A^*$ with respect to the $g$-orthonormal basis $(\partial_s,\partial_t)$ is given by~\eqref{eq198}. Hence, 
\[
\sigma^*(\partial_s,\partial_s)
=
g(A^*\partial_s,\partial_s)=1,
\qquad
\sigma^*(\partial_s,\partial_t)
=
g(A^*\partial_s,\partial_t)
=1.
\]
Using the Gauss and Weingarten equations for $f^*$, we obtain
\begin{equation}\label{eq199}
\begin{aligned}
(\gamma^*)''
&=
\nabla^{\mathbb{R}^3}
_{df^*(\partial_s)}df^*(\partial_s)
=
df^*(\nabla_{\partial_s}\partial_s)
+
\sigma^*(\partial_s,\partial_s)N^*
=
N^*,
\\
(\eta^*)'
&=
\nabla^{\mathbb{R}^3}_{df^*(\partial_s)}df^*(\partial_t)=
df^*(\nabla_{\partial_s}\partial_t)
+
\sigma^*(\partial_s,\partial_t)N^*
=
N^*,
\\
(N^*)'
&=
\nabla^{\mathbb{R}^3}
_{df^*(\partial_s)}N^*
=
-df^*(A^*\partial_s)
=
-df^*(\partial_s+\partial_t)
=
-(\gamma^*)'-\eta^*.
\end{aligned}
\end{equation}
If we now use the Darboux equations for $\gamma^*$, we get
\begin{equation}\label{eq200}
\begin{aligned}
(\gamma^*)''
&=
\kappa_g\, \eta^*+\kappa^*_n\, N^*,
\\
(\eta^*)'
&=
-\kappa_g\, (\gamma^*)'+\tau^*_g\, N^*,
\\
(N^*)'
&=
-\kappa^*_n\, (\gamma^*)'-\tau^*_g\, \eta^*.
\end{aligned}
\end{equation}
From~\eqref{eq199} and~\eqref{eq200} we have
\[
\kappa_g=0,\qquad
\kappa_n^*=1,\qquad
\tau^*_g=1.
\]
which shows that $\gamma^*$ is a helix
with curvature and torsion both equal to $1$.

Now fix a Jacobi function $u=u(s,t)$ for $f$. Any choice of $(\omega_0,X_0^*)\in \mathbb{R}^3\times \mathbb{R}^3$ produces unique solutions $\omega,X^*\in C^{\infty}(\Sigma,\mathbb{R}^3)$ of~\eqref{eq:CMC-omega-PDE}, \eqref{eq:CMC-X-star-PDE}, respectively.
Expressing $\omega,X^*$ along $f^*(\gamma)$ with respect to $\{(\gamma^*)',\eta^*,N^*\}$ produces
\begin{equation}\label{omegaX*example}
\omega=x(\gamma^*)'+y\eta^*+zN^*,\qquad
X^*=p(\gamma^*)'+q\eta^*+wN^*,
\end{equation}
where $x,y,z,p,q,w$ are smooth functions of $s$.

Let $\partial_{\eta}u(s)=\partial_tu(s,0)$ and $\partial_{\eta}u^*(s)=\partial_tu^*(s,0)$; thus $w(s)=u^*(s,0)$. We next identify the coefficients in the last displayed equation.
Using $Y=\partial_s$ in~\eqref{eq:CMC-omega-form} and~\eqref{eq:CMC-omega-PDE} we get $d\omega (\partial _s)=u\, df^*((I-J)A\partial _s)-du(\partial_s+J\partial_s)\, N^*$.
Since for the immersion $f$ we have $A\partial_s=0$, then
$d\omega (\partial _s)=-du(\partial_s+\partial_t) N^*=
-(\partial_su+\partial_tu)N^*$. Note that this equality holds along the
whole $\Sigma$ (i.e., for any values of $(s,t)$). Evaluating along $\gamma=\{ t=0\}$, we have
\begin{equation}\label{eq202}
\omega'=-\bigl( (u\circ \gamma)'+\partial_{\eta}u\bigr) N^*.
\end{equation}
Taking derivatives in the first equation of~\eqref{omegaX*example} with respect to $s$ we deduce
\begin{eqnarray}
\omega'&=&
x'(\gamma^*)'+x(\gamma^*)''+y'\eta^*+y(\eta^*)'+z'N^*+z(N^*)'  \nonumber
\\
&\stackrel{\eqref{eq199}}{=}&
(x'-z)(\gamma^*)'+(y'-z)\eta^*+(x+y+z')N^*.  \label{eq203}
\end{eqnarray}
Comparison of coefficients in~\eqref{eq202} and \eqref{eq203} gives
\begin{equation}\label{eq204}
x'=z,\qquad y'=z,\qquad -(u\circ \gamma)'-\partial_{\eta}u=x+y+z'.
\end{equation}
Analogously, using $Y=\partial_s$ in~\eqref{eq:CMC-X-star-form} and~\eqref{eq:CMC-X-star-PDE} we get $dX^*(\partial _s)=-\omega \times df^*(\partial_s)-u\, df^*(A\partial_s)+du(J\partial_s)N^*$.
As $A\partial_s=0$, then for every $(s,t)$
\[
dX^*(\partial _s)=-\omega \times df^*(\partial _s)+\partial_tu\, N^*.
\]
Evaluating along $\gamma=\{ t=0\}$ and using that the frame
$\{(\gamma^*)',\eta^*,N^*\}$ is positively oriented, we get
\begin{equation}\label{eq205}
(X^*)'=-\omega\times (\gamma^*)'+\partial_{\eta}u\, N^*
=
-z\, \eta^*+(y+\partial_{\eta}u) N^*.
\end{equation}
Taking derivatives in the second equation of~\eqref{omegaX*example} with respect to $s$,
\begin{eqnarray}
(X^*)'&=&
p'(\gamma^*)'+p(\gamma^*)''+q'\eta^*+q(\eta^*)'+w'N^*+w(N^*)'  \nonumber
\\
&\stackrel{\eqref{eq199}}{=}&
(p'-w)(\gamma^*)'+(q'-w)\eta^*+(p+q+w')N^*.  \label{eq206}
\end{eqnarray}
Comparing~\eqref{eq205} and \eqref{eq206}, we have
\begin{equation}\label{eq207}
p'=w,\qquad -z=q'-w,\qquad y+\partial_{\eta}u=p+q+w'.
\end{equation}

Observe that the six functions $x,y,z,p,q,w$ are arbitrary, so every Euclidean Killing gauge is covered, and the ODE systems~\eqref{eq204}, \eqref{eq207} must be satisfied by any Euclidean Killing gauge.

Since $w(s)=u^*(s,0)$, we can write the last equation of~\eqref{eq207} as
$y+\partial_{\eta}u=p+q+(u^*\circ \gamma)'$. Differentiating this last equation with respect to $s$ we obtain
\[
y'+(\partial_{\eta}u)'=p'+q'+(u^*\circ \gamma)''
\stackrel{\eqref{eq207}}{=}
2(u^*\circ \gamma)-z+(u^*\circ \gamma)''.
\]
Recall that~\eqref{eq196} expresses $X^*,\omega$ in tangent and normal components to $f^*$. Comparing this with~\eqref{omegaX*example} we get
the following equalities for all $(s,t)$:
\begin{equation}\label{eq208}
W=x\,\partial_s+y\,\partial_t,
\qquad
Z=p\,\partial_s+q\,\partial_t.
\end{equation}
Taking $Y=\partial_t$ in~\eqref{eq197}, we obtain
\begin{eqnarray}
\partial_\eta u^*
&=&
du^*(\partial_t)
=
g(W,J\partial_t)+du(J\partial_t)
-g(A^*\partial_t,Z)  \nonumber
\\
&=&
-g(W,\partial_s)-du(\partial_s)
-g(A^*\partial_t,Z)  \nonumber
\\
&\stackrel{\eqref{eq198},\eqref{eq208}}{=}&
-x-(u\circ \gamma)'-(p+q).  \label{eq209}
\end{eqnarray}
Differentiating \eqref{eq209} along $\gamma$ gives
\begin{equation}\label{eq:CMC-cylinder-boundary-law}
(\partial_\eta u^*)'
=
-x'-(u\circ \gamma)''-p'-q'
\stackrel{\eqref{eq204},\eqref{eq207}}{=}
-(u\circ \gamma)''-2(u^*\circ \gamma),
\end{equation}
or equivalently, $(\partial_s\partial_tu^*)(s,0)+2u^*(s,0)=-(\partial_s^2u)(s,0)$ for all $s$.

So far we have used an arbitrary Jacobi function on our cylindrical example. We now make two particular choices to prove that both pure Dirichlet--Neumann and Neumann--Dirichlet exchange fail.
\begin{enumerate}
\item Consider the function $u(s,t):=e^s\sin(\sqrt5\,t)$, which is Jacobi since $\mathcal L_f=\partial_s^2+\partial_t^2+4$. Observe that $u\circ \gamma=0$ and $\partial_{\eta}u(s)=u_t(s,0)=\sqrt5\,e^s$ along $\gamma$. If a gauge gave $\partial_{\eta}u^*=0$ along $\gamma$, then the left-hand side of~\eqref{eq:CMC-cylinder-boundary-law} would vanish identically; since the
first term in the right-hand side also vanishes, we would deduce that 
$u^*\circ \gamma=0$, i.e. $w=0$. The third equation in~\eqref{eq207} then gives $y+\partial_{\eta}u=p+q$, while~\eqref{eq209} yields
$0=\partial_{\eta}u^*=-x-(u\circ \gamma)'-(p+q)=-x-(p+q)$, i.e., $p+q=-x$.
Therefore, $y+\partial_{\eta}u=-x$. Plugging this equality in the third equation of~\eqref{eq204} gives $z'=0$. On the other hand, taking derivatives in $y+\partial_{\eta}u=-x$ we get 
\[
z+(\partial_{\eta}u)'\stackrel{\eqref{eq204}}{=}
y'+(\partial_{\eta}u)'=-x'\stackrel{\eqref{eq204}}{=}-z.
\]
Taking one more derivative, $(\partial_{\eta}u)''=-2z'=0$. This is a contradiction, as $(\partial_{\eta}u)''=\sqrt5\,e^s$. Consequently, no choice of Euclidean Killing gauge can make $\partial_{\eta}u^*$ vanish identically along $\gamma$.

\item The function $u(s,t)=e^s\cos(\sqrt5\,t)$ is also Jacobi, with $\partial_{\eta}u(s)=u_t(s,0)=0$ and $(u\circ \gamma)(s)=e^s$. If a gauge gave $u^*\circ \gamma =0$, hence $w=0$ and also $w'=0$. Since
$\partial_{\eta}u=0$, the third equation in~\eqref{eq207} gives $y=p+q$.
Differentiating this identity with respect to $s$, we obtain $y'=p'+q'$.
Now, 
\[
z\stackrel{\eqref{eq204}}{=}y'=p'+q'\stackrel{\eqref{eq207}}{=}w+(w-z)=-z,
\]
hence $z=0$ along $\gamma$. In particular, $z'=0$. This equality and  $\partial_{\eta}u=0$ reduce the third equation in~\eqref{eq204} to $-(u\circ\gamma)'=x+y$,
and differentiating once more with respect to $s$ gives $x'+y'+(u\circ\gamma)''=0$.
The first two equations in~\eqref{eq204}, together with $z=0$, give $x'=y'=z=0$.
Consequently, $(u\circ\gamma)''=0$. This is impossible, since $(u\circ\gamma)(s)=e^s$ and thus, $(u\circ\gamma)''(s)=e^s$. This contradiction shows that no choice of Euclidean Killing gauge can make $u^*\circ\gamma$ vanish identically.
\end{enumerate}

Thus both pure Dirichlet--Neumann and Neumann--Dirichlet exchange fail in general for CMC~$1$ conjugation in $\mathbb{R}^3$ even modulo Euclidean Jacobi--Killing functions. For this cylindrical example, Dirichlet data for a Jacobi function $u$ do imply the modified condition $\partial_s\partial_tu^*+2u^*=0$ by~\eqref{eq:CMC-cylinder-boundary-law}, which is not an ordinary Robin condition.
}
\end{example}

\section*{Acknowledgements}

Jos\'e M. Espinar is partially supported by MINECO/MICINN/FEDER grant no PID2024-\allowbreak160586NB-I00 and by the ``Mar\'\i a de Maeztu'' Excellence Unit IMAG, reference CEX2020-001105-M, funded by \allowbreak MCINN/AEI/10.13039/501100011033/CEX2020-001105-M.
\par\noindent
Joaqu\'in P\'erez is partially supported by MINECO/MICINN/FEDER grant
PID2020-117868GB-I00, regional grant P18-FR-4049, and the ``Mar\'ia de
Maeztu'' Excellence Unit IMAG, reference CEX2020-001105-M, funded by
MCINN/AEI/10.13039/501100011033.

\vspace{.3cm}
\par\noindent
{\sc Jos\'e M. Espinar:} jespinar@ugr.es\\
{\sc Joaqu\'\i n P\'{e}rez:} jperez@ugr.es									
\begin{center}
\sc Department of Geometry and Topology and Institute of Mathematics (IMAG)\\
University of Granada, 18071, Granada, Spain
\end{center}

\end{document}